\documentclass[11pt]{article}
\usepackage{}

\usepackage[total={6in, 8.5in}]{geometry}
\usepackage{amsfonts}
\usepackage{amsthm}
\usepackage{amssymb}
\usepackage{amsmath}
\usepackage{enumerate}
\usepackage[
pdfauthor={ESYZ},
pdftitle={Toughness and spanning trees in K4mf graphs},
pdfstartview=XYZ,
bookmarks=true,
colorlinks=true,
linkcolor=blue,
urlcolor=blue,
citecolor=blue,
bookmarks=false,
linktocpage=true,
hyperindex=true
]{hyperref}

\makeatletter
\newcommand{\setword}[2]{%
	\phantomsection
	#1\def\@currentlabel{\unexpanded{#1}}\label{#2}%
}
\makeatother

\makeatletter
\renewcommand*\env@matrix[1][*\c@MaxMatrixCols c]{%
	\hskip -\arraycolsep
	\let\@ifnextchar\new@ifnextchar
	\array{#1}}
\makeatother

\usepackage{graphicx}
\usepackage[natural]{xcolor}
\usepackage{tikz}
\usepackage{tikz-3dplot}
\usetikzlibrary{shapes}
\usetikzlibrary{arrows,decorations.pathmorphing,backgrounds,positioning,fit,petri,automata}
\usetikzlibrary {positioning}
\usepackage{tkz-graph}
\usetikzlibrary{arrows}
\usepackage{tkz-euclide}
\usetkzobj{all}

\usepackage{pgf,tikz,pgfplots}

\usepackage{mathrsfs}

\usetikzlibrary{arrows}

\usepackage{xcolor}
\long\def\ignore#1{}

\let\oldi\ignore

 \oldi{
\newtheorem{THM}{\textbf{Theorem}}[section]
\newtheorem{THMs}{\textbf{Theorem}}[section]
\newtheorem{DEF}[THM]{\textbf{Definition}}[section]
\newtheorem{LEM}[THM]{\textbf{Lemma}}
\newtheorem{CON}[THM]{\textbf{Conjecture}}
\newtheorem{PROP}[THM]{\textbf{Proposition}}
\newtheorem{COR}[THM]{\textbf{Corollary}}
\newtheorem{CORs}{\textbf{Corollary}}[section]
\newtheorem{PRO}[THM]{\textbf{Problem}}
\newcommand{\pf}{\textbf{Proof}.\quad}

\newtheorem{FAC}{\textbf{Fact}}
\newtheorem{REM}{\textbf{Remark}}
\newtheorem{OPR}{\textbf{Operation}}
\newtheorem{CLA}{\textbf{Claim}}[section]
\setcounter{CLA}{1}
 }

\newtheorem{THM}{Theorem}[section]

\newtheorem{DEF}[THM]{Definition}
\newtheorem{LEM}[THM]{Lemma}
\newtheorem{CON}[THM]{Conjecture}

\newtheorem{COR}[THM]{Corollary}

\newtheorem{CLA}{Claim}[section]
\newcommand{\pf}{\textbf{Proof}.\quad}

\newtheorem*{THM2}{\textbf{Theorem 2.6}}
\newtheorem*{THM3}{\textbf{Theorem 2.7}}

\usepackage{thmtools}
\usepackage{thm-restate}

\newcommand{\CC}{\mathcal{C}}

\newcommand{\Of}{\mathcal{O}}

\newcommand{\pbar}{\overline{\varphi}}

\begin{document}
\title{Proof of the Core Conjecture of Hilton and Zhao}

\author{%
 Yan Cao
 \qquad Guantao Chen\thanks{This work was supported in part by NSF grant DMS-1855716.}\\
  Department of Mathematics and Statistics, \\
  Georgia State University, Atlanta, GA 30302\\
   \texttt{ycao17@gsu.edu}
  \qquad
   \texttt{gchen@gsu.edu}%
\and
 Guangming Jing\\
  Department of Mathematics,\\
  Augusta University, Augusta, GA 30912\\
  \texttt{gjing@augusta.edu}%
 \and 
 Songling Shan\thanks{This work was partially supported by the NSF-AWM Mentoring Travel Grant 1642548 and by the New Faculty Initiative Grant of Illinois State University.}\\
 Department of Mathematics, \\
 Illinois State Univeristy, Normal, IL 61790 \\
 \texttt{sshan12@ilstu.edu}
} 

\date{March 31, 2020}
\maketitle

 \begin{abstract}
 Let $G$ be a simple graph with maximum degree $\Delta$.  We call $G$ 
  \emph{overfull} if $|E(G)|>\Delta \lfloor |V(G)|/2\rfloor$. The \emph{core} of $G$, denoted 
 $G_{\Delta}$, is the subgraph of $G$ induced by its vertices of degree $\Delta$. 
 A classic result of Vizing shows that $\chi'(G)$, the chromatic index of $G$, is either $\Delta$
 or $\Delta+1$.   It is NP-complete to determine  the chromatic index  for a general graph.   However, if $G$ is overfull then  $\chi'(G)=\Delta+1$. 
 Hilton and Zhao in 1996 conjectured  that if $G$ is a simple connected graph with $\Delta\ge 3$ and  $\Delta(G_\Delta)\le 2$, then $\chi'(G)=\Delta+1$ if and only if $G$
 is overfull or $G=P^*$, where $P^*$ is obtained from the Petersen graph by deleting a vertex. 
 This conjecture, if true, implies an easy approach for calculating $\chi'(G)$  for graphs $G$
 satisfying the conditions. The progress on the conjecture has been slow: it was only confirmed for 
 $\Delta=3,4$, respectively, in 2003 and 2017.  
 In this paper, we confirm this conjecture for all $\Delta\ge 4$.   


 \end{abstract}

\vspace{2mm}
\newpage 
\tableofcontents
\newpage

\section{Introduction}

In this paper, we use graph to mean a simple graph, and 
use multigraph for graphs that may contain parallel edges 
but no loops.  Let $G$ be a graph with maximum degree $\Delta(G)=\Delta$.  We denote by $V(G)$ and $E(G)$ the vertex set and edge set of $G$, respectively. 
The \emph{core} of  $G$,  denoted $G_\Delta$, 
is the subgraph of $G$ induced by its vertices of degree $\Delta$. 
An {\it edge $k$-coloring\/} of $G$ is a mapping $\varphi$ from $E(G)$ to the set of integers
$[1,k]:=\{1,\cdots, k\}$, called {\it colors\/}, such that  no two adjacent edges receive the same color with respect to $\varphi$.  
The {\it chromatic index\/} of $G$, denoted $\chi'(G)$, is defined to be the smallest integer $k$ so that $G$ has an edge $k$-coloring.  
We denote by $\CC^k(G)$ the set of all edge $k$-colorings of $G$. 
In 1965, Vizing~\cite{Vizing-2-classes} showed that a graph of maximum degree
$\Delta$ has  chromatic index either $\Delta$ or $\Delta+1$.
If $\chi'(G)=\Delta$, then $G$ is said to be of {\it Class 1\/}; otherwise, it is said to be
of {\it Class 2\/}.  
Holyer~\cite{Holyer} showed that it is NP-complete to determine whether an arbitrary graph is of Class 1.  

For  a  multigraph  $G$ with $|V(G)|\ge 3$, define its density
\begin{equation}\label{density}
\omega(G)=\max\limits_{X\subseteq V(G), |X|\ge 3 } \frac{|E(G[X])|}{\lfloor|X|/2\rfloor}, 
\end{equation}
or zero by convention if $|V(G)|\le 2$. 
If $\omega(G)>\Delta(G)$, then  $\omega(G)$ is achieved by some $X\subseteq V(G)$  with an odd cardinality. 
Note that  $\omega(G)$ is a lower bound on $\chi'(G)$, 
since every matching of $G$ contains at most $\lfloor|X|/2\rfloor$ edges 
with both endpoints in $X$ for every $X\subseteq V(G)$.    We 
call $G$ {\it overfull} if  $|E(G)|>\Delta \lfloor |V(G)|/2\rfloor$. 
Thus, if $G$ is overfull, $\omega(G)\ge \frac{|E(G)|}{\lfloor |V(G)|/2\rfloor}>\Delta$.
Consequently,  $|V(G)|$ is odd and $\chi'(G)=\Delta+1$. 


%

Although it is NP-complete to compute the chromatic index of a graph $G$, 
as shown by Seymour~\cite{MR532981} using
Edmonds' matching polytope theorem~\cite{MR0183532},  $\max\{\Delta(G),\omega(G)\}$, which equals  $\chi'_f(G)$, 
 the {\it fractional chromatic index}  of $G$,  can be computed in  polynomial time.
This, naturally, leads to the question of  characterizing  graphs $G$ such that  $\chi'(G)=\lceil \chi'_f(G)\rceil$.  
The following conjectures  indicate that  there might be  a large class of graphs and multigraphs $G$
satisfying $\chi'(G)=\lceil \chi'_f(G)\rceil$.  

\begin{CON}[Goldberg-Seymour Conjecture~\cite{MR0354429},~\cite{MR532981}]\label{G-S}
Every  multigraph $G$ with $\chi'(G)\ge \Delta(G)+2$ satisfies    $\chi'(G)= \lceil\omega(G)\rceil$. 
\end{CON}

\begin{CON}[Seymour's Exact Conjecture~\cite{MR629483}]
Every   planar multigraph  $G$ satisfies $\chi'(G)=\lceil \chi'_f(G)\rceil$.  
\end{CON}

\begin{CON}[Hilton's Overfull Conjecture~\cite{MR848854}, \cite{MR975994}]
Every graph $G$ with $\Delta(G)>\frac{1}{3}|V(G)|$ satisfies $\chi'(G)=\lceil \chi'_f(G)\rceil$.  
\end{CON}

Goldberg-Seymour Conjecture was confirmed recently by Chen, Jing, and Zang~\cite{1901.10316}. 
Seymour's Exact Conjecture is equivalent to the Four-Color-Theorem when it is restricted to 3-regular planar multigraphs, and it  implies the Four-Color-Theorem when it is restricted to 4-regular planar multigraphs. 
Both Seymour's Exact Conjecture and Hilton's Overfull Conjecture are wide open. 

Classifying a graph as Class 1 or Class 2 is a very difficult problem in general even when restricted 
to the class of graphs with maximum degree three, see~\cite{Holyer}.   
Therefore, this problem is usually studied on particular classes of graphs.  One possibility is 
to consider graphs whose core has a simple structure (see~\cite[Sect.~4.2]{StiebSTF-Book}). 
Vizing~\cite{Vizing-2-classes} proved that if $G_\Delta$ has at most two vertices then $G$ 
is Class 1. Fournier~\cite{MR0349458} generalized Vizing's result by showing that 
if $G_\Delta$ contains no cycles then $G$ is Class 1. 
Thus a necessary condition for a graph to be Class 2 is to have a core
that contains cycles. Hilton and Zhao~\cite{MR1395947} considered the problem of classifying graphs
whose core is the disjoint union of cycles. Only a few such graphs are known to be Class 2.
These include the overfull graphs and the graph $P^*$, which is obtained from the Petersen
graph by removing one vertex.  In 1996, Hilton and Zhao~\cite{MR1395947} proposed the following conjecture, which  again relates $\chi'(G)$ to $\omega(G)$. 

\begin{CON}[Core Conjecture]\label{Core Conjecture}
	Let $G$ be a connected simple graph with $\Delta\ge 3$ and $\Delta(G_\Delta)\le 2$. 
	Then $G$ is Class 2 if and only if $G$ is overfull or $G=P^*$. 
\end{CON}

This conjecture has been one of the most fundamental unsolved problems in graph edge colorings. 
It attempts to classify graphs with $\Delta(G_\Delta)\le 2$.  This attempt   is 
an initial but  significant  move from the  result by Fournier~\cite{MR0349458}  that 
if $G_\Delta$ contains no cycles then $G$ is Class 1.  Secondly, if the Core Conjecture is true, 
it leads to an easy approach to determine the 
chromatic index for graphs $G$ with $\Delta(G_\Delta)\le 2$, by just counting the number of 
edges in $G$ if $G\ne P^*$.  

We call a connected Class 2 graph $G$ with $\Delta(G_\Delta)\leq 2$ a \emph{Hilton-Zhao graph (HZ-graph)}.  Clearly, $P^*$ is an HZ-graph  with $\chi'(P^*)=4$ and $\Delta(P^*)=3$. Hence the Core Conjecture is equivalent to the claim that every HZ-graph  $G\not=P^*$ with $\Delta(G)\geq 3$ is overfull. Not much progress has been made since the conjecture was proposed by Hilton and Zhao in 1996. A first breakthrough was achieved in 2003, when Cariolaro and Cariolaro \cite{CariolaroC2003} settled the
base case $\Delta=3$. They proved that $P^*$ is the only HZ-graph with maximum degree $\Delta=3$, an alternative proof was given later by Kr\'al', Sereny, and Stiebitz (see \cite[pp. 67--63]{StiebSTF-Book}). The next case $\Delta=4$ was recently solved by Cranston and Rabern \cite{CranstonR2018hilton}, they proved that the only HZ-graph with maximum degree $\Delta=4$ is $ K_5^-$ ($K_5$ with an edge deleted). 
The conjecture is wide open for $\Delta \ge 5$.  In this paper, we confirm 
the Core Conjecture for all HZ-graphs $G$ with $\Delta\ge 4$ 
as below.
\begin{THM}\label{Thm:main}
	Let $G$ be a connected  graph with $\Delta(G)=\Delta$. If $\Delta\ge 4$ and $\Delta(G_\Delta)\le 2$, then  $G$ is Class 2 if and only if $G$ is overfull. 
\end{THM}


Since overfull graphs are Class 2, it suffices to only show the ``only if'' part of the statement above. The proof of Theorem~\ref{Thm:main} develops certain edge  coloring techniques in dealing with the occurrence of  ``lollipop'' structures
in a $\Delta$-critical graphs,  where a lollipop structure can be seen as a combination of a multifan and a Kierstead path. 
The establishment of the coloring techniques  indicates that in general, in a $\Delta$-critical 
graph $G$, if $u$ is a vertex that is adjacent to many small degree vertices, then vertices that are of distance at most three 
to $u$ in $G$ are not adjacent to too many small degree vertices outside $N_G(u)$. 
In a sense, it says that   a $\Delta$-critical graph $G$ cannot have too many small degree vertices, or equivalently  $G$ is close to a $\Delta$-regular graph.
This brings hope  to  find a subgraph $H$ in $G$ so that the density of  $H$ is close to $\omega(G)$, which  
thereby shedding some light on  attacking density related conjectures such as Seymour's Exact Conjecture and Hilton's Overfull Conjecture. 
The confirmation of Seymour's Exact Conjecture applying edge coloring techniques 
will provide a computer-free proof for the Four-Color-Theorem. 

It  is  also worth mentioning that our proofs imply a polynomial-time algorithm
to edge color any graph $G$ with  $\Delta(G)\ge 4$ and $\Delta(G_\Delta)\le 2$
by using exactly $\chi'(G)$ colors. 




\section{Main Theorems} 

In this section, we prove Theorem~\ref{Thm:main} by assuming the truth of  Theorem~\ref{Thm:vizing-fan} to Theorem~\ref{Thm:nonadj_Delta_vertex}. 
 We start with some  concepts and auxiliary results. 

Let $G$ be a graph. 
For a vertex $v\in V(G)$, $N_G(v)$ is the set of neighbors of $v$ in $G$, and 
$d_G(v)=|N_G(v)|$ is the degree of  $v$ in $G$. 
The closed neighborhood of $v$ in $G$, denoted  $N_G[v]$,  is defined by $N_G(v)\cup \{v\}$. 
We simply write $N(v), N[v]$,  and 
$d(v)$ if $G$ is clear. 
For $e\in E(G)$, $G-e$
denotes the graph obtained from $G$ by deleting the edge $e$. 
We write $u\not\sim v$ if $u$ is nonadjacent to $v$ in $G$. 
The symbol $\Delta$  is reserved for $\Delta(G)$, the maximum degree of $G$
throughout  this paper.  A \emph{$k$-vertex}  in $G$  is a vertex of degree exactly $k$
in $G$, and a \emph{$k$-neighbor}  of a vertex $v$ is a neighbor of $v$ that is a $k$-vertex in $G$.
Let  $i\ge 1$ be an integer and $v\in V(G)$. Define 
$$ V_i=\{w\in V(G)\,:\, d_G(w)=i\} , \quad \quad N_{i}(v)=N_G(v)\cap V_i, \quad \,\mbox{and} \quad N_i[v]=N_i(v)\cup \{v\}. $$
For  $X\subseteq V(G)$, we define  $\mathit {N_{G}(X)=\bigcup_{x\in X}N_G(x)}$ 
 and $ \mathit {N_i(X)=N_G(X)\cap V_i}$.
 For $H\subseteq G$, we simple write $N_G(H)$ for $N_G(V(H))$. 

Similar to vertex coloring, it is essential to color the ``core'' part of a graph and then extend the coloring to the whole graph without increasing the total number of colors. This leads to the concept of {\it edge-chromatic criticality\/}. An edge $e\in E(G)$ is a \emph{critical edge} of $G$ if $\chi'(G-e)<\chi'(G)$. 
A graph $G$ is called {\it edge $\Delta$-critical} or simply {\it $\Delta$-critical\/}  if $G$ is connected,  $\chi(G)=\Delta+1$,  and every edge of $G$ is critical.    
Critical graphs are useful since they provide more information about the  structure around a vertex than general Class 2 graphs. For 
example,  Vizing's Adjacency Lemma (VAL) from 1965~\cite{Vizing-2-classes} is a useful tool that reveals certain structure at a vertex
by assuming the criticality of an edge.

\begin{LEM}[Vizing's Adjacency Lemma (VAL)]Let $G$ be a Class 2 graph with maximum degree $\Delta$. If $e=xy$ is a critical edge of $G$, then $x$ has at least $\Delta-d_G(y)+1$ $\Delta$-neighbors in $V(G)\setminus \{y\}$.
	\label{thm:val}
\end{LEM}
Let $G$ be a graph and 
$\varphi\in \CC^k(G-e)$ for some edge $e\in E(G)$ and some integer $k\ge 0$. 
For any $v\in V(G)$, the set of colors \emph{present} at $v$ is 
$\varphi(v)=\{\varphi(f)\,:\, \text{$f$ is incident to $v$}\}$, and the set of colors \emph{missing} at $v$ is $\pbar(v)=[1,k]\setminus\varphi(v)$.  
For a vertex set $X\subseteq V(G)$,  define 
$$
\pbar(X)=\bigcup _{v\in X} \pbar(v).
$$
The set $X$ is called \emph{elementary} with respect to $\varphi$  or simply \emph{$\varphi$-elementary} if $\pbar(u)\cap \pbar(v)=\emptyset$
for every two distinct vertices $u,v\in X$.   Sometimes, we just say that $X$ 
is elementary if the  edge coloring is understood.  

A graph $G$ with $\chi'(G)=\lceil\omega(G)\rceil$
is called an  \emph{elementary graph}. 
Note that for $e\in E(G)$ and  $\varphi \in \CC^\Delta (G-e)$,  $V(G)$ is $\varphi$-elementary implies that $G$ is elementary 
by taking $X=V(G)$ in Definition~\eqref{density}.  Overfull graphs are certainly elementary. 
All known HZ-graphs except $P^*$ are elementary. Hilton and Zhao in~\cite{MR1172373} proved that every HZ-graph also satisfies the following properties. 
\begin{LEM}\label{biregular}
If $G$ is  an HZ-graph with maximum degree $\Delta$, then  the following statements hold.
	\begin{enumerate}[(a)]
		\item $G$ is $\Delta$-critical and $G_\Delta$ is 2-regular. 
		\item $\delta(G)=\Delta-1$, or $\Delta=2$ and $G$ is an odd cycle. 
		\item Every vertex of $G$ has at least two neighbors in $G_\Delta$. 
	\end{enumerate}
\end{LEM}

Stiebitz et al.  in~\cite[Sect.~4.2]{StiebSTF-Book} defined a 
class of elementary graphs, and showed that  if an HZ-graph is elementary, then it belongs to the class. 
\begin{DEF}\label{class-O}
	Let $\Of_\Delta$  be a class of graphs constructed as follows: Let  $\Delta\ge4$,  $n_1$ and $n_2$ be integers with $3\le n_1\le \Delta-1$, $n_2=\Delta-2$, and  $n_1+n_2$ being odd. 
	Graphs in $\Of_\Delta$ are obtained from a complete bipartite graph $K_{n_1,n_2}$ 
	by inserting on the set of $n_1$ independent vertices a 2-regular simple graph with $n_1$ vertices
	and on the set of $n_2$ independent vertices a $(\Delta-1-n_1)$-regular simple  graph with $n_2$ vertices. 
\end{DEF}

By counting edges, it is  straightforward  to verify that   graphs in $\Of_\Delta$
are overfull, and so are elementary. 
Stiebitz et al.  in~\cite[Sect.~4.2]{StiebSTF-Book} proved that for an HZ-graph $G$ with maximum degree $\Delta$, if  it is  elementary,    then either $\Delta=2$ and $G$ is an odd cycle, or $\Delta\ge 4$ and $G\in \Of_\Delta$. As 
a consequence of this result and the observation that overfull graphs and $P^*$
are Class 2, Conjecture~\ref{Core Conjecture} is equivalent to the following conjecture. 

\begin{CON}\label{core-conjecture-2}
If $G$ is an HZ-graph with maximum degree $\Delta$,  then 
either $G\in \Of_\Delta$, or 
$\Delta=2$ and $G$ is an odd cycle, or $\Delta=3$ and $G=P^*$. 
\end{CON}

We will in fact prove this equivalent form of the Core Conjecture for $\Delta\ge 4$  based on  the following results. 



\begin{THM}\label{Thm:vizing-fan}
	If $G$ is an HZ-graph with maximum degree $\Delta\ge 4$, then the following two statements hold.
	\begin{enumerate}[(i)]
		\item For every two adjacent vertices  $u,v\in V_{\Delta}$, $N_{\Delta-1}(u)=N_{\Delta-1}(v)$.  \label{common}
		\item For any $r\in V_\Delta$, there exist a vertex $s\in N_{\Delta-1}(r)$
		and  a coloring 
		$\varphi\in \CC^\Delta(G-rs)$ such that $N_{\Delta-1}[r]$ 
		is $\varphi$-elementary.  
		\label{ele}
	\end{enumerate}
\end{THM}

For an HZ-graph $G$ with $\Delta(G)\ge 4$,  each component of $G_\Delta$ is a cycle by Lemma~\ref{biregular}. So Theorem~\ref{Thm:vizing-fan}~(\ref{common}) implies that $N_{\Delta-1}(x)=N_{\Delta-1}(y)$ for every two vertices $x,y\in V_\Delta$ that are  from a same cycle of $G_\Delta$.

%

\begin{THM}
	\label{Thm:adj_small_vertex}
If $G$ is an HZ-graph with maximum degree $\Delta\ge 4$,  then for every two adjacent vertices $x, y\in V_{\Delta-1}$, $N_\Delta(x)=N_\Delta(y)$.
\end{THM}


\begin{THM}
	\label{Thm:nonadj_Delta_vertex}
	Let $G$ be an HZ-graph with maximum degree $\Delta\ge 7$ and 
	$u, v\in V_{\Delta}$ be two non-adjacent vertices.
	If $N_{\Delta-1}(u)\ne N_{\Delta-1}(v)$ and $N_{\Delta-1}(u)\cap  N_{\Delta-1}(v)\ne \emptyset$,
then $|N_{\Delta-1}(u)\cap  N_{\Delta-1}(v)|=\Delta-3$, i.e. $|N_{\Delta-1}(u)\setminus N_{\Delta-1}(v)|=|N_{\Delta-1}(v)\setminus  N_{\Delta-1}(u)|=1$.
\end{THM}	


\begin{COR}
	\label{Thm:adj_small_vertex2}
If $G$ is an HZ-graph with maximum degree $\Delta\ge 7$  and there exist $u, v\in V_{\Delta}$ such that  $N_{\Delta-1}(u)\ne N_{\Delta-1}(v)$, 
		then $V_{\Delta-1}$ is an independent set in $G$.
\end{COR}
\pf
Assume to the contrary that there exist $x,y\in V_{\Delta-1}$ such that $xy\in E(G)$. By Lemma~\ref{biregular}, $N_\Delta(x)\not=\emptyset$. Let $w\in N_\Delta(x)$. By the assumption that there exist $u,v\in V_\Delta$ such that $N_{\Delta-1}(u)\ne N_{\Delta-1}(v)$,  there exists some $w''\in V_\Delta$ such that $N_{\Delta-1}(w)\ne N_{\Delta-1}(w'')$. By Theorem~\ref{Thm:vizing-fan} $(i)$, such $w''$ is not on the same cycle containing $w$ in $G_\Delta$. Since $G$ is connected, along a path in $G$ joining $w$ and $w''$, there exists some 
$w'\in V_{\Delta}$ such that either 
$N_{\Delta-1}(w)\ne N_{\Delta-1}(w')$ and $N_{\Delta-1}(w)\cap N_{\Delta-1}(w')\ne\emptyset$ or there exists an edge between $N_{\Delta-1}(w)$ and $N_{\Delta-1}(w')$. 
For the latter case, applying Theorem~\ref{Thm:adj_small_vertex}, we again 
see that $N_{\Delta-1}(w)\cap N_{\Delta-1}(w')\ne\emptyset$. 
Therefore, there exists $w'\in V_{\Delta}$ such that 
$N_{\Delta-1}(w)\ne N_{\Delta-1}(w')$ and $N_{\Delta-1}(w)\cap N_{\Delta-1}(w')\ne\emptyset$. 
 The choice of $w'$ implies 
\begin{equation}\label{2}
  |N_{\Delta-1}(w)\cap N_{\Delta-1}(w')|=\Delta-3  
\end{equation}
by Theorem~\ref{Thm:nonadj_Delta_vertex}. Thus by~\eqref{2} and Theorem~\ref{Thm:adj_small_vertex}
\begin{equation}\label{3}
  x,y\in N_{\Delta-1}(w)\cap N_{\Delta-1}(w').
\end{equation}

Let	$N_{\Delta-1}(w')\setminus N_{\Delta-1}(w)=\{z\}$. We claim that $N_{\Delta-1}(z)=\emptyset$. For otherwise, let $z'\in N_{\Delta-1}(z)$. Clearly $z'\not=z$. By Theorem~\ref{Thm:adj_small_vertex}, $z'\in N_{\Delta-1}(w')\setminus N_{\Delta-1}(w)$, giving a contradiction to $N_{\Delta-1}(w')\setminus N_{\Delta-1}(w)=\{z\}$. We then claim that $N_\Delta(z)\subseteq N_\Delta(x)$. For otherwise let $w^*\in N_\Delta(z)\setminus N_\Delta(x)$. As $x\in N_{\Delta-1}(w')$, it follows that $w^*\ne w'$. Since $z\in N_{\Delta-1}(w^*)\cap N_{\Delta-1}(w')$, it follows that $|N_{\Delta-1}(w^*)\cap N_{\Delta-1}(w')|\ge \Delta-3$ by Theorem~\ref{Thm:nonadj_Delta_vertex} (it can happen that $N_{\Delta-1}(w^*)=N_{\Delta-1}(w')$). Thus $N_{\Delta-1}(w^*)\cap N_{\Delta-1}(w')$ contains at least one of $x,y$ by~\eqref{3}, and so $x\in N_{\Delta-1}(w^*)$ by Theorem~\ref{Thm:adj_small_vertex}, contradicting the choice of $w^*$. Therefore we have $N_{\Delta-1}(z)=\emptyset$ and $N_\Delta(z)\subseteq N_\Delta(x)$.  However,  $d_G(z)\le |N_\Delta(x)|<|N_\Delta(x)\cup \{y\}|\le d_G(x)$, contradicting $d_G(x)=d_G(z)=\Delta-1$. This completes the proof. \qed
	
We now prove Conjecture~\ref{core-conjecture-2} for $\Delta\ge 4$ as below. 

\begin{THM}\label{mainthm2}
If $G$ is an HZ-graph with maximum degree $\Delta\ge 4$, then 
 $G\in \Of_\Delta$. 
\end{THM}
\begin{proof}

 Assume to the contrary that there exists an HZ-graph $G$ with maximum degree $\Delta\ge 4$ such that $G\not\in \Of_\Delta$. Let $n=|V(G)|$.  First assume that $N_{\Delta-1}(u)=N_{\Delta-1}(v)$ for every pair $u,v\in V_\Delta$.   Then $V_\Delta$, $V_{\Delta-1}$ and the edges between them form a complete bipartite graph. Since  $G\not\in \Of_\Delta$, it follows that $n$ is even. Let $r\in V_\Delta$. The above assumption also implies that  $N_{\Delta-1}(r)=V_{\Delta-1}$. By Theorem~\ref{Thm:vizing-fan} \eqref{ele}, there exist $s\in N_{\Delta-1}(r)=V_{\Delta-1}$ and 
$\varphi\in \CC^\Delta(G-rs)$ such that   
$N_{\Delta-1}[r]=V_{\Delta-1}\cup \{r\}$ is $\varphi$-elementary, which thereby implies that $V(G)$ is $\varphi$-elementary. Therefore, each color in $\pbar(N_{\Delta-1}[r])$ is missed  at exactly one vertex in $V(G)$, showing that $n$ is odd. 
This gives 
 a contradiction.

We now assume that there exist $u, v\in V_{\Delta}$ such that  $N_{\Delta-1}(u)\ne N_{\Delta-1}(v)$. For each $w\in V_\Delta$, let $C_w$ be the cycle in  $G_\Delta$ that contains $w$. 
By Theorem~\ref{Thm:vizing-fan} \eqref{common},   $C_u$ and $C_v$
are disjoint. Since $G_\Delta$ is 2-regular, $uv\notin E(G)$.
As $G$ is connected, there is a path in $G$ joining $u$ and $v$.
Furthermore, by Theorem~\ref{Thm:vizing-fan} (\ref{common}) and Theorem~\ref{Thm:adj_small_vertex}, there exists a path in $G$ joining $u$ and $v$ with alternating vertices from $V_\Delta$ and $V_{\Delta-1}$. Thus we may choose $u,v\in V_\Delta$ with  
$N_{\Delta-1}(u)\ne N_{\Delta-1}(v)$ such that there exists a path $uwv$ with $w\in V_{\Delta-1}$. Then $d_G(w)\ge |V(C_u)|+|V(C_v)|\ge 6$
by Theorem~\ref{Thm:vizing-fan} \eqref{common}.  
Thus $\Delta=d_G(w)+1\ge 7$.  Applying Corollary~\ref{Thm:adj_small_vertex2}, it follows that $V_{\Delta-1}$ is an independent set of $G$.
 
Let $A\subseteq V_\Delta$ be the set of all vertices $a$ satisfying $N_{\Delta-1}(a)=N_{\Delta-1}(u)$, and let $B\subseteq V_\Delta$ be the set of all vertices $b$ satisfying $N_{\Delta-1}(b)\ne N_{\Delta-1}(u)$ and $N_{\Delta-1}(b)\cap N_{\Delta-1}(u)\ne \emptyset$. Clearly $u\in A$ and $v\in B$, so $A$ and $B$ are non-empty. Partition $B$ into non-empty subsets $B_1,B_2,\ldots,B_t$ such that for each $i\in[1,t]$, all vertices in $B_i$ have the same neighborhood in $V_{\Delta-1}$. By Theorem~\ref{Thm:vizing-fan}~\eqref{common}, each of $A,B_1,B_2,\ldots,B_t$ induces a union of disjoint cycles in $G_\Delta$. So $|A|\ge 3$ and $|B_i|\ge 3$ for each $i\in[1,t]$.
 
Now we claim   $t\ge \Delta-2$. 
 Assume otherwise  $t\le \Delta-3$.  
Since each $i\in [1,t]$, $|N_{\Delta-1}(A)\setminus N_{\Delta-1}(B_i)|=1$ by Theorem~\ref{Thm:nonadj_Delta_vertex} and  $|N_{\Delta-1}(A)|=\Delta-2$, by the Pigeonhole Principle, it follows that $N_{\Delta-1}(A)\cap \left(\bigcap_{i=1}^t N_{\Delta-1}(B_i)\right)\not=\emptyset$.   Let $z\in N_{\Delta-1}(A)\cap \left(\bigcap_{i=1}^t N_{\Delta-1}(B_i)\right)$ and $z'\in N_{\Delta-1}(A)\setminus N_{\Delta-1}(B_1)$. Then 
  \begin{eqnarray*}
 	 |A|+\sum\limits_{1\le i\le t}|B_i|= d_G(z)=d_G(z')\le |A|+ \sum\limits_{2\le i\le t}|B_i|, 
 \end{eqnarray*}
 achieving a contradiction. Hence $t\ge \Delta-2$. 
 
 We now achieve a contradiction to the assumption  $\Delta\ge 7$
 by counting the number of edges in $G$ between $N_{\Delta-1}(A)$ and $A\cup B$.
 Note that $|N_{\Delta-1}(A)|=\Delta-2$. 
 Since each vertex in $B$ has exactly $\Delta-3$ neighbors in $N_{\Delta-1}(A)$ and $|B_i| \ge 3$ for each $i\in [1,t]$, we have
 \[ |E_G (A\cup B,N_{\Delta-1}(A))| = 
 |A|(\Delta-2) + 
 |\cup_{i=1}^tB_i | (\Delta-3) \ge 3(\Delta-2)+3t(\Delta -3) \ge 3(\Delta -2)^2. 
 \]
 
On the other hand, since $N_{\Delta-1}(A)$ is an independent set and every vertex in it has degree $\Delta -1$ in $G$, we  have 
 \[
  |E_G (A\cap B,N_{\Delta-1}(A))| = (\Delta-1) (\Delta-2). 
 \]
 Solving $(\Delta-1)(\Delta -2) \ge 3(\Delta -2)^2$ gives  $2\le \Delta\le 2.5$, 
achieving a  desired contradiction. 
\end{proof}

%

\section{Definitions and Preliminary Lemmas}\label{lemma}
In this section, we present a few known results and some new results.
Those will be the foundations for showing one of the main theorems. 
 Let $G$ be a graph, 
$\varphi \in \CC^k(G-e)$ for some $e\in E(G)$ and some integer $k\ge 0$. 
We start with some definitions and notation.



  For two distinct colors $\alpha,\beta \in [1,k]$, let  $H$ be the subgraph of $G$
with $V(H)=V(G)$ and $E(H)$ consisting of edges from $E(G)$ that are colored by $\alpha$
or $\beta$ with respect to $\varphi$. Each component of $H$ is either 
an even cycle or a path, which is called an \emph{$(\alpha,\beta)$-chain} of $G$
with respect to $\varphi$.  If we interchange the colors $\alpha$ and $\beta$
on an $(\alpha,\beta)$-chain $C$ of $G$, we get a new edge $k$-coloring  of $G$, 
and we write $$\varphi'=\varphi/C.$$
This operation is called a \emph{Kempe change}. 
For a color $\alpha$, a sequence of 
{\it Kempe  $(\alpha,*)$-changes}  is a sequence of  
Kempe changes that each involves the exchanging of the color $\alpha$
and another color from $[1,k]$.

Let  $x,y\in V(G)$, and  $\alpha, \beta, \gamma\in [1,k]$ be three colors.   If $x$ and $y$
are contained in a same  $(\alpha,\beta)$-chain of $G$ with respect to $\varphi$, we say $x$ 
and $y$ are \emph{$(\alpha,\beta)$-linked} with respect to $\varphi$.
Otherwise, $x$ and $y$ are \emph{$(\alpha,\beta)$-unlinked} with respect to $\varphi$. Without specifying $\varphi$, when we just say  $x$ and $y$ are $(\alpha,\beta)$-linked or $x$ and $y$ are $(\alpha,\beta)$-unlinked, we mean they are linked or unlinked with respect to the current edge coloring. 
Let $P$ be an 
$(\alpha,\beta)$-chain of $G$ with respect to $\varphi$ that contains both $x$ and $y$. 
If $P$ is a path, denote by $\mathit{P_{[x,y]}(\alpha,\beta, \varphi)}$  the subchain  of $P$ that has endvertices $x$
and $y$.  By \emph{swapping  colors} along  $P_{[x,y]}(\alpha,\beta,\varphi)$, we mean 
exchanging the two colors $\alpha$
and $\beta$ on the path $P_{[x,y]}(\alpha,\beta,\varphi)$. 
The notion $P_{[x,y]}(\alpha,\beta)$ always represents the $(\alpha,\beta)$-chain
with respect to the current edge coloring. 
Define  $P_x(\alpha,\beta,\varphi)$ to be an $(\alpha,\beta)$-chain or an $(\alpha,\beta)$-subchain of $G$ with respect to $\varphi$ that starts at $x$  and ends at a different vertex missing exactly one of $\alpha$ and $\beta$.    
(If $x$ is an endvertex of the $(\alpha,\beta)$-chain that contains $x$, then $P_x(\alpha,\beta,\varphi)$ is unique.  Otherwise, we take one segment of the whole chain to be 
$P_x(\alpha,\beta,\varphi)$. We will specify the segment when it is used.) 
If  $u$  is a vertrex on  $P_x(\alpha,\beta,\varphi)$, we  write  {$\mathit {u\in P_x(\alpha,\beta, \varphi)}$}; and if $uv$  is an edge on  $P_x(\alpha,\beta,\varphi)$, we  write  {$\mathit {uv\in P_x(\alpha,\beta, \varphi)}$}.  Similarly, the notion $P_x(\alpha,\beta)$ always represents the $(\alpha,\beta)$-chain
with respect to the current edge coloring. 
If $u,v\in P_x(\alpha,\beta)$ such that $u$ lies between $x$ and $v$, 
then we say that $P_x(\alpha,\beta)$ \emph{meets $u$ before $v$}. 
Suppose that   $\alpha\in \pbar(x)$ and  $\beta,\gamma\in \varphi(x)$. An $\mathit{(\alpha,\beta)-(\beta,\gamma)}$
\emph{swap at $x$}  consists of two operations:  first swaps colors on $P_x(\alpha,\beta, \varphi)$ to get an edge  $k$-coloring $\varphi'$, and then swaps
colors on $P_x(\beta,\gamma, \varphi')$. 
By convention, an	$(\alpha,\alpha)$-swap at $x$ does  nothing at $x$. 
Suppose the current color of an  edge $uv$ of $G$
is $\alpha$, the notation  $\mathit{uv: \alpha\rightarrow \beta}$  means to recolor  the edge  $uv$ using the color $\beta$. 
Recall that $\pbar(x)$ is the set of colors not present 
at $x$. 
If $|\pbar(x)|=1$, we will also use $\pbar(x)$ to denote the  color that is missing at $x$. 
When recoloring a graph, we say the current coloring is {\it conflicting at $x$ with respect to a color}
if there are at least two edges incident to $x$ that are colored  by the specified color.


Let   
 $T$ be  a sequence of vertices  and edges of  $G$. We denote by \emph{$V(T)$}  
 the set of vertices from $V(G)$ that are contained in $T$, and by  
\emph{$E(T)$}  the set of edges  from $E(G)$ that are contained in $T$.
If $V(T)$ is $\varphi$-elementary,
then for a color  $\tau\in \pbar(V(T))$,  we denote by  $\mathit{\pbar^{-1}_T(\tau)}$ the  unique vertex  in $V(T)$ at which $\tau$ 
is missed.  For a coloring $\varphi'\in \CC^\Delta(G-e)$, 
$\varphi'$ is called {\it $T$-stable} with respect to $\varphi$ if for every  $x\in V(T)$ it holds that $\pbar'(x)=\pbar(x)$, 
and for every  $f\in E(T)$ it holds that $\varphi'(f)=\varphi(f)$.  Clearly, $\varphi$ is 
$T$-stable with respect to itself. 

Let $\alpha, \beta, \gamma, \tau\in[1,k]$. 
We will use a  matrix with two rows to denote a sequence of operations  taken  on $\varphi$.
Each entry in the first row represents a path or  a sequence of vertices. 
Each entry in the second row, indicates the action taken on the object above this entry. 
We require the operations to be taken to follow the ``left to right'' order as they appear in 
the matrix. 
For example,   the matrix below indicates 
three operations taken on the graph based 
on the coloring from the previous step:
\[
\begin{bmatrix}
P_{[a, b]}(\alpha, \beta)  & s_c:s_{d} & rs\\
\alpha/\beta & \text{shift} & \gamma \rightarrow \tau 
\end{bmatrix}.
\]
\begin{enumerate}[Step 1]
	\item Swap colors on the $(\alpha,\beta)$-subchain $P_{[a, b]}(\alpha, \beta,\varphi) $.
	\item Based on the coloring obtained from  Step 1, shift from $s_c$ to $s_d$
	for vertices $s_c, s_{c+1}, \ldots, s_d$. (Shifting will be defined shortly.)
	
	\item Based on the coloring obtained from  Step 2,  do  $rs: \gamma \rightarrow \tau $. 
\end{enumerate}


%
 Let $w\in V(G)$ and $p\ge 1$.  A \emph{star} centered at $w$ with 
 $p$ leaves 
 is a subgraph of $G$ that is isomorphic to the complete bipartite graph $K_{1,p}$ such that $w$ 
 has degree $p$ in the subgraph. 
 If  $v_1, \ldots, v_p \in N_G(w)$  are the leaves,  we denote the star by 
 $S(w; v_1,\ldots, v_p)$. 
 
 Let $a,b$ be two positive integers. 
 If $b\ge a$, we  abbreviate  a vertex 
 sequence $s_a, s_{a+1}, \ldots, s_{b}$ as $s_a:s_b$. 
 If $b<a$, then $s_a:s_b$ denotes  an empty sequence. 
 The notation $[a,b]$ stands for the set $\{a,  \ldots, b\}$ 
 if $b\ge a$, and $\emptyset$ otherwise. 
 If $F=(a_1,\ldots, a_t )$ is a sequence, then for a new 
 entry $b$, 
 $(F, b)$  denotes the sequence $(a_1,\ldots, a_t, b)$.

\subsection{Multifan}

Let  $G$ be a graph, $e=rs_1\in E(G)$ and $\varphi\in \CC^k(G-e)$ for some integer $k\ge 0$.
A \emph{multifan} centered at $r$ with respect to $e$ and $\varphi$
is a sequence $F_\varphi(r,s_1:s_p):=(r, rs_1, s_1, rs_2, s_2, \ldots, rs_p, s_p)$ with $p\geq 1$ consisting of  distinct vertices $r, s_1,s_2, \ldots , s_p$ and distinct edges $rs_1, rs_2,\ldots, rs_p$ satisfying the following condition:
\begin{enumerate}[(F1)]
	\item For every edge $rs_i$ with $i\in [2,p]$,  there is a vertex $s_j$ with $j\in [1,i-1]$ such that 
	$\varphi(rs_i)\in \pbar(s_j)$. 
\end{enumerate}
We will simply denote a multifan  $F_\varphi(r,s_1: s_{p})$ by $F$ if 
$\varphi$ and the vertices and edges in $F_\varphi(r,s_1: s_{p})$  are clear. 
Let $F_\varphi(r,s_1: s_{p})$ be a multifan. 
By its definition, for any $p^*\in [1,p]$,  $F_\varphi(r,s_1: s_{p^*})$
is a multifan. 
The following result regarding a multifan can be found in \cite[Theorem~2.1]{StiebSTF-Book}.

\begin{LEM}
	\label{thm:vizing-fan1}
	Let $G$ be a Class 2 graph and $F_\varphi(r,s_1:s_p)$  be a multifan with respect to a critical edge $e=rs_1$ and a coloring $\varphi\in \CC^\Delta(G-e)$. Then  the following statements  hold. 
	 \begin{enumerate}[(a)]
	 	\item $V(F)$ is $\varphi$-elementary. \label{thm:vizing-fan1a}
	 	\item Let $\alpha\in \pbar(r)$. Then for every $i\in [1,p]$  and $\beta\in \pbar(s_i)$,  $r$ 
	 	and $s_i$ are $(\alpha,\beta)$-linked with respect to $\varphi$. \label{thm:vizing-fan1b}
	 \end{enumerate}
\end{LEM}

Let $F_\varphi(r,s_1:s_p)$  be a multifan.  We call $s_{\ell_1},s_{\ell_2}, \ldots, s_{\ell_k}$, a subsequence of $s_1:s_p$, an  \emph{$\alpha$}-sequence with respect to $\varphi$ and $F$ if the following holds:
$$
\varphi(rs_{\ell_1})= \alpha\in \pbar(s_1),  \quad \varphi(rs_{\ell_i})\in \pbar(s_{\ell_{i-1}}), \quad  i\in [2,k].
$$
A vertex in an $\alpha$-sequence is called an \emph{$\alpha$-inducing vertex} with respect to $\varphi$ and $F$, and a missing color at an $\alpha$-inducing vertex is called an \emph{$\alpha$-inducing color}. For convenience, $\alpha$ itself is also an $\alpha$-inducing color. We say $\beta$ is {\it induced by} $\alpha$ if $\beta$ is $\alpha$-inducing. By Lemma~\ref{thm:vizing-fan1} (a) and the definition of multifan, each color in $\pbar(V(F))$ is induced by a unique color in $\pbar(s_1)$. Also if $\alpha_1,\alpha_2$ are two distinct colors in $\pbar(s_1)$, then an $\alpha_1$-sequence is disjoint with an $\alpha_2$-sequence. For two distinct $\alpha$-inducing colors $\beta$ and $\delta$, we write {$\mathit \delta \prec \beta$} if there exists an $\alpha$-sequence $s_{\ell_1},s_{\ell_2}, \ldots, s_{\ell_k}$ such that $\delta\in\pbar(s_{\ell_i})$, $\beta\in\pbar(s_{\ell_j})$ and $i<j$. For convenience, $\alpha\prec\beta$ for any $\alpha$-inducing color $\beta\not=\alpha$. Then all $\alpha$- inducing colors form a poset with order $\prec$. An $\alpha$-inducing color is called a \emph{last $\alpha$-inducing color} if it is a maximal element in the poset.

As a consequence of Lemma~\ref{thm:vizing-fan1} (a), we have the following properties for a multifan. 
 \begin{LEM}
 	\label{thm:vizing-fan2}
Let $G$ be a Class 2 graph and $F_\varphi(r,s_1:s_p)$  be a multifan with respect to a critical edge $e=rs_1$ and a coloring $\varphi\in \CC^\Delta(G-e)$. For two colors $\delta\in \pbar(s_i)$ and $\lambda\in \pbar(s_j)$ with  $i,j\in [1,p]$ and $i\ne j$, the following statements  hold.
 	\begin{enumerate}[(a)]
 		\item If $\delta$ and $\lambda$ are induced by different colors, then $s_i$ and $s_j$ are $(\delta, \lambda)$-linked with respect to $\varphi$. 
 		\label{thm:vizing-fan2-a}
 		\item If $\delta$ and $\lambda$ are induced by the same color, $\delta\prec\lambda$, and $s_i$ and $s_j$ are $(\delta, \lambda)$-unlinked with respect to $\varphi$, 
 		then $r\in P_{s_j}(\lambda, \delta, \varphi)$.  	\label{thm:vizing-fan2-b}
 	\end{enumerate}
 	
 \end{LEM}
\begin{proof}
	
For (a),  suppose otherwise that $s_i$ and $s_j$ are $(\delta,\lambda)$-unlinked with respect to $\varphi$. Assume that $\delta$ and $\lambda$ are induced by $\alpha$ and $\beta$ respectively where $\alpha,\beta$ are two distinct colors from $\pbar(s_1)$. Let $s_{i_1},s_{i_2},\ldots,s_{i_k}=s_i$ be the $\alpha$-sequence containing $s_i$, and $s_{j_1},s_{j_2},\ldots,s_{j_\ell}=s_j$ be the $\beta$-sequence containing $s_j$. Since $V(F)$ is $\varphi$-elementary, $s_i$ is the only vertex in $F$ 
that misses $\delta$. Therefore, the other end  
of $P_{s_j}(\delta,\lambda,\varphi)$ is outside of $V(F)$.  Let $\varphi'=\varphi/ P_{s_j}(\delta,\lambda,\varphi)$. 
It is clear that  $\varphi'\in \CC^\Delta(G-e)$, and $F^*=(r,rs_{i_1},s_{i_2},\ldots,s_{i_k},rs_{j_1},s_{j_1},\ldots,s_{j_\ell})$
is a multifan under $\varphi'$. However, $\delta\in \pbar'(s_i)\cap \pbar'(s_j)$, contradicting Theorem~\ref{thm:vizing-fan1} (a).  

For (b),  suppose otherwise that $r\not\in P_{s_j}(\lambda, \delta,\varphi)$. 
Assume, without loss of generality, that  $i<j$, and  $  s_2, \ldots, s_i, s_{i+1}, \ldots, s_j$ is an $\alpha$-sequence where $\alpha\in \pbar(s_1)$.
Since $V(F)$ is $\varphi$-elementary, $s_i$ is the only vertex in $F$ 
that misses $\delta$. Therefore,  when $s_i$ and $s_j$ are $(\delta,\lambda)$-unlinked with respect to $\varphi$, the other end  
of $P_{s_j}(\delta,\lambda,\varphi)$ is outside of $V(F)$.   
Let $\varphi'=\varphi/ P_{s_j}(\delta,\lambda,\varphi)$.  It is clear that  $\varphi'\in \CC^\Delta(G-e)$. 
Since $r\not\in P_{s_j}(\lambda, \delta,\varphi)$, 
$\varphi'$ agrees with $\varphi$ on $F$ at every edge and every vertex except $s_j$. 
Therefore,  the sequence $F_{\varphi'}(r,s_1: s_j)$, 
obtained from $F_\varphi(r,s_1:s_p)$ by deleting every edge and every vertex after $s_j$
is still a multifan. However, $\delta\in \pbar'(s_i)\cap \pbar'(s_j)$, contradicting Theorem~\ref{thm:vizing-fan1} (a).  
\end{proof}

\subsection{Multifan in HZ-graphs}\label{vizingfanhz}
As $\Delta$-degree vertices in a multifan do not miss any color, 
for  multifans in an HZ-graph, we add a further requirement in its definition as follows: all vertices of the multifan except the center vertex have degree $\Delta-1$.  In the remainder of this paper, we use this new definition for all multifans in HZ-graphs. 

Let $G$ be an HZ-graph with $\Delta\ge 3$, $rs_1\in E(G)$ with $r\in V_\Delta$ and $s_1\in V_{\Delta-1}$,
and $\varphi\in \CC^\Delta(G-rs_1)$. Let $F_\varphi(r,s_1:s_p)$ be a multifan. By its definition,  except $s_1$,  every other $s_i$ misses exactly one color with respect to $\varphi$ in $F$. Note that $|\pbar(s_1)|=2$, and so every color in $\pbar(V(F))$ is induced by one of the two colors in $\pbar(s_1)$. So $s_1,s_2,\ldots,s_p$ can be divided into two sequences. Therefore, we can equip  $F$ with additional  properties.  

Let $F$ be a multifan in  $G$ with respect to $rs_1$ and $\varphi\in \CC^\Delta(G-rs_1)$. 
We call $F$ a \emph{typical multifan}, denoted $F_\varphi(r, s_1:s_\alpha:s_\beta)$,  if $\pbar(r)=\{1\}$, $\pbar(s_1)=\{2,\Delta\}$ and the following hold:
\begin{enumerate}[(1)]
	\item Either $|V(F)|=2$, or $|V(F)|\ge 3$ and 
	there exist  $\alpha\in[2,\beta]$  such that 
	$s_2, \ldots, s_\alpha$  is a $2$-inducing sequence 
	and $s_{\alpha+1}, \ldots, s_{\beta}$  is a $\Delta$-inducing sequence of $F$, where $\beta=|V(F)|-1$. 	  
	\item If $|V(F)|\ge 3$, then for each $i\in [2,\beta]$, $\varphi(rs_i)=i$ and $\pbar(s_i)=i+1$ except when $i=\alpha+1\in [3,\beta]$.  In this case,	$\varphi(rs_{\alpha+1})=\Delta$ and $\pbar(s_{\alpha+1})=\alpha+2$.
\end{enumerate}
 
Clearly by relabelling   vertices and colors if necessary, any multifan in an HZ-graph can be assumed to be a typical multifan. If $\alpha\ge 2$
	and $\beta>\alpha$, we  say $F$  
	has  {\it two sequences}. Otherwise we say $F$ has {\it one sequence}. For a typical multifan $F=F_\varphi(r, s_1:s_\alpha:s_\beta)$, if $\alpha=\beta$, then we write $F=F_\varphi(r,s_1:s_\alpha)$ and say that $F$ is a {\it typical 2-inducing multifan}. The graph depicted  in Figure~\ref{f11} 
 shows a typical multifan within the neighborhood of a $\Delta$-vertex $r$ in an HZ-graph.

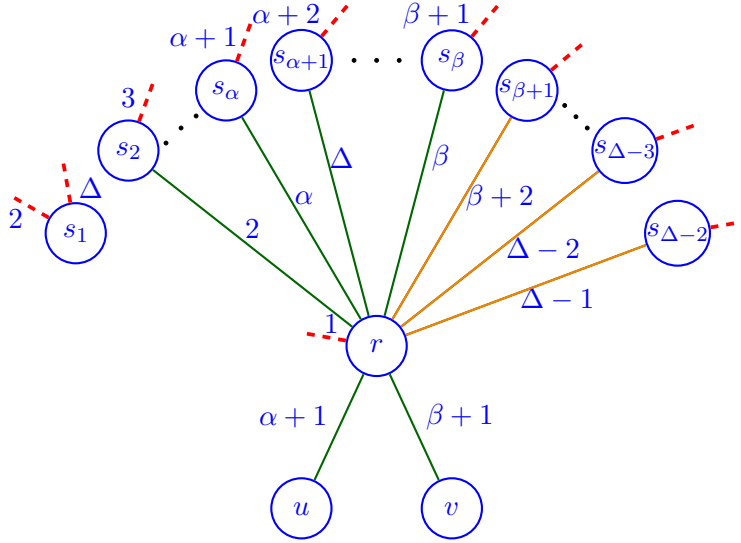
\begin{figure}[!htb]
\begin{center}
	\begin{tikzpicture}[scale=1]
	
	{\tikzstyle{every node}=[draw ,circle,fill=white, minimum size=0.8cm,
		inner sep=0pt]
		\draw[blue,thick](0,-3) node (r)  {$r$};
	\draw [blue,thick](-4, -1.5) node (s1)  {$s_1$};
	\draw [blue,thick](-3.3, -0.4) node (s2)  {$s_2$};
	\draw[blue,thick] (-2, 0.4) node (sa)  {$s_\alpha$};
	\draw [blue,thick](-1, 0.8) node (sa2)  {$s_{\alpha+1}$};
		\draw [blue,thick](1, 0.8) node (sb)  {$s_{\beta}$};
		\draw [blue,thick](2, 0.4) node (sb2)  {$s_{\beta+1}$};
		\draw[blue,thick] (3.3, -0.4) node (sd1)  {$s_{\Delta-3}$};
		\draw [blue,thick](4, -1.5) node (sd2)  {$s_{\Delta-2}$};
	\draw [blue,thick](-1, -5.16) node (u)  {$u$};
		\draw [blue,thick](1, -5.16) node (v)  {$v$};
		}
	\path[draw,thick,black!60!green]
	(r) edge node[name=la,above,pos=0.5] {\color{blue}$2$} (s2)
		(r) edge node[name=la,pos=0.6] {\color{blue}\quad$\alpha$} (sa)
			(r) edge node[name=la,pos=0.7] {\color{blue}\quad$\Delta$} (sa2)
				(r) edge node[name=la,pos=0.7] {\color{blue}\quad$\beta$} (sb)
					(r) edge node[name=la,pos=0.6] {\color{blue}\qquad\,\,\,$\beta+2$} (sb2)
						(r) edge node[name=la,pos=0.5] {\color{blue}\quad \qquad$\Delta-2$} (sd1)
							(r) edge node[name=la,pos=0.4] {\color{blue}\qquad\quad\,\,\,\,\,$\Delta-1$} (sd2)
							(r) edge node[name=la,pos=0.4] {\color{blue}$\alpha+1$\qquad\quad\,\,\,} (u)
							(r) edge node[name=la,pos=0.4] {\color{blue}\qquad\quad\,\,\,$\beta+1$} (v);

		\draw[orange, thick] (r) --(sb2); 
	\draw [orange, thick](r) --(sd1); 
	\draw [orange, thick](r) --(sd2); 
	
	\draw[dashed, red, line width=0.5mm] (r)--++(170:1cm); 
	\draw[dashed, red, line width=0.5mm] (s1)--++(150:1cm); 
	\draw[dashed, red, line width=0.5mm] (s1)--++(100:1cm); 
	\draw[dashed, red, line width=0.5mm] (s2)--++(70:1cm); 
	\draw[dashed, red, line width=0.5mm] (sa)--++(70:1cm); 
	\draw[dashed, red, line width=0.5mm] (sa2)--++(50:1cm); 
	\draw[dashed, red, line width=0.5mm] (sb)--++(50:1cm); 
	\draw[dashed, red, line width=0.5mm] (sb2)--++(40:1cm); 
	\draw[dashed, red, line width=0.5mm] (sd1)--++(20:1cm); 
	\draw[dashed, red, line width=0.5mm] (sd2)--++(10:1cm); 

	\draw[blue] (-0.6, -2.7) node {$1$};  
	\draw[blue] (-4.8, -1.3) node {$2$};  
	\draw[blue] (-3.8, -0.9) node {$\Delta$};  
	\draw[blue] (-3.3, 0.3) node {$3$};  
		\draw[blue] (-2.3, 1.1) node {$\alpha+1$};  
		\draw[blue] (-1.2, 1.4) node {$\alpha+2$};  
			\draw[blue] (0.8, 1.4) node {$\beta+1$}; 
			
			{\tikzstyle{every node}=[draw ,circle,fill=black, minimum size=0.05cm,
inner sep=0pt]
\draw(-2.4,0.06) node (f1)  {};
\draw(-2.6,-0.1) node (f1)  {};
\draw(-2.8,-0.3) node (f1)  {};
\draw(-0.3,0.8) node (f1)  {};
\draw(0,0.8) node (f1)  {};
\draw(0.3,0.8) node (f1)  {};
\draw(2.5,0.2) node (f1)  {};
\draw(2.65,0.05) node (f1)  {};
\draw(2.8
,-0.1) node (f1)  {};
} 
	\end{tikzpicture}
\end{center}
	\caption{A typical multifan $F_\varphi(r, s_1:s_\alpha:s_\beta)$ in the neighborhood of $r$, where a dashed line at a vertex indicates a color missing at the vertex.}
\label{f11}
\end{figure}


The following Lemma indicates  that  in an HZ-graph, any multifan can be assumed to be a typical multifan with only  one sequence. 

\begin{LEM}\label{2-inducing }
	Let $G$ be an HZ-graph with maximum degree $\Delta\ge 3$, 
	$rs_1\in E(G)$ with $r\in V_\Delta$ and $s_1\in V_{\Delta-1}$,
	and $\varphi\in \CC^\Delta(G-rs_1)$. 
	Then  for every  multifan $F=F_\varphi(r,s_1: s_p)$ of $G$,  there exists a coloring $\varphi'\in \CC^\Delta(G-rs_p)$ and a typical multifan $F^*$ centered at $r$ with respect to $rs_p$ and $\varphi'$ such that $V(F^*)=V(F)$ and $F^*$ has one sequence.  
\end{LEM}
\pf 
By the definition of multifan, $s_p$ is the last $\eta$-inducing color for some $\eta\in\pbar(s_1)$. Thus we may assume, without loss of generality, that $F=F_\varphi(r,s_1:s_\alpha:s_\beta)$ is a typical multifan and $s_p=s_\beta$. 
Clearly if $F$ has only one sequence then we are done. Thus we assume that  $\beta\ge \alpha+1 \ge 3$. Let $\varphi'$ be obtained from $\varphi$ by uncoloring $rs_\beta$, doing  $rs_i: i \rightarrow i+1$ for each $i\in [\alpha+1, \beta-1]$ and coloring $rs_1$ by $\Delta$. Now $\pbar'(s_\beta)=\{\beta,\beta+1\}$, $F^*=(r, rs_\beta, s_\beta, rs_{\beta-1}, s_{\beta-1}, \ldots,  rs_{\alpha+1}, s_{\alpha+1}, rs_1, s_1, rs_2, s_2, \ldots, rs_\alpha, s_\alpha)$ is a $\beta$-inducing multifan with respect to $rs_\beta$ and $\varphi'$.
By permuting the name of the colors
  and the label of the vertices, we obtain the desired  multifan.
\qed

\subsection{Kierstead path}

Let $G$ be a graph, $e=v_0v_1\in E(G)$, and  $\varphi\in \CC^k(G-e)$ for some integer $k\ge 0$.
A \emph{Kierstead path}  with respect to $e$ and $\varphi$
is a sequence $K=(v_0, v_0v_1, v_1, v_1v_2, v_2, \ldots, v_{p-1}, v_{p-1}v_p,  v_p)$ with $p\geq 1$ consisting of  distinct vertices $v_0,v_1, \ldots , v_p$ and distinct edges $v_0v_1, v_1v_2,\ldots, v_{p-1}v_p$ satisfying the following condition:
\begin{enumerate}[(K1)]
	\item For every edge $v_{i-1}v_i$ with $i\in [2,p]$,  there is a vertex $v_j$ with $j\in [1,i-1]$ such that 
	$\varphi(v_{i-1}v_i)\in \pbar(v_j)$. 
\end{enumerate}

Clearly a Kierstead path with at most 3 vertices is a multifan. We consider Kierstead paths with $4$ vertices. In the following lemma, statement $(a)$ was proved in Theorem 3.3 from~\cite{StiebSTF-Book} and, analogous to Lemma~\ref{thm:vizing-fan2}, statement $(b)$ is a consequence of $(a)$.

\begin{LEM}[]\label{Lemma:kierstead path1}
	Let $G$ be a Class 2 graph,
	 $e=v_0v_1\in E(G)$ be a critical edge, and $\varphi\in \CC^\Delta(G-e)$. If $K=(v_0, v_0v_1, v_1, v_1v_2,  v_2, v_2v_3, v_3)$ is a Kierstead path with respect to $e$
		and $\varphi$, then the following statements hold.
	\begin{enumerate}[(a)]
 		\item If $\max\{d_G(v_2), d_G(v_3)\}<\Delta$, then $V(K)$ is $\varphi$-elementary.
 		\item For any two colors $\alpha\in\pbar(v_0)$ and 
$\delta\in \pbar(v_3)$, if  $\max\{d_G(v_2), d_G(v_3)\}<\Delta$ and $\alpha \not\in \{\varphi(v_1v_2), \varphi(v_2v_3)\}$, 
then 
 $v_3$ and $v_0$ are $(\delta, \alpha)$-linked with respect to $\varphi$.
 	\end{enumerate}

\end{LEM}

\subsection{Pseudo-multifan}

In this subsection, we introduce a concept called ``pseudo-multifan'' and study some properties 
of it.  
Let $G$ be a graph,  $e=rs_1\in E(G)$,  and  $\varphi\in \CC^k(G-e)$ for some integer $k\ge 0$. A multifan $F_\varphi(r,s_1:s_t)$ is called \emph{maximum} at $r$ if $|V(F)|$ is maximum among all multifans with respect to $rs$  and $\varphi'\in \CC^k(G-rs)$ for some $s\in N_{\Delta-1}(r)$.

A  \emph{pseudo-multifan} with respect to $rs_1$ and $\varphi\in \CC^k(G-rs_1)$ is a sequence $$S:=S_\varphi(r,s_1:s_t:s_p):=(r, rs_1, s_1, rs_2, s_2, \ldots,rs_t, s_t, rs_{t+1},  s_{t+1}, \ldots, s_{p-1},  rs_p, s_p)$$ 
with $p\geq 2$  and  $t\ge 1$ consisting of distinct  vertices $r,s_1, \ldots , s_p$ and distinct edges $rs_1, rs_2, \ldots, rs_p$ satisfying the following conditions:
\begin{enumerate}[(P1)]
	\item $(r, rs_1, s_1, rs_2, s_2, \ldots,rs_t, s_t)$, denoted by $F_\varphi(r,s_1:s_t)$, is a multifan, and it is maximum at $r$.
		\item $V(S)$ is $\varphi$-elementary. Moreover, for every $F$-stable $\varphi'\in \CC^k(G-e)$ with respect to $\varphi$, $V(S)$ is $\varphi'$-elementary.
\end{enumerate}
Clearly every maximum multifan is a pseudo-multifan, and if $S$ is a pseudo-multifan 
with respect to  $\varphi$ and a multifan  $F$,  then by the definition above,  $S$ is a pseudo-multifan under every  $F$-stable coloring $\varphi'$.
 Colors in $\pbar(\{s_{t+1}, \ldots, s_p\})$
are called {\it pseudo-missing colors} of $S$. Let $G$ be an HZ-graph, $e=rs_1\in E(G)$ with $r\in V_\Delta$ and $s_1\in V_{\Delta-1}$, and $\varphi\in \CC^\Delta(G-e)$.  
We call a pseudo-multifan  $S$ {\it typical} (resp. {\it typical 2-inducing}) if the multifan that is contained in $S$
is typical (resp. typical 2-inducing). 

A sequence of distinct vertices 
$s_{h_1}, s_{h_2}, \ldots, s_{h_t} \in N_{\Delta-1}(r)$  form a \emph{rotation}  if 
\begin{enumerate}[(1)]
	\item $\{s_{h_1}, s_{h_t},  \ldots, s_{h_t}\}$ is $\varphi$-elementary, and 
	\item for each $\ell$ with $\ell\in [1,t]$, it holds $\varphi(rs_{h_\ell})=\pbar(s_{h_{\ell-1}})$ where $h_0=h_t$. 
\end{enumerate}
An example of a rotation is given in Figure~\ref{rt}. 

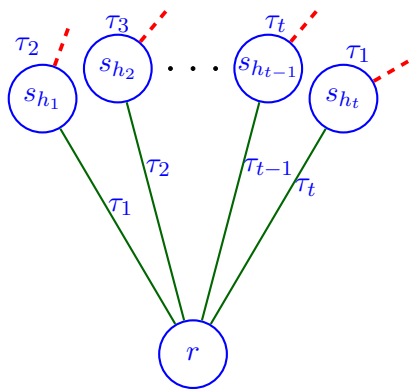
\begin{figure}[!htb]
\begin{center}
	\begin{tikzpicture}[scale=1]
	
	{\tikzstyle{every node}=[draw ,circle,fill=white, minimum size=0.9cm,
		inner sep=0pt]
		\draw[blue,thick] (0, -3) node (r)  {$r$};
	\draw[blue,thick] (-2, 0.4) node (sa)  {$s_{h_1}$};
	\draw [blue,thick](-1, 0.8) node (sa2)  {$s_{h_2}$};
		\draw [blue,thick](1, 0.8) node (sb)  {$s_{h_{t-1}}$};
			\draw [blue,thick](2, 0.4) node (sb2)  {$s_{h_t}$};
		}
	\path[draw,thick,black!60!green]
		(r) edge node[name=la,pos=0.6] {\color{blue}\quad$\tau_1$} (sa)
			(r) edge node[name=la,pos=0.7] {\color{blue}\quad$\tau_2$} (sa2)
			(r) edge node[name=la,pos=0.7] {\color{blue}\quad\,\,\,\,\,$\tau_{t-1}$} (sb)
				(r) edge node[name=la,pos=0.7] {\color{blue}\quad$\tau_t$} (sb2);

	\draw[dashed, red, line width=0.5mm] (sa)--++(70:1cm); 
	\draw[dashed, red, line width=0.5mm] (sa2)--++(50:1cm); 
	\draw[dashed, red, line width=0.5mm] (sb)--++(50:1cm); 
	\draw[dashed, red, line width=0.5mm] (sb2)--++(30:1cm);

		\draw[blue] (-2.2, 1.1) node {$\tau_2$};  
		\draw[blue] (-1.0, 1.4) node {$\tau_3$};  
			\draw[blue] (1.1, 1.4) node {$\tau_{t}$}; 
			\draw[blue] (2.2, 1.0) node {$\tau_{1}$}; 
			
			{\tikzstyle{every node}=[draw ,circle,fill=black, minimum size=0.05cm,
inner sep=0pt]

\draw(-0.3,0.8) node (f1)  {};
\draw(0,0.8) node (f1)  {};
\draw(0.3,0.8) node (f1)  {};

} 
	\end{tikzpicture}
\end{center}
	\caption{A rotation in the neighborhood of $r$.}
	\label{rt}
\end{figure}

Assume $N_{\Delta-1}(r)=\{s_1, s_2,\ldots, s_{\Delta-2}\}$. Let $i,j$ be integers with $2\le i\le j\le \Delta-2$. Then the \emph{shifting from $s_i$ to $s_j$}  is an operation that,  for each $\ell $ with $ \ell\in[i,j]$,  replaces the current color of  $rs_\ell$ by  the color in $\pbar(s_\ell)$.  We will apply shifting either on a sequence of vertices from a multifan 
or on a rotation. Note that we sometimes have $i>j$ when applying a shifting, in which  case the shifting does not change any color.
%
\begin{LEM}\label{pseudo-fan-ele:e}
Let $G$ be an HZ-graph with maximum degree $\Delta\ge 3$, 
	$r\in V_{\Delta}$
	with $N_{\Delta-1}(r)=\{s_1, s_2,\ldots, s_{\Delta-2}\}$,  and $\varphi\in \CC^\Delta(G-rs_1)$. Suppose that there exists a typical pseudo-multifan $S:=S_\varphi(r, s_1: s_t: s_{\Delta-2})$, and the maximum multifan $F:=F_\varphi(r,s_1:s_\alpha:s_t)$ contained in $S$ is typical.  Then for each  $i\in [2,\alpha]$, if $\varphi'$ is the coloring obtained from $\varphi$ by uncoloring $rs_{i}$, shifting from $s_{2}$ to $s_{i-1}$ and coloring $rs_1$ by $2$, then $$S^*=(r,rs_{i},s_{i},rs_{i-1},s_{i-1},\ldots,s_{2},rs_1,s_1,rs_{i+1},s_{i+1},\ldots,s_t,\ldots,s_{\Delta-2})$$ is a pseudo-multifan with respect to $\varphi'$.
\end{LEM}
\pf
By the definition of shifting, we know that $$F^*=(r,rs_{i},s_{i},rs_{i-1},s_{i-1},\ldots,s_{2},rs_1,s_1,rs_{i+1},s_{i+1},\ldots,s_t)$$ is a multifan. Since $|V(F^*)|=|V(F)|$, $F^*$ is also a maximum multifan at $r$. So to show Lemma~\ref{pseudo-fan-ele:e}, it suffices to show that for any $F^*$-stable $\varphi''\in \CC^\Delta(G-rs_i)$ with respect to $\varphi'$, $V(S)$ is $\varphi''$-elementary. 
Suppose to  the contrary that there exists $F^*$-stable $\varphi''\in \CC^\Delta(G-rs_i)$ with respect to $\varphi'$ but $V(S)$
is not $\varphi''$-elementary.  As $\varphi''$ is $F^*$-stable with respect to $\varphi'$,
we can undo the operations we did before. More specifically, let $\varphi'''$ be the coloring obtained from $\varphi''$ by uncoloring $rs_1$, shifting from $s_2$ to $s_{i-1}$ and coloring $s_{i}$ by $i$. Then $\varphi'''$ is $F$-stable with respect to $\varphi$ and $\pbar'''(V(S))=\pbar''(V(S))$.
Thus, $V(S)$ is  not $\varphi''$-elementary implies that   $V(S)$ is not $\varphi'''$-elementary. This contradicts the assumption that $V(S)$
is elementary under any $F$-stable coloring with respect to $\varphi$. \qed

\begin{LEM}\label{pseudo-fan-ele}
	Let $G$ be an HZ-graph with maximum degree $\Delta\ge 3$, 
	$r\in V_{\Delta}$
	with $N_{\Delta-1}(r)=\{s_1, s_2,\ldots, s_{\Delta-2}\}$,  and $\varphi\in \CC^\Delta(G-rs_1)$. If there exists a pseudo-multifan $S:=S_\varphi(r, s_1: s_t: s_{\Delta-2})$  with 
	 $\delta \in \pbar(s_j)$ for some $j\in [t+1,\Delta-2]$ and with $F:=F_\varphi(r,s_1:s_t)$ being  the maximum multifan contained in $S$, then the following statements hold.
	\begin{enumerate}[(a)]
		\item $\{s_{t+1}, \ldots, s_{\Delta-2}\}$ can be partitioned 
		into rotations with respect to $\varphi$. \label{pseudo-a}
		\item $s_j$ and $r$ are $(\delta,1)$-linked with respect to  $\varphi$. \label{pseudo-a1}
		\item For every color $\gamma\in \pbar(V(F)\setminus \{r\})$,  it holds 
		$r\in P_{y}(\delta,\gamma)=P_{s_j}(\delta,\gamma)$, where $y=\mathit{\pbar_{F}^{-1}(\gamma)}$.  
		Furthermore, for $z\in N_G(r)$ such that $\varphi(rz)=\gamma$,  
		 $P_{y}(\delta,\gamma)$
		meets $z$ before  $r$.  \label{pseudo-b}
		\item For every $\delta^*\in \pbar(V(S)\setminus V(F))$ with $\delta^*\ne \delta$, it holds $P_{y}(\delta,\delta^*)=P_{s_{j}}(\delta,\delta^*)$,  where $y=\mathit{\pbar_{S}^{-1}(\delta^*)}$. Furthermore, either $r\in P_{s_j}(\delta,\delta^*)$ or $P_r(\delta, \delta^*)$ is an even cycle. 
			 \label{pseudo-c}
	\end{enumerate}
\end{LEM}
\pf By relabeling colors and vertices, we assume $F$ is typical. Let  $F=F_\varphi(r,s_1:s_\alpha:s_\beta)$ be a typical multifan, where  $\beta=t$.

For   statement \eqref{pseudo-a}, by the definition of multifan, we have $\varphi(\{rs_2,rs_3,\ldots,rs_\beta\})=\pbar(V(F)) \setminus \{1,\alpha+1,\beta+1\}$. Since $V(S)$ is $\varphi$-elementary, $\cup_{i=\beta+1}^{\Delta-2}\pbar(s_i)=[1,\Delta] \setminus \pbar(V(F))$. Note that $\{\varphi(rs_i)\,:\, i\in [\beta+1, \Delta-2]\}=[1,\Delta]\setminus \big (\varphi(\{rs_2,rs_3,\ldots,rs_\beta\})\cup\{1,\alpha+1,\beta+1\}\big)$.  Hence
$$\cup_{i=\beta+1}^{\Delta-2}\pbar(s_i)=\{\varphi(rs_i)\,:\, i\in [\beta+1, \Delta-2]\}.$$
Thus, the sequence of missing colors  $\pbar(s_{\beta+1}), \ldots, \pbar(s_{\Delta-2})$
is a  permutation  of the sequence of colors $\varphi(rs_{\beta+1}), \ldots, \varphi(rs_{\Delta-2})$.    
Since every permutation can be partitioned into disjoint cycles, 
 $\{{s_{\beta+1}, \ldots, s_{\Delta-2}}\}$  has a partition into rotations. 
This finishes the proof for \eqref{pseudo-a}. 

Notice that $j\in [t+1,\Delta-2]$,  and statement \eqref{pseudo-a} implies that there is a rotation containing $s_j$. Assume without loss of generality that this rotation is $s_j, s_{j+1}, \ldots, s_\ell$
in the remainder of this proof. 

For  \eqref{pseudo-a1}, if $s_j$ and $r$
 are $(\delta,1)$-unlinked with respect to $\varphi$, then $P_{s_j}(\delta,1)$ ends at a vertex outside $V(F)$ and does not contain any edge in $F$. Thus $\varphi'=\varphi/P_{s_j}(\delta,1)$ is $F$-stable with respect to $\varphi$. But $V(S)$ is not $\varphi'$-elementary, giving a contradiction to (P2) in the definition of a pseudo-multifan.  

For the first part of statement\eqref{pseudo-b}, suppose to the contrary that
there exists $\gamma\in \pbar(s_i)$ for some $i\in[1, \beta]$ such that $r\in P_{s_i}(\delta,\gamma)=P_{s_j}(\delta,\gamma)$ does not hold. Assume without loss of generality that $i\in [1,\alpha]$. Then we have the following three cases: $r\notin P_{s_i}(\delta,\gamma)$ and $r\notin P_{s_j}(\delta,\gamma)$; $r\notin P_{s_i}(\delta,\gamma)$ and $r\in P_{s_j}(\delta,\gamma)$; and $r\in P_{s_i}(\delta,\gamma)$ and $r\notin P_{s_j}(\delta,\gamma)$.

Suppose that $r\notin P_{s_i}(\delta,\gamma)$ and $r\notin P_{s_j}(\delta,\gamma)$. Then let $\varphi'=\varphi/Q$,  where $Q$ is the $(\delta,\gamma)$-chain containing $r$. Note that $\varphi'$ and $\varphi$ agree on every edge incident to $r$ except two edges $rs_{j+1}$ and $rz$ where $z$ is defined in statement~\eqref{pseudo-b}. Since $r\notin P_{s_i}(\delta,\gamma)$, $r\notin P_{s_j}(\delta,\gamma)$ and $N_{\Delta-1}(r)$ is $\varphi$-elementary, $\pbar'(s_i)=\pbar(s_i)$ for all $s_i\in N_{\Delta-1}(r)$. Thus under the new coloring $\varphi'$, $F^*=(r, rs_1, s_1, \ldots,s_i, rs_{j+1}, s_{j+1}, \ldots, rs_\ell, s_\ell, rs_j, s_j, rs_{i+1},  s_{i+1}, \ldots,s_\beta)$ is a multifan (recall that $\{s_j,s_{j+1},\ldots,s_l\}$ is the rotation containing $s_j$) because $\pbar'(s_i)=\gamma=\varphi'(rs_{j+1})$,   $\pbar'(s_j)=\delta=\varphi'(rs_{i+1})$ if $i<\alpha$, and   $\varphi'(s_{i+1})=\Delta=\pbar'(s_{1})$ if $i=\alpha$.  As $|V(F)|<| V(F^*)|$, 
we obtain a contradiction to the maximality of $F$.

Suppose that $r\notin P_{s_i}(\delta,\gamma)$ and $r\in P_{s_j}(\delta,\gamma)$. Then let $\varphi'=\varphi/P_{s_j}(\delta,\gamma)$. Similar to the case above, one can easily check that $F^*=(r, rs_1, s_1, \ldots,s_i, rs_{j+1}, s_{j+1}, \ldots, rs_\ell, s_\ell, rs_j, s_j)$ is a multifan. 
Since $\pbar'(s_i)=\pbar'(s_j)=\gamma$, we obtain a contradiction to 
Lemma~\ref{thm:vizing-fan1} \eqref{thm:vizing-fan1a} that 
 $V(F^*)$ is $\varphi'$-elementary. 

Suppose that $r\in P_{s_i}(\delta,\gamma)$ and $r\notin P_{s_j}(\delta,\gamma)$. Then let $\varphi'=\varphi/P_{s_j}(\delta,\gamma)$. Note that $\varphi'$ is $F$-stable with respect to $\varphi$, thus by the definition of a pseudo-multifan, $V(S)$ is $\varphi'$-elementary. But $\pbar'(s_i)=\pbar'(s_j)=\gamma$, giving a contradiction. This completes the proof of the first part of statement \eqref{pseudo-b}.

For the second  part of statement\eqref{pseudo-b}, assume to the contrary that $P_{y}(\delta,\gamma)$
meets $r$ before  $z$. Then $P_{y}(\delta,\gamma)$ meets $s_{j+1}$ before $r$. Let $\varphi'$ be obtained from $\varphi$ by shifting from $s_{j}$
to $s_{\ell}$. Then  $r\not\in P_{y}(\delta,\gamma, \varphi')$, showing a contradiction to 
the first part of~\eqref{pseudo-b}. 

For the first part of statement\eqref{pseudo-c}, assume to the contrary that there exists $\delta^*=\pbar(s_{j^*})$ for some $j^*\ne j$ and $j^*\in [t+1, \Delta-2]$ such that $P_{s_j}(\delta,\delta^*)\ne P_{s_{j^*}}(\delta,\delta^*)$. Then let $\varphi'=\varphi/P_{s_j}(\delta,\delta^*)$.
Note that $\varphi'$ is $F$-stable with respect to $\varphi$, but $V(S)$ is not $\varphi'$-elementary, 
showing a contradiction to the  definition of a pseudo-multifan. 
For the second part of \eqref{pseudo-c}, 
assume that $r\not\in P_{s_j}(\delta,\delta^*)$
and the $(\delta,\delta^*)$-chain  containing $r$
is a path $Q$. By~\eqref{pseudo-a}, we let $s_{\ell_1}, \ldots, s_{\ell_k}$ be a rotation with $\varphi(rs_{\ell_1})=\pbar(s_{\ell_t})=\delta^*$ (note $s_{\ell_k}=s_{j^*}$). Note that the path $Q$ contains $rs_{j+1}$ and $rs_{j^*}$ since $\varphi(rs_{j+1})=\delta$ and $\varphi(s_{j^*})=\delta^*$. So $Q-rs_{j+1}-rs_{j^*}$ consists of two disjoint paths, say $Q_j$ and $Q_{j^*}$, which contain $s_{j+1}$ and $s_{j^*}$ respectively. Let $\varphi'$ be obtained from $\varphi$ by shifting from $s_{j}$ to $s_{l}$
and from $s_{\ell_1}$ to $s_{\ell_k}$ (only shift once if they are the same sequence up to permutation). Then $P_{s_{j+1}}(\delta,\delta^*,\varphi')=Q_j$. Let $\varphi^*=\varphi'/P_{s_{j+1}}(\delta, \delta^*,\varphi')$. We see that $\varphi^*$ is $F$-stable with respect to $\varphi$, but $V(S)$
is not $\varphi^*$-elementary, giving a contradiction.  \qed

\subsection{Lollipop}


Let $G$ be an $HZ$ graph,   $e=rs_1\in E(G)$ with $r\in V_\Delta$ and 
$s_1\in V_{\Delta-1}$, and let $\varphi\in \CC^\Delta(G-e)$. Then a \emph{lollipop}
centered at $r$ (shown in Figure~\ref{lol}) is a sequence 
$L=(F, ru, u,ux, x)$ of distinct vertices and edges such that $F=F_\varphi(r,s_1:s_\alpha:s_\beta)$ is a typical multifan, $u\in N_\Delta(r)$ and $x\in N_{\Delta-1}(u)$ 
with $x\not\in\{s_1,\ldots, s_\beta\}$.

In this section, we will establish  fundamental properties for a lollipop in an HZ-graph, which will enable us to show that a lollipop in an HZ-graph is elementary in the next section, revealing the truth of Theorem~\ref{Thm:vizing-fan}. 

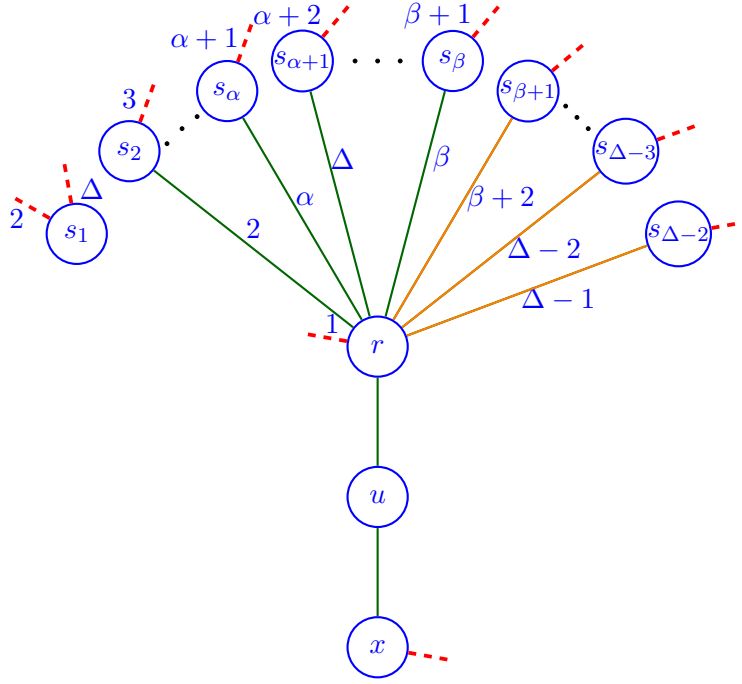
\begin{figure}[!htb]
	\begin{center}
		\begin{tikzpicture}[scale=1]
		
		{\tikzstyle{every node}=[draw ,circle,fill=white, minimum size=0.8cm,
			inner sep=0pt]
			\draw[blue,thick](0,-3) node (r)  {$r$};
			\draw [blue,thick](-4, -1.5) node (s1)  {$s_1$};
			\draw [blue,thick](-3.3, -0.4) node (s2)  {$s_2$};
			\draw[blue,thick] (-2, 0.4) node (sa)  {$s_\alpha$};
			\draw [blue,thick](-1, 0.8) node (sa2)  {$s_{\alpha+1}$};
			\draw [blue,thick](1, 0.8) node (sb)  {$s_{\beta}$};
			\draw [blue,thick](2, 0.4) node (sb2)  {$s_{\beta+1}$};
			\draw[blue,thick] (3.3, -0.4) node (sd1)  {$s_{\Delta-3}$};
			\draw [blue,thick](4, -1.5) node (sd2)  {$s_{\Delta-2}$};
			\draw [blue,thick](0, -5) node (u)  {$u$};
			\draw [blue,thick](0, -7) node (x)  {$x$};
		}
		\path[draw,thick,black!60!green]
		(r) edge node[name=la,above,pos=0.5] {\color{blue}$2$} (s2)
		(r) edge node[name=la,pos=0.6] {\color{blue}\quad$\alpha$} (sa)
		(r) edge node[name=la,pos=0.7] {\color{blue}\quad$\Delta$} (sa2)
		(r) edge node[name=la,pos=0.7] {\color{blue}\quad$\beta$} (sb)
		(r) edge node[name=la,pos=0.6] {\color{blue}\qquad\,\,\,$\beta+2$} (sb2)
		(r) edge node[name=la,pos=0.5] {\color{blue}\quad \qquad$\Delta-2$} (sd1)
		(r) edge node[name=la,pos=0.4] {\color{blue}\qquad\quad\,\,\,\,\,$\Delta-1$} (sd2)
		(r) edge node[name=la,pos=0.4] {\color{blue}\qquad\quad\,\,\,} (u)
		(u) edge node[name=la,pos=0.4] {\color{blue}\qquad\quad\,\,\,} (x);

		\draw[orange, thick] (r) --(sb2); 
		\draw [orange, thick](r) --(sd1); 
		\draw [orange, thick](r) --(sd2); 
		
		\draw[dashed, red, line width=0.5mm] (r)--++(170:1cm); 
		\draw[dashed, red, line width=0.5mm] (s1)--++(150:1cm); 
		\draw[dashed, red, line width=0.5mm] (s1)--++(100:1cm); 
		\draw[dashed, red, line width=0.5mm] (s2)--++(70:1cm); 
		\draw[dashed, red, line width=0.5mm] (sa)--++(70:1cm); 
		\draw[dashed, red, line width=0.5mm] (sa2)--++(50:1cm); 
		\draw[dashed, red, line width=0.5mm] (sb)--++(50:1cm); 
		\draw[dashed, red, line width=0.5mm] (sb2)--++(40:1cm); 
		\draw[dashed, red, line width=0.5mm] (sd1)--++(20:1cm); 
		\draw[dashed, red, line width=0.5mm] (sd2)--++(10:1cm); 
			\draw[dashed, red, line width=0.5mm] (x)--++(350:1cm); 
		\draw[blue] (-0.6, -2.7) node {$1$};  
		\draw[blue] (-4.8, -1.3) node {$2$};  
		\draw[blue] (-3.8, -0.9) node {$\Delta$};  
		\draw[blue] (-3.3, 0.3) node {$3$};  
		\draw[blue] (-2.3, 1.1) node {$\alpha+1$};  
		\draw[blue] (-1.2, 1.4) node {$\alpha+2$};  
		\draw[blue] (0.8, 1.4) node {$\beta+1$}; 
		
			{\tikzstyle{every node}=[draw ,circle,fill=black, minimum size=0.05cm,
			inner sep=0pt]
			\draw(-2.4,0.06) node (f1)  {};
			\draw(-2.6,-0.1) node (f1)  {};
			\draw(-2.8,-0.3) node (f1)  {};
			\draw(-0.3,0.8) node (f1)  {};
			\draw(0,0.8) node (f1)  {};
			\draw(0.3,0.8) node (f1)  {};
			\draw(2.5,0.2) node (f1)  {};
			\draw(2.65,0.05) node (f1)  {};
			\draw(2.8
			,-0.1) node (f1)  {};
		} 
		
		\end{tikzpicture}
	\end{center}
	\caption{A lollipop centered at $r$, where $x$ can be the same as some $s_i$ for $i\in [\beta+1, \Delta-2]$.}
\label{lol}
\end{figure}

\subsubsection{Fundamental properties of a lollipop}
\begin{LEM}\label{Lemma:extended multifan}
Let  $G$ be an HZ-graph with maximum degree $\Delta\ge 3$, $r\in V_\Delta$,  $N_{\Delta-1}(r)=\{s_1, s_2,\ldots, s_{\Delta-2}\}$, and $\varphi\in \CC^\Delta(G-rs_1)$, and let 
$F=F_\varphi(r,s_1:s_\alpha:s_\beta)$ be a typical multifan and  $L=(F,ru,u,ux,x)$ be a lollipop centered at $r$ such that  $\varphi(ru)=\alpha+1$.  If $\pbar(x)=\alpha+1$,
then $\varphi(ux)\ne 1$.  Furthermore,  if $\varphi(ux)=\tau$ 
is a 2-inducing color with respect to $\varphi$ and $F$, then  the following statements hold.
	\begin{enumerate}[(a)]
		\item   $ux\in P_r(\tau,1)$.  \label{Evizingfan-a}

		\item Let $P_x(\tau,1)$ be the $(\tau,1)$-chain starting at $x$ not containing $ux$. Then $P_x(\tau,1)$ ends at $r$.  \label{Evizingfan-b}
		\item For any 2-inducing color $\delta$ with $\tau\prec \delta$, 
		$r\in P_{s_1}(\Delta, \delta)=P_{s_{\delta-1}}(\Delta,\delta)$. \label{Evizingfan-c}
		\item  For any $\Delta$-inducing color $\delta$, $r\in P_{s_{\delta-1}}(\delta, \alpha+1 )=P_{s_{\alpha}}(\delta, \alpha+1)$, where $s_{\Delta-1}=s_1$ if $\delta=\Delta$.
		 \label{Evizingfan-d}
		 \item For any 2-inducing color $\delta$ with $\delta\prec \tau$, 
		 $r\in P_{s_\alpha}(\delta, \alpha+1)=P_{s_{\delta-1}}(\delta, \alpha+1)$. \label{Evizingfan-e}
	\end{enumerate}
\end{LEM}

\begin{proof}
	The assertion $\varphi(ux)\ne 1$ is clear. As otherwise, 
	$P_r(1,\alpha+1,\varphi)=rux$, contradicting Lemma~\ref{thm:vizing-fan1} \eqref{thm:vizing-fan1b} that 
 $r$ and $s_\alpha$ are $(1,\alpha+1)$-linked  with respect to $\varphi$. 
	
	Suppose that~\eqref{Evizingfan-a} fails, then   $ux\notin P_r(1,\tau)$. 
	Since $\tau$ is 2-inducing, $\tau\ne 1$.   Let $Q$ be the $(1,\tau)$-chain containing $u$.  Then $Q$ does not contain $r$ or $s_{\tau-1}$, since 
	$r$ and $s_{\tau-1}$ are $(1,\tau)$-linked with respect to $\varphi$ by 
	Lemma~\ref{thm:vizing-fan1} \eqref{thm:vizing-fan1b}. 
Therefore $\varphi'=\varphi/Q$  is $F$-stable. Consequently,  $r$ and $s_\alpha$ are still $(1,\alpha+1)$-linked  with respect to $\varphi'$.   However $P_r(1,\alpha+1,\varphi')$ ends at $x$,  giving a contradiction.

	For statement~\eqref{Evizingfan-b},  by~\eqref{Evizingfan-a},  it follows that $ux\in P_r(\tau,1)$. 
	Thus $P_x(\tau,1)$ is a subpath of $P_r(\tau,1)=P_{s_{\tau-1}}(\tau,1)$. 
	So $P_x(\tau,1)$   ends at either $r$  or $s_{\tau-1}$. 
	Assume to the contrary that $P_{x}(\tau,1)$ ends at $s_{\tau-1}$. Then $P_r(1,\tau)$ meets $u$ before $x$, and so $P_{[s_\tau,u]}(1,\tau)$ does not contain any edge from the lollipop $L$. Hence we can do the following operations:
	\[
\begin{bmatrix}
		s_\tau:s_\alpha & P_{[s_\tau,u]}(1,\tau)  & ux & ur\\
		\text{shift} & 1/\tau & \tau \rightarrow \alpha+1 & \alpha+1 \rightarrow 1
	\end{bmatrix}.
	\]
Clearly $(r,rs_1,s_1,\ldots,s_{\tau-1})$ is still a multifan under the new coloring, but $\tau$ is missing at both $r$ and $s_{\tau-1}$, showing a contradiction to Lemma~\ref{thm:vizing-fan1} \eqref{thm:vizing-fan1a}.	

Before proving the remaining statements, we introduce a new coloring $\varphi^*$ established on statement ~\eqref{Evizingfan-b}. Let $\varphi^*$ be the coloring obtained from $\varphi$ by doing the following operations:
	\[
\begin{bmatrix}
		P_{x}(1,\tau)  & rux\\
	1/\tau & \tau/(\alpha+1)
	\end{bmatrix}.
	\]
where $P_{x}(1,\tau)$ is defined in~\eqref{Evizingfan-b}. Let $E_{ch}=E(P_{x}(1,\tau))\cup \{ux,ur\}$. Clearly $\varphi^*$ and $\varphi$ agree on all edges in $E(G)\setminus E_{ch}$. Note that $\pbar^*(r)=\alpha+1$ and $\pbar^*(s)=\pbar(s)$ for all $s\in V(F)\setminus\{r\}$, and $(r,rs_1,s_1,rs_2,s_2\ldots,s_{\tau-1})$ and $(r,rs_1,s_1,rs_{\alpha+1},s_{\alpha+1},\ldots,s_{\beta})$ are multifans under $\varphi^*$. These properties will be frequently used in the following proof.

%

%
%

Now for the statement(c), firstly  
 $P_{s_1}(\Delta, \delta)=P_{s_{\delta-1}}(\Delta,\delta)$ by Lemma~\ref{thm:vizing-fan2} \eqref{thm:vizing-fan2-a}.
 Assume to the contrary that $r\not\in P_{s_1}(\Delta, \delta)$. Since $\{\Delta,\delta\}\cap \{1,\tau\}=\emptyset$, $P_{s_1}(\Delta, \delta)=P_{s_{\delta-1}}(\Delta,\delta)$ does not contain any edge from $E_{ch}$. Thus $P_{s_1}(\Delta, \delta,\varphi^*)=P_{s_{\delta-1}}(\Delta,\delta,\varphi^*)=P_{s_1}(\Delta, \delta,\varphi)$, and so $r\notin P_{s_1}(\Delta, \delta,\varphi^*)$. Let $\varphi'=\varphi^*/P_{s_1}(\Delta, \delta,\varphi^*)$. Then $\delta\in \pbar'(s_1)$ and $(r,rs_1,s_1,rs_{\delta},s_{\delta},\ldots,s_{\alpha})$ is a multifan under $\varphi'$.  However $\alpha+1$ is missing at both $r$ and $s_{\alpha}$, giving a contradiction to Lemma~\ref{thm:vizing-fan1} \eqref{thm:vizing-fan1a}.


For  statement~\eqref{Evizingfan-d},  firstly 
$P_{s_{\delta-1}}(\alpha+1, \delta)=P_{s_{\alpha}}(\alpha+1, \delta)$ by Lemma~\ref{thm:vizing-fan2} \eqref{thm:vizing-fan2-a}.
Assume to the contrary that $r\not\in P_{s_{\delta-1}}(\alpha+1, \delta)$. Since $\{\alpha+1,\delta\}\cap \{1,\tau\}=\emptyset$, $P_{s_\alpha}(\alpha+1, \delta)=P_{s_{\delta-1}}(\alpha+1,\delta)$ does not contain any edge from $E_{ch}$. Thus $P_{s_\alpha}(\alpha+1, \delta,\varphi^*)=P_{s_{\delta-1}}(\alpha+1,\delta,\varphi^*)=P_{s_{\delta-1}}(\alpha+1, \delta,\varphi)$, and so $r\notin P_{s_{\delta-1}}(\alpha+1, \delta,\varphi^*)$. 
This gives a contradiction to Lemma~\ref{thm:vizing-fan1}~\eqref{thm:vizing-fan1b} that 
$r$ and $s_{\delta-1}$ are $(\alpha+1,\delta)$-linked with respect to $\varphi^*$, since 
$(r,rs_1,s_1,rs_{\alpha+1},s_{\alpha+1},\ldots,s_{\beta})$ is a multifan under $\varphi^*$.


For  statement~\eqref{Evizingfan-e},  if it fails then we have that either  $P_{s_\alpha}(\alpha+1, \delta)\ne P_{s_{\delta-1}}(\alpha+1,\delta )$ or 
$P_{s_\alpha}(\alpha+1, \delta)=P_{s_{\delta-1}}(\alpha+1,\delta)$ but 
$r\not\in P_{s_\alpha}(\alpha+1, \delta)$. 
For the first case,  $r\in P_{s_\alpha}(\alpha+1, \delta)$ by Lemma~\ref{thm:vizing-fan2}~\eqref{thm:vizing-fan2-b} and so $r\not\in P_{s_{\delta-1}}(\alpha+1,\delta )$. 
Therefore  $r\not\in P_{s_{\delta-1}}(\alpha+1,\delta )$ in both cases. 
Consequently, $E(P_{\delta-1}(\alpha+1, \delta))\cap E_{ch}=\emptyset$.
Hence,  $P_{s_{\delta-1}}(\alpha+1,\delta,\varphi^*)= P_{s_{\delta-1}}(\alpha+1, \delta,\varphi)$
and $r\not\in P_{s_{\delta-1}}(\alpha+1,\delta,\varphi^*)$.  This gives a contradiction, since 
under $\varphi^*$, $(r,rs_1,s_1,rs_2,s_2\ldots,s_{\tau-1})$ is a multifan, and so $r$ and $s_{\delta-1}$ should be $(\alpha+1,\delta)$-linked Lemma~\ref{thm:vizing-fan1}~\eqref{thm:vizing-fan1b}.  This finishes  the proof of statement~\eqref{Evizingfan-e} and Lemma~\ref{Lemma:extended multifan}.  \end{proof}

	\begin{LEM}\label{Lemma:pseudo-fan0}
	Let  $G$ be an HZ-graph with maximum degree $\Delta\ge 3$, $r\in V_\Delta$,  $N_{\Delta-1}(r)=\{s_1, s_2,\ldots, s_{\Delta-2}\}$, and $\varphi\in \CC^\Delta(G-rs_1)$, and let 
	$F:=F_\varphi(r,s_1:s_\alpha:s_\beta)$ be a typical multifan and  $L:=(F,ru,u,ux,x)$ be a lollipop centered at $r$ such that  $\varphi(ru)=\alpha+1$.   Then for   $s_{h_1}\in \{s_{\beta+1}, \ldots, s_{\Delta-2} \}$ with $\varphi(rs_{h_1})=\tau_1\in \{\beta+2, \cdots, \Delta-1\}$,  the following statements hold.
		\begin{enumerate}[(1)] 
			\item \label{only-fan}	
			If exists a vertex $w\in V(G)\setminus (V(F)\cup \{s_{h_1}\})$ such that $w\in P_r(\tau_1,1,\varphi')$
			for every {\bf $F$-stable} $\varphi' \in \CC^{\Delta}(G-rs_1)$,   
					then there exists a sequence of distinct vertices $s_{h_1},s_{h_2}, \ldots, s_{h_t}\in  \{s_{\beta+1}, \ldots, s_{\Delta-2} \}$ satisfying the following conditions: 
		\begin{enumerate}[(a)]
			\item  $\varphi(rs_{h_{i+1}})=\pbar(s_{h_i})\in  \{\beta+2, \cdots, \Delta-1\}$ for each $i\in [1,t-1]$;
						 \label{Lemma:pseudo-fan0-a}
			\item  $s_{h_i}$ and $r$ are $(\pbar(s_{h_i}), 1)$-linked with respect to $\varphi$ for each $i\in[1,t]$; \label{Lemma:pseudo-fan0-b}
			\item $\pbar(s_{h_t})=\tau_1$. \label{Lemma:pseudo-fan0-c}
				\end{enumerate}
			\item \label{fan-and-x}
			If  $\pbar(x)=\alpha+1$ and there exists a vertex $w\in V(G)\setminus (V(F)\cup \{s_{h_1}\})$ such that $w\in P_r(\tau_1,1,\varphi')$
for every  {\bf $L$-stable} $\varphi' \in \CC^{\Delta}(G-rs_1)$  obtained from $\varphi$ 
through a sequence of Kempe  $(1,*)$-changes, 
then there exists a sequence of distinct vertices $s_{h_1},s_{h_2}, \ldots, s_{h_t}\in  \{s_{\beta+1}, \ldots, s_{\Delta-2} \}$ satisfying the following conditions: 
	\begin{enumerate}[(a)]
	\item  $\varphi(rs_{h_{i+1}})=\pbar(rs_{h_i})\in  \{\beta+2, \cdots, \Delta-1\}$  for each $i\in [1,t-1]$;
	\label{Lemma:pseudo-fan0-a1}
	\item $s_{h_i}$ and $r$ are $(\pbar(s_{h_i}), 1)$-linked with respect to $\varphi$ for each $i\in[1,t-1]$; \label{Lemma:pseudo-fan0-b1}
	\item $\pbar(s_{h_t})=\tau_1$ or $\pbar(s_{h_t})=\alpha+1$.  If $\pbar(s_{h_t})=\tau_1$,
	then $s_{h_t}$ and $r$ are $(\tau_{1}, 1)$-linked with respect to $\varphi$. \label{Lemma:pseudo-fan0-c1}
\end{enumerate}
	\end{enumerate}
\end{LEM} 
\begin{proof}
	We show~\eqref{only-fan} and \eqref{fan-and-x} simultaneously.  Let $\pbar(s_{h_1})=\tau_2$.
	For \eqref{fan-and-x}, we may assume that $\tau_2\ne \alpha+1$, as otherwise, $s_{h_1}$
	is the desired sequence for it.  Note that $\tau_2\ne 1$, since otherwise $w\notin P_r(\tau_1,1,\varphi)=rs_{h_1}$, giving a contradiction.
	We claim that $s_{h_1}$ satisfies $(a)$ and $(b)$ in  \eqref{only-fan}. 
	If  $(a)$ fails, then $\tau_2\in [1,\Delta]\setminus\{\beta+2,\ldots, \Delta-1\}=\pbar(V(F))$, so $r$ and $\mathit{\pbar^{-1}_F(\tau_2)}$ are $(\tau_2,1)$-linked by Lemma~\ref{thm:vizing-fan1}~\eqref{thm:vizing-fan1b}; if $(b)$ fails, then $s_{h_1}$ and $r$ are $(\tau_2,1)$-unlinked. In both cases, we have $s_{h_1}$ and $r$ are $(\tau_2,1)$-unlinked. Now let $\varphi'=\varphi/P_{s_{h_1}}(\tau_2,1)$.  
	Clearly, $\varphi'$ is $F$-stable. Since 
	$r$ and $s_\alpha$ are $(1,\alpha+1)$-linked 
	with respect to $\varphi$ by Lemma~\ref{thm:vizing-fan1}~\eqref{thm:vizing-fan1b}
	and $\varphi(ru)=\alpha+1$, 
	we have  $\varphi'(ru)=\varphi(ru)$.
	For ~\eqref{only-fan}, we already achieve a contradiction
	since $\pbar'(s_{h_1})=1$ implies 
	$w\not\in P_r(\tau_1,1,\varphi')=rs_{h_1}$. 
	For \eqref{fan-and-x},  since  $\tau_2\ne \alpha+1$,  $\pbar'(x)=\pbar(x)$. 
	By Lemma~\ref{Lemma:extended multifan},
	the color $\varphi(ux)$ on $ux$ will keep unchanged under 
	any Kempe $(1,*)$-change not involving vertices from $V(F)\cup\{x\}$. 
	Thus, $\varphi'$ is $L$-stable. We again reach a contradiction 
	since  $w\notin P_r(\tau_1,1,\varphi')=rs_{h_1}$.

	Now $s_{h_1}$ is a sequence that satisfies $(a)$ and $(b)$ in  \eqref{only-fan}.  
	 Let $s_{h_1}, \ldots, s_{h_k}$ be a longest sequence of vertices from $ \{s_{\beta+1}, \ldots, s_{\Delta-2} \}$ that satisfies $(a)$ and $(b)$ in \eqref{only-fan}. 
	 We are done if $\pbar(s_{h_k})=\tau_1$. 
	 Thus, assume  $\pbar(s_{h_k})=\tau_{k+1} \ne \tau_1$. 
	 By $(a)$ we have $\tau_{k+1}\in\{\beta+1, \ldots, \Delta-1\}$. Since each $s_{h_i}$, including $s_{h_{k}}$, is $(\tau_{i+1},1)$-linked with $r$, we know  $\tau_{k+1}\notin \{\tau_1,\tau_2,\ldots,\tau_{k}\}$. Thus, there exists $s_{h_{k+1}}\in N_{\Delta-1}(r)$ such that $\varphi(rs_{h_{k+1}})=\tau_{k+1}$.   Let $\pbar(s_{h_i})=\tau_{i+1}$ for each $i\in [1,k+1]$. 
	 By the maximality of the sequence $s_{h_1},\ldots, s_{h_k}$, 
	 either  $\pbar(s_{h_{k+1}})\in \pbar(V(F))$ or $\pbar(s_{h_{k+1}})\in \{\beta+2,\ldots, \Delta-1\}$ and $s_{h_{k+1}}$ and $r$ are $(\tau_{k+2},1)$-unlinked.  
	 In both cases,  
	 $s_{h_{k+1}}$ and $r$ are $(\tau_{k+2},1)$-unlinked.  
	 Again,  for \eqref{fan-and-x},  we assume  $\pbar(s_{h_{k+1}}) \ne \alpha+1$. 
	 Thus, we assume that we are  proving \eqref{only-fan} and proving \eqref{fan-and-x}
	under the assumption that $\pbar(s_{h_{k+1}})\ne \alpha+1$. 
		In both cases, we 
	do a sequence of Kempe changes around $r$ from $s_{h_{k+1}}$ to $s_{h_1}$ as below 
	to reach a contradiction: 
	\begin{enumerate}[(1)]
		\item Swap colors along $P_{s_{h_{k+1}}}(\tau_{k+2},1)$ (after (1), $P_r(1, \tau_{k+1})=rs_{{h_{k+1}}}$); \vspace{-2mm}
		\item Swap colors along $P_{s_{h_{k}}}(\tau_{k+1},1)$ (after (2), $P_r(1, \tau_{k})=rs_{{h_{k}}}$); \vspace{-2mm}
		\item Continue the same kind of Kempe change from $s_{h_{k-1}}$ to $s_{h_3}$;  \vspace{-2mm}
		\item Swap colors along $P_{s_{h_{2}}}(\tau_{3},1)$ (after (4), $P_r(1, \tau_{2})=rs_{{h_{2}}}$);  \vspace{-2mm}
		\item Swap colors along $P_{s_{h_{1}}}(\tau_{2},1)$ (after (5), $P_r(1, \tau_{1})=rs_{{h_{1}}}$). 
	\end{enumerate}
	
	Let the current  coloring  be $\varphi'$. Clearly,  $\varphi'$ is obtained from $\varphi$ 
	through a sequence of   Kempe $(1,*)$-changes. For \eqref{only-fan},  
	$\varphi'$ is $F$-stable with respect to $\varphi$ with $\varphi'(ru)=\varphi(ru)$. 
	For the case of proving~\eqref{fan-and-x},   as we assumed  $\pbar(s_{h_{k+1}})\ne \alpha+1$, 
	we have      $\pbar'(x)=\pbar(x)$.  By Lemma~\ref{Lemma:extended multifan},  	the color $\varphi(ux)$ on $ux$ will keet unchanged under 
	any Kempe $(1,*)$-change not involving vertices from $V(F)\cup\{x\}$. 
	Thus, $\varphi'$
	is $L$-stable with respect to $\varphi$. 
	However, in both cases, 
	$w\not\in P_r(\tau_1,1,\varphi')$.  
	This gives a contradiction to the assumptions  in ~\eqref{only-fan} and \eqref{fan-and-x}. 
\end{proof}

\subsubsection{Adjacency in   a lollipop}
Let $G$ be an $HZ$ graph,   $e=rs_1\in E(G)$ with $r\in V_\Delta$ and 
$s_1\in V_{\Delta-1}$, and $\varphi\in \CC^\Delta(G-e)$, 
and let $F:=F_\varphi(r,s_1:s_\alpha:s_\beta)$ be a typical multifan.  
In this subsection, we show that if  there exists a lollipop $L=(F, ru, u,ux, x)$ in $G$, 
 then $u$ is not adjacent to at least two vertices in $\{s_1, \ldots, s_\beta\}$ (if $\beta\ge 2$).
The existence of more small degree neighbors of $u$ outside the multifan $F$ provides us more room to work 
on. This extra structure will provide us a tool to proof Theorem~\ref{Thm:vizing-fan} in Section~\ref{thm2.1proof}.


\begin{LEM}\label{Lem:2-non-adj1}
	Let  $G$ be an HZ-graph with maximum degree $\Delta\ge 3$, $r\in V_\Delta$,  $N_{\Delta-1}(r)=\{s_1, s_2,\ldots, s_{\Delta-2}\}$, and $\varphi\in \CC^\Delta(G-rs_1)$, and let 
	$F:=F_\varphi(r,s_1:s_\alpha)$ be a typical 2-inducing  multifan and  $L:=(F,ru,u,ux,x)$ be a lollipop centered at $r$.  If $\varphi(ru)=\alpha+1$, $\pbar(x)=\alpha+1$, 
	and $\varphi(ux)=\Delta$, 
	then  $u\not\sim s_1$ and $u\not\sim s_\alpha$. 
\end{LEM}

Instead of proving Lemma~\ref{Lem:2-non-adj1}, we prove the following stronger result,
besides implies Lemma~\ref{Lem:2-non-adj1},  which will also be used to prove 
Lemma~\ref{Lem:2-non-adj2}. 

\begin{LEM}\label{Lem:2-non-adj1*}
	Let  $G$ be an HZ-graph with maximum degree $\Delta\ge 3$, $r\in V_\Delta$,  $N_{\Delta-1}(r)=\{s_1, s_2,\ldots, s_{\Delta-2}\}$, and $\varphi\in \CC^\Delta(G-rs_1)$, and let 
$F:=F_\varphi(r,s_1:s_\alpha:s_\beta)$ be a typical multifan and  $L:=(F,ru,u,ux,x)$ be a lollipop centered at $r$.  If $\varphi(ru)=\alpha+1$, $\pbar(x)=\alpha+1$, 
 and $\varphi(ux)=\Delta$, 
then  the following two statements hold.
\begin{enumerate}[(1)]
\item If $u\sim s_1$, then $\varphi(us_1)$ is a $\Delta$-inducing color. 
\item If $u\sim s_\alpha$, then $\varphi(us_\alpha)$ is a $\Delta$-inducing color.
\end{enumerate}
\end{LEM}

\begin{proof}
	Assume to the contrary that the statements fail. We 
	naturally have two cases. 
	
	\medskip 
	
	\emph{\noindent Case 1: $u\sim s_1$} and $\varphi(us_1)$ is not a $\Delta$-inducing color.  
			Let $\varphi(us_1)=\tau$. 	Note that $\tau\ne 2, \alpha+1, \Delta$. We first show that $us_1$ can not be 1 under any $L$-stable coloring.


	\begin{CLA}\label{not1}
For every  $L$-stable $\varphi^*\in \CC^\Delta(G-rs_1)$, it holds that $\varphi^*(us_1)\ne 1$.
Furthermore, if 
 $\varphi^*(us_1)=\varphi(us_1)=\tau$,   then $us_1\in P_r(\tau,1, \varphi^*)$. 
	\end{CLA}
\proof[Proof of Claim~\ref{not1}] 
Suppose instead  that $\varphi^*(us_{1})=1$ for the first part,  
and $us_{1}\notin P_r(\tau,1, \varphi^*)$  for the second part.  Let $\varphi'=\varphi^*$ in the former case and let $\varphi'=\varphi^*/Q$ in the latter case,  where $Q$ is the $(\tau,1)$-chain containing $us_1$. Clearly $\varphi'$ is $F$-stable with respect to $\varphi$. Since $P_r(\alpha+1,1,\varphi')= P_{s_\alpha}(\alpha+1,1,\varphi')$ by Lemma~\ref{thm:vizing-fan1}~\eqref{thm:vizing-fan1b}, $P_x(\alpha+1,1,\varphi')$ does not contain $r$. Thus $\varphi''=\varphi'/P_x(\alpha+1,1,\varphi')$ is $F$-stable with respect to $\varphi^*$ and $\varphi''(us_1)=\varphi'(us_1)=1$.  However,  $P_{s_1}(\Delta,1,\varphi'')=s_1ux$, contradicting Lemma~\ref{thm:vizing-fan1}~\eqref{thm:vizing-fan1b} that $s_1$ and $r$ are $(\Delta,1)$-linked with respect to $\varphi''$.  
\qed 

	\medskip 

\emph{\noindent Subcase 1.1: $\tau\in \pbar(V(F))$ is $2$-inducing}.  A precoloring  of  $L$ in this case  is depicted in Figure~\ref{f2}. 

\smallskip 
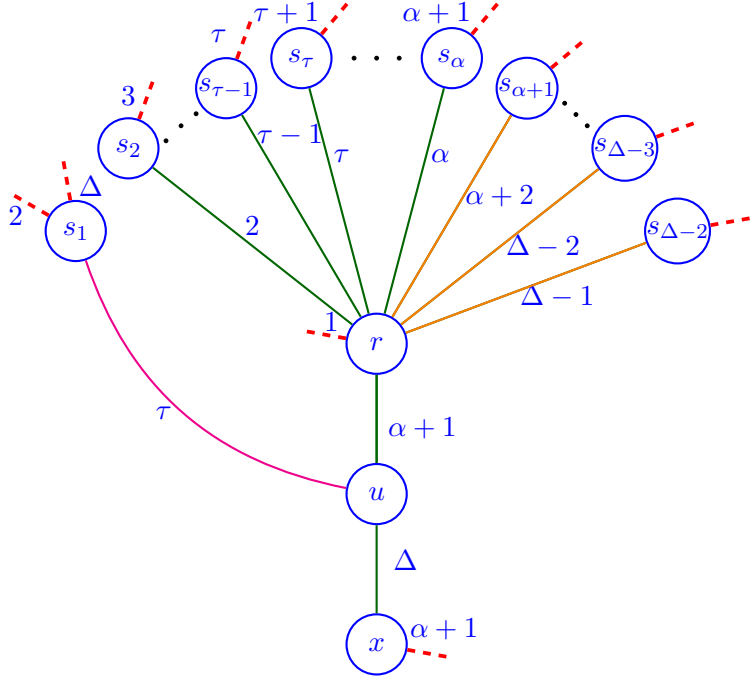
\begin{figure}[!htb]
	\begin{center}
		\begin{tikzpicture}[scale=1]
		
		{\tikzstyle{every node}=[draw ,circle,fill=white, minimum size=0.8cm,
			inner sep=0pt]
			\draw[blue,thick](0,-3) node (r)  {$r$};
			\draw [blue,thick](-4, -1.5) node (s1)  {$s_1$};
			\draw [blue,thick](-3.3, -0.4) node (s2)  {$s_2$};
			\draw[blue,thick] (-2, 0.4) node (sa)  {$s_{\tau-1}$};
			\draw [blue,thick](-1, 0.8) node (sa2)  {$s_{\tau}$};
			\draw [blue,thick](1, 0.8) node (sb)  {$s_{\alpha}$};
			\draw [blue,thick](2, 0.4) node (sb2)  {$s_{\alpha+1}$};
			\draw[blue,thick] (3.3, -0.4) node (sd1)  {$s_{\Delta-3}$};
			\draw [blue,thick](4, -1.5) node (sd2)  {$s_{\Delta-2}$};
			\draw [blue,thick](0, -5) node (u)  {$u$};
			\draw [blue,thick](0, -7) node (x)  {$x$};
		}
		\path[draw,thick,black!60!green]
		(r) edge node[name=la,above,pos=0.5] {\color{blue}$2$} (s2)
		(r) edge node[name=la,pos=0.9] {\color{blue}\quad\quad \,\,\,$\tau-1$} (sa)
		(r) edge node[name=la,pos=0.7] {\color{blue}\quad$\tau$} (sa2)
		(r) edge node[name=la,pos=0.7] {\color{blue}\quad$\alpha$} (sb)
		(r) edge node[name=la,pos=0.6] {\color{blue}\qquad\,\,\,$\alpha+2$} (sb2)
		(r) edge node[name=la,pos=0.6] {\color{blue}\qquad\,\quad$\alpha+1$} (u)
		(r) edge node[name=la,pos=0.5] {\color{blue}\quad \qquad$\Delta-2$} (sd1)
		(r) edge node[name=la,pos=0.4] {\color{blue}\qquad\quad\,\,\,\,\,$\Delta-1$} (sd2)
		(r) edge node[name=la,pos=0.4] {\color{blue}\qquad\quad\,\,\,} (u)
		(u) edge node[name=la,pos=0.4] {\color{blue}\qquad$\Delta$} (x);
		\path[draw,thick,magenta]
		(u) edge [bend left] node [left,name=la,pos=0.5] {\color{blue}\qquad$\tau$} (s1);

		\draw[orange, thick] (r) --(sb2); 
		\draw [orange, thick](r) --(sd1); 
		\draw [orange, thick](r) --(sd2); 
		
		\draw[dashed, red, line width=0.5mm] (r)--++(170:1cm); 
		\draw[dashed, red, line width=0.5mm] (s1)--++(150:1cm); 
		\draw[dashed, red, line width=0.5mm] (s1)--++(100:1cm); 
		\draw[dashed, red, line width=0.5mm] (s2)--++(70:1cm); 
		\draw[dashed, red, line width=0.5mm] (sa)--++(70:1cm); 
		\draw[dashed, red, line width=0.5mm] (sa2)--++(50:1cm); 
		\draw[dashed, red, line width=0.5mm] (sb)--++(50:1cm); 
		\draw[dashed, red, line width=0.5mm] (sb2)--++(40:1cm); 
		\draw[dashed, red, line width=0.5mm] (sd1)--++(20:1cm); 
		\draw[dashed, red, line width=0.5mm] (sd2)--++(10:1cm); 
		\draw[dashed, red, line width=0.5mm] (x)--++(350:1cm); 
		\draw[blue] (-0.6, -2.7) node {$1$};  
		\draw[blue] (-4.8, -1.3) node {$2$};  
		\draw[blue] (-3.8, -0.9) node {$\Delta$};  
		\draw[blue] (-3.3, 0.3) node {$3$};  
		\draw[blue] (-2.1, 1.1) node {$\tau$};  
		\draw[blue] (-1.2, 1.4) node {$\tau+1$};  
		\draw[blue] (0.8, 1.4) node {$\alpha+1$}; 
		\draw[blue] (0.9, -6.8) node {$\alpha+1$}; 
		{\tikzstyle{every node}=[draw ,circle,fill=black, minimum size=0.05cm,
			inner sep=0pt]
			\draw(-2.4,0.06) node (f1)  {};
			\draw(-2.6,-0.1) node (f1)  {};
			\draw(-2.8,-0.3) node (f1)  {};
			\draw(-0.3,0.8) node (f1)  {};
			\draw(0,0.8) node (f1)  {};
			\draw(0.3,0.8) node (f1)  {};
			\draw(2.5,0.2) node (f1)  {};
			\draw(2.65,0.05) node (f1)  {};
			\draw(2.8
			,-0.1) node (f1)  {};
		} 
		
		\end{tikzpicture}
	\end{center}
	\caption{Precoloring of $L$ in Subcase 1.1 of Lemma~\ref{Lem:2-non-adj1*}}
	\label{f2}
	
\end{figure}

By Claim~\ref{not1},  $us_1\in P_r(\tau,1)=P_{s_{\tau-1}}(\tau,1)$. Let $P_u(\tau,1)$ be the subpath of $P_r(\tau,1)$ starting at $u$ not containing $us_1$. Then $P_u(\tau,1)$ ends at  either $r$ or $s_{\tau-1}$. Consequently if we shift from $s_\tau$ to $s_\alpha$, then $P_u(\tau,1)$ will end at  either $s_\tau$ or $s_{\tau-1}$. Thus we can do the following operations:
\[
\begin{bmatrix}
s_{\tau}:s_{\alpha} & P_{u}(\tau,1)  & us_1 & ux & ur\\
\text{shift} & \tau/1 & \tau \rightarrow \Delta & \Delta \rightarrow \alpha+1 & \alpha+1 \rightarrow 1
\end{bmatrix}.
\]

Denote the new coloring by $\varphi'$. Now $\pbar'(s_1)=\pbar'(r)=\{\tau\}$, we can color $rs_1$ by $\tau$ to obtain a $\Delta$-edge coloring of $G$, which contradicts the fact that $G$ is Class 2.

\emph{\noindent Subcase 1.2: $\tau\notin \pbar(V(F))$}.  
Thus $\tau\in \{\beta+2,\ldots, \Delta-1\}$. 
 A  precoloring of  $L$ in this case is depicted in  Figure~\ref{f3}.

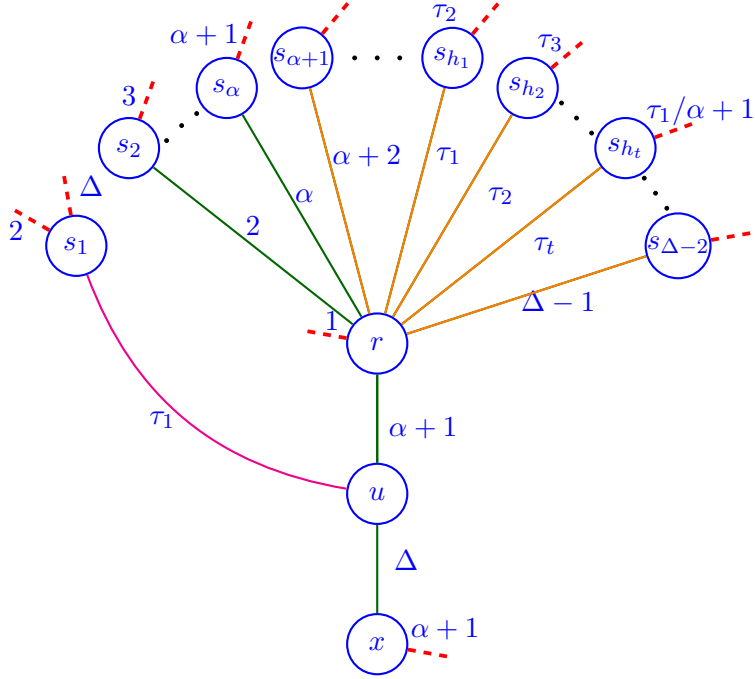
\begin{figure}[!htb]
	\begin{center}
		\begin{tikzpicture}[scale=1]
		
		{\tikzstyle{every node}=[draw ,circle,fill=white, minimum size=0.8cm,
			inner sep=0pt]
			\draw[blue,thick](0,-3) node (r)  {$r$};
			\draw [blue,thick](-4, -1.7) node (s1)  {$s_1$};
			\draw [blue,thick](-3.3, -0.4) node (s2)  {$s_2$};
			\draw[blue,thick] (-2, 0.4) node (sa)  {$s_\alpha$};
			\draw [blue,thick](-1, 0.8) node (sa2)  {$s_{\alpha+1}$};
			\draw [blue,thick](1, 0.8) node (sb)  {$s_{h_1}$};
			\draw [blue,thick](2, 0.4) node (sb2)  {$s_{h_2}$};
			\draw[blue,thick] (3.3, -0.4) node (sd1)  {$s_{h_t}$};
			\draw [blue,thick](4, -1.7) node (sd2)  {$s_{\Delta-2}$};
			\draw [blue,thick](0, -5) node (u)  {$u$};
			\draw [blue,thick](0, -7) node (x)  {$x$};
		}
		\path[draw,thick,black!60!green]
		(r) edge node[name=la,above,pos=0.5] {\color{blue}$2$} (s2)
		(r) edge node[name=la,pos=0.6] {\color{blue}\quad$\alpha$} (sa)
		(r) edge node[name=la,pos=0.7] {\color{blue}\quad\quad\,\,\,\,$\alpha+2$} (sa2)
		(r) edge node[name=la,pos=0.7] {\color{blue}\quad\,\,\,$\tau_1$} (sb)
		(r) edge node[name=la,pos=0.6] {\color{blue}\qquad\,\,\,$\tau_2$} (sb2)
		(r) edge node[name=la,pos=0.6] {\color{blue}\qquad\,\quad$\alpha+1$} (u)
		(r) edge node[name=la,pos=0.5] {\color{blue}\quad \qquad$\tau_t$} (sd1)
		(r) edge node[name=la,pos=0.4] {\color{blue}\qquad\quad\,\,\,\,\,$\Delta-1$} (sd2)
		(r) edge node[name=la,pos=0.4] {\color{blue}\qquad\quad\,\,\,} (u)
		(u) edge node[name=la,pos=0.4] {\color{blue}\qquad$\Delta$} (x);
		\path[draw,thick,magenta]
		(u) edge [bend left] node [left,name=la,pos=0.5] {\color{blue}\qquad$\tau_1$} (s1);
		
%
		
			\draw[orange, thick] (r) --(sb); 
		\draw[orange, thick] (r) --(sb2); 
		\draw [orange, thick](r) --(sd1); 
		\draw [orange, thick](r) --(sd2); 
		\draw [orange, thick](r) --(sa2); 
		
		\draw[dashed, red, line width=0.5mm] (r)--++(170:1cm); 
		\draw[dashed, red, line width=0.5mm] (s1)--++(150:1cm); 
		\draw[dashed, red, line width=0.5mm] (s1)--++(100:1cm); 
		\draw[dashed, red, line width=0.5mm] (s2)--++(70:1cm); 
		\draw[dashed, red, line width=0.5mm] (sa)--++(70:1cm); 
		\draw[dashed, red, line width=0.5mm] (sa2)--++(50:1cm); 
		\draw[dashed, red, line width=0.5mm] (sb)--++(50:1cm); 
		\draw[dashed, red, line width=0.5mm] (sb2)--++(40:1cm); 
		\draw[dashed, red, line width=0.5mm] (sd1)--++(20:1cm); 
		\draw[dashed, red, line width=0.5mm] (sd2)--++(10:1cm); 
		\draw[dashed, red, line width=0.5mm] (x)--++(350:1cm); 
		\draw[blue] (-0.6, -2.7) node {$1$};  
		\draw[blue] (-4.8, -1.5) node {$2$};  
		\draw[blue] (-3.8, -0.9) node {$\Delta$};  
		\draw[blue] (-3.3, 0.3) node {$3$};  
		\draw[blue] (-2.3, 1.1) node {$\alpha+1$};  
		\draw[blue] (0.9, 1.4) node {$\tau_2$}; 
		\draw[blue] (0.9, -6.8) node {$\alpha+1$}; 
			\draw[blue] (2.3, 1.0) node {$\tau_3$}; 
			\draw[blue] (4.3, 0.1) node {$\tau_1/\alpha+1$}; 
		

			
				{\tikzstyle{every node}=[draw ,circle,fill=black, minimum size=0.05cm,
				inner sep=0pt]
				\draw(-2.4,0.06) node (f1)  {};
				\draw(-2.6,-0.1) node (f1)  {};
				\draw(-2.8,-0.3) node (f1)  {};
					\draw(-0.3,0.8) node (f1)  {};
				\draw(0,0.8) node (f1)  {};
				\draw(0.3,0.8) node (f1)  {};
					\draw(2.45,0.2) node (f1)  {};
				\draw(2.65,0) node (f1)  {};
				\draw(2.85,-0.2) node (f1)  {};
				\draw(3.55,-0.8) node (f1)  {};
				\draw(3.7,-1) node (f1)  {};
				\draw(3.85,-1.2) node (f1)  {};
			} 
		\end{tikzpicture}
	\end{center}
	\caption{Precoloring of $L$ in Subcase 1.2 of Lemma~\ref{Lem:2-non-adj1*}}
	\label{f3}
\end{figure}

%
%

%
Let $\tau_1=\tau$ and $s_{h_1}\in N_{\Delta-1}(r)$ such that $\varphi(rs_{h_1})=\tau_1$. 
Since $us_1\in P_r(\tau,1, \varphi) $  and  the color on $us_{1}$ 
is not 1 under  every $L$-stable coloring by Claim~\ref{not1_3},  
for every  $L$-stable $\varphi' \in \CC^{\Delta}(G-rs_1)$  obtained from $\varphi$ 
through a sequence of  Kempe $(1,*)$-changes, it holds that 
$\varphi'(us_1)=\varphi(us_1)$.   Therefore, 
$u\in P_r(\tau,1, \varphi')$  by Claim~\ref{not1} again. 
Applying  Lemma~\ref{Lemma:pseudo-fan0} on $L$ with $u$ playing the role of $w$, 
we find a sequence of distinct vertices $s_{h_1}, \ldots, s_{h_t}\in  \{s_{\beta+1}, \ldots, s_{\Delta-2} \}$ satisfying the following conditions: 
\begin{enumerate}[(a)]
	\item  $\varphi(rs_{h_{i+1}})=\pbar(s_{h_i})=\tau_{i+1}\in  \{\beta+2, \cdots, \Delta-1\} $ for each $i\in [1,t-1]$; 
			\item $s_{h_i}$ and $r$ are $(\tau_{i+1}, 1)$-linked with respect to $\varphi$ for each $i\in [1,t-1]$;
	\item $\pbar(s_{h_t})=\tau_1$ or $\alpha+1$, and if $\pbar(s_{h_t})=\tau_1$, then  $s_{h_t}$ and $r$ are $(\tau_{1}, 1)$-linked with respect to $\varphi$. 
\end{enumerate} 
\medskip 

\emph{\noindent Subcase $\pbar(s_{h_t})=\tau_1$.}
In this case,  $t\ge 2$. By Claim~\ref{not1},  $us_1\in P_r(\tau,1)=P_{s_{h_t}}(\tau,1)$. 

{\it Suppose first that $P_{s_{h_t}}(\tau,1)$ meets $u$ before $s_1$}. 
We do the following operations:

\[
\begin{bmatrix}
P_{[s_{h_t},u]}(\tau,1)  &  us_1 & ux & ur\\
\tau/1 & \tau \rightarrow \Delta & \Delta \rightarrow \alpha+1 & \alpha+1 \rightarrow 1
\end{bmatrix}.
\]
The new coloring is $F$-stable, but $\alpha+1$ is missing at both $r$ and $s_\alpha$, giving a contradiction to Lemma~\ref{thm:vizing-fan1}~\eqref{thm:vizing-fan1b}.  {\it Thus $P_{s_{h_1}}(\tau,1)$ meets $s_1$ before $u$}.  Then shifting from $s_{h_1}$
to $s_{h_t}$  gives back to the previous case with $s_{h_1}$ playing the role of $s_{h_t}$. 


\medskip

\emph{\noindent Subcase $\pbar(s_{h_t})=\alpha+1$.}
Note that $t$ could be 1 in this case. 
By Claim~\ref{not1},  $us_1\in P_r(\tau,1)=P_{z}(\tau,1)$, for some vertex
$z\in V(G)\setminus (V(F)\cup \{x,s_{h_1}, \ldots, s_{h_t}\})$. \emph{Assume first that $s_{h_t}\ne x$}. 

{\it If $P_z(\tau,1)$ meets $u$ before $s_1$},  we do the following operations: 
\[
\begin{bmatrix}
P_{[z,u]}(\tau,1)  &  us_1 & ux & ur\\
\tau/1 & \tau \rightarrow \Delta & \Delta \rightarrow \alpha+1 & \alpha+1 \rightarrow 1
\end{bmatrix}.
\]
The new coloring is $F$-stable, but $\alpha+1$ is missing at both $r$ and $s_\alpha$, giving a contradiction to Lemma~\ref{thm:vizing-fan1}~\eqref{thm:vizing-fan1a}.


{\it If $P_z(\tau,1)$ meets $s_1$ before $u$},  we do the following operations:
\[
\begin{bmatrix}
P_{[z,s_1]}(\tau,1) & s_{h_1}:s_{h_t}  &  us_1 & ux & ur\\
\tau/1 & \text{shift} & \tau \rightarrow \Delta & \Delta \rightarrow \alpha+1 & \alpha+1 \rightarrow \tau
\end{bmatrix}.
\]
The new coloring is $F$-stable,  but $1$ is missing at both $r$ and $s_1$, giving a contradiction to Lemma~\ref{thm:vizing-fan1}~\eqref{thm:vizing-fan1a}.

%

\emph{Assume now  that $s_{h_t}= x$}. 
We first consider the case when $t\ge 2$. Note that $P_r(1,\alpha+1)=P_{s_\alpha}(1,\alpha+1)$ and so $r\notin P_x(1,\alpha+1)$.  Let $\varphi_1=\varphi/P_x(1,\alpha+1)$. Then $P_r(1,\tau_{t},\varphi_1)=rx$. We next let  $\varphi_2=\varphi_1/P_{s_{h_{t-1}}}(\tau_t,1,\varphi_1)$. Then $\varphi_2$ is $F$-stable with respect to $\varphi$ and $\pbar_2(x)=\pbar_2(s_{h_{t-1}})=1$. Now do $(\alpha+1,1)$-swap at both $x$ and $s_{h_{t-1}}$. 
This gives back to the previous case when $\pbar(s_{h_t})=\alpha+1$ and $s_{h_t}\ne x$
with $s_{h_{t-1}}$ in place of $s_{h_t}$.

Thus we assume that $t=1$. Let $\varphi_1=\varphi/P_x(1,\alpha+1)$. Then $P_r(1,\tau_{1},\varphi_1)=rx$. We next let  $\varphi_2=\varphi_1/Q$,  where $Q$ is the $(\tau,1)$-chain containing $us_1$ under $\varphi_1$. Then $\varphi_2$ is $F$-stable with respect to $\varphi$, but  $P_{s_1}(1,\Delta,\varphi_2)$ ends at $x$, giving a contradiction to Lemma~\ref{thm:vizing-fan1}~\eqref{thm:vizing-fan1b}.
%
	\medskip 

\emph{\noindent Case 2: $u\sim s_\alpha$ and $\varphi(rs_\alpha)$ is not a $\Delta$-inducing color}.  

Let $\varphi(us_\alpha)=\tau$. 	Note that $\tau\ne \alpha, \alpha+1, \Delta$. By Lemma~\ref{thm:vizing-fan2}~\eqref{thm:vizing-fan2-a}, $P_{s_1}(\Delta,\alpha+1)=P_{s_\alpha}(\Delta,\alpha+1)$. Since $r\in P_x(\Delta,\alpha+1)$, we have $r\notin P_{s_1}(\Delta,\alpha+1)$. Now let $\varphi'=\varphi/P_{s_1}(\Delta,\alpha+1)$ and let $\varphi^*$ be obtained from $\varphi'$ by uncoloring $rs_\alpha$, shifting from $s_2$ to $s_{\alpha-1}$ and coloring $rs_1$ by 2. Then $F^*=(r,rs_{\alpha},s_{\alpha},rs_{\alpha-1},s_{\alpha-1},\ldots,s_1,$ $rs_{\alpha+1},s_{\alpha+1},\ldots,s_\beta)$ is a typical multifan centered at $r$ with respect to $rs_\alpha$ and $\varphi^*$, where $\pbar^*(s_\alpha)=\{\alpha,\Delta\}$, $\varphi^*(ru)=\pbar^*(x)=\alpha+1$ is the last $\alpha$-inducing color, $\varphi^*(ux)=\Delta$, and $u\sim s_\alpha$ and $\tau$ is not $\Delta$-inducing. This gives back to \emph{Case 1},  finishing  the proof of Lemma~\ref{Lem:2-non-adj1*}.  \end{proof}

\begin{LEM}\label{Lem:2-non-adj2}
Let  $G$ be an HZ-graph with maximum degree $\Delta\ge 3$, $r\in V_\Delta$,  $N_{\Delta-1}(r)=\{s_1, s_2,\ldots, s_{\Delta-2}\}$, and $\varphi\in \CC^\Delta(G-rs_1)$, and let 
$F:=F_\varphi(r,s_1:s_\alpha)$ be a typical 2-inducing  multifan and  $L:=(F,ru,u,ux,x)$ be a lollipop centered at $r$.  If $\varphi(ru)=\alpha+1$, $\pbar(x)=\alpha+1$, 
 and $\varphi(ux)=\mu\in \pbar(V(F))$ is a 2-inducing color, 
then  $u\not\sim s_{\mu-1}$ and $u\not\sim s_\mu$.  
\end{LEM}

\begin{proof}
	Assume to the contrary that $u\sim s_{\mu-1}$ or $u\sim s_\mu$. We consider two cases below. 
	
	\medskip 
	
	\emph{\noindent Case 1: $u\sim s_{\mu-1}$}.  
		Let $\varphi(us_{\mu-1})=\tau$. 	Note that $\tau\ne \mu-1, \mu,  \alpha+1$.

	\begin{CLA}\label{not1_3}
		For every  $L$-stable $\varphi^*\in \CC^\Delta(G-rs_1)$, it holds that	$\varphi^*(us_{\mu-1})\ne 1$.  Furthermore, if  
		  $\varphi^*(us_{\mu-1})=\varphi(us_{\mu-1})=\tau$,  then  $us_{\mu-1}\in P_r(\tau,1, \varphi^*)$.  
	\end{CLA}
	\proof[Proof of Claim~\ref{not1_3}] 
Suppose instead  that $\varphi^*(us_{\mu-1})=1$ for the first part,  
	and $us_{\mu-1}\notin P_r(\tau,1, \varphi^*)$  for the second part. 
Let $\varphi'=\varphi^*$ in the former case and let $\varphi'=\varphi^*/Q$ in the latter case, where $Q$ is the $(\tau,1)$-chain containing $us_{\mu-1}$. Clearly $\varphi'$ is $F$-stable with respect to $\varphi$. Since $P_r(\alpha+1,1,\varphi')= P_{s_\alpha}(\alpha+1,1,\varphi')$ by Lemma~\ref{thm:vizing-fan1}~\eqref{thm:vizing-fan1b}, $P_x(\alpha+1,1,\varphi')$ does not contain $r$. Thus $\varphi''=\varphi'/P_x(\alpha+1,1,\varphi')$ is $F$-stable with respect to $\varphi^*$ and $\varphi''(us_{\mu-1})=\varphi'(us_{\mu-1})=1$.  However  $P_{s_{\mu-1}}(\mu,1,\varphi'')=s_{\mu-1}ux$, contradicting Lemma~\ref{thm:vizing-fan1}~\eqref{thm:vizing-fan1b} that $s_{\mu-1}$ and $r$ are $(\mu,1)$-linked.
	\qed 
	
\smallskip 
	
	\emph{\noindent Subcase 1.1: $\tau\in \pbar(V(F))$}.  
A precoloring of $L$ in  this case is depicted in Figure~\ref{f7}. 
	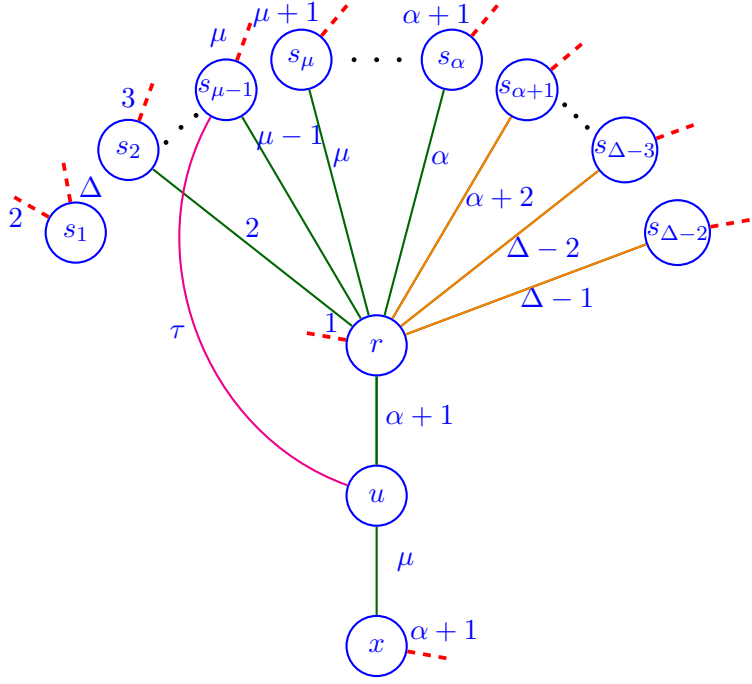
\begin{figure}[!htb]
		\begin{center}
			\begin{tikzpicture}[scale=1]
			
			{\tikzstyle{every node}=[draw ,circle,fill=white, minimum size=0.8cm,
				inner sep=0pt]
				\draw[blue,thick](0,-3) node (r)  {$r$};
				\draw [blue,thick](-4, -1.5) node (s1)  {$s_1$};
				\draw [blue,thick](-3.3, -0.4) node (s2)  {$s_2$};
				\draw[blue,thick] (-2, 0.4) node (sa)  {$s_{\mu-1}$};
				\draw [blue,thick](-1, 0.8) node (sa2)  {$s_{\mu}$};
				\draw [blue,thick](1, 0.8) node (sb)  {$s_{\alpha}$};
				\draw [blue,thick](2, 0.4) node (sb2)  {$s_{\alpha+1}$};
				\draw[blue,thick] (3.3, -0.4) node (sd1)  {$s_{\Delta-3}$};
				\draw [blue,thick](4, -1.5) node (sd2)  {$s_{\Delta-2}$};
				\draw [blue,thick](0, -5) node (u)  {$u$};
				\draw [blue,thick](0, -7) node (x)  {$x$};
			}
			\path[draw,thick,black!60!green]
			(r) edge node[name=la,above,pos=0.5] {\color{blue}$2$} (s2)
			(r) edge node[name=la,pos=0.9] {\color{blue}\quad\quad \,\,\,$\mu-1$} (sa)
			(r) edge node[name=la,pos=0.7] {\color{blue}\quad$\mu$} (sa2)
			(r) edge node[name=la,pos=0.7] {\color{blue}\quad$\alpha$} (sb)
			(r) edge node[name=la,pos=0.6] {\color{blue}\qquad\,\,\,$\alpha+2$} (sb2)
			(r) edge node[name=la,pos=0.45] {\color{blue}\qquad\quad$\alpha+1$} (u)
			(r) edge node[name=la,pos=0.5] {\color{blue}\quad \qquad$\Delta-2$} (sd1)
			(r) edge node[name=la,pos=0.4] {\color{blue}\qquad\quad\,\,\,\,\,$\Delta-1$} (sd2)
			(r) edge node[name=la,pos=0.4] {\color{blue}\qquad\quad\,\,\,} (u)
			(u) edge node[name=la,pos=0.4] {\color{blue}\qquad$\mu$} (x);
			\path[draw,thick,magenta]
			(u) edge [bend left=50] node [left,name=la,pos=0.5] {\color{blue}\qquad$\tau$} (sa);

			\draw[orange, thick] (r) --(sb2); 
			\draw [orange, thick](r) --(sd1); 
			\draw [orange, thick](r) --(sd2); 
			
			\draw[dashed, red, line width=0.5mm] (r)--++(170:1cm); 
			\draw[dashed, red, line width=0.5mm] (s1)--++(150:1cm); 
			\draw[dashed, red, line width=0.5mm] (s1)--++(100:1cm); 
			\draw[dashed, red, line width=0.5mm] (s2)--++(70:1cm); 
			\draw[dashed, red, line width=0.5mm] (sa)--++(70:1cm); 
			\draw[dashed, red, line width=0.5mm] (sa2)--++(50:1cm); 
			\draw[dashed, red, line width=0.5mm] (sb)--++(50:1cm); 
			\draw[dashed, red, line width=0.5mm] (sb2)--++(40:1cm); 
			\draw[dashed, red, line width=0.5mm] (sd1)--++(20:1cm); 
			\draw[dashed, red, line width=0.5mm] (sd2)--++(10:1cm); 
			\draw[dashed, red, line width=0.5mm] (x)--++(350:1cm); 
			\draw[blue] (-0.6, -2.7) node {$1$};  
			\draw[blue] (-4.8, -1.3) node {$2$};  
			\draw[blue] (-3.8, -0.9) node {$\Delta$};  
			\draw[blue] (-3.3, 0.3) node {$3$};  
			\draw[blue] (-2.1, 1.1) node {$\mu$};  
			\draw[blue] (-1.2, 1.4) node {$\mu+1$};  
			\draw[blue] (0.8, 1.4) node {$\alpha+1$}; 
			\draw[blue] (0.9, -6.8) node {$\alpha+1$}; 
			{\tikzstyle{every node}=[draw ,circle,fill=black, minimum size=0.05cm,
				inner sep=0pt]
				\draw(-2.4,0.06) node (f1)  {};
				\draw(-2.6,-0.1) node (f1)  {};
				\draw(-2.8,-0.3) node (f1)  {};
				\draw(-0.3,0.8) node (f1)  {};
				\draw(0,0.8) node (f1)  {};
				\draw(0.3,0.8) node (f1)  {};
				\draw(2.5,0.2) node (f1)  {};
				\draw(2.65,0.05) node (f1)  {};
				\draw(2.8
				,-0.1) node (f1)  {};
			} 
			
			\end{tikzpicture}
		\end{center}
		\caption{Precoloring of $L$ in Subcase 1.1 of Lemma~\ref{Lem:2-non-adj2}}
		\label{f7}
	\end{figure}


	\medskip 
	
	\emph{\noindent Subcase 1.1.1: $\tau \prec \mu$}.  
	
	\smallskip 
	
	By Lemma~\ref{Lemma:extended multifan} \eqref{Evizingfan-e},  $r\in P_{s_\alpha}(\alpha+1,\tau)=P_{s_{\tau-1}}(\alpha+1,\tau)$.  
	Let $\varphi'=\varphi/P_x(\alpha+1,\tau)$. 
	Then $P_x(\tau,\mu,\varphi')=xus_{\mu-1}$. Swapping colors 
	along $P_x(\tau,\mu,\varphi')=xus_{\mu-1}$ to get a new coloring $\varphi''$. Then both $s_{\tau-1}$
	and $s_{\mu-1}$ miss $\tau$ with respect to $\varphi''$, giving a contradiction
	to Lemma~\ref{thm:vizing-fan1}~\eqref{thm:vizing-fan1a} that $V(F_{\varphi''}(r,s_1: s_{\mu-1}))$ is $\varphi''$-elementary.
	
	\medskip 
	
	\emph{\noindent Subcase 1.1.2: $\tau=\Delta$}.  
	
	\smallskip 
	
	By Lemma~\ref{Lemma:extended multifan} \eqref{Evizingfan-d},  $r\in P_{s_\alpha}(\alpha+1,\Delta)=P_{s_1}(\alpha+1,\Delta)$.  
	Let $\varphi'=\varphi/P_x(\alpha+1,\Delta)$. 
	Then $P_x(\Delta,\mu,\varphi')=xus_{\mu-1}$. Swapping colors 
	along $P_x(\Delta,\mu,\varphi')=xus_{\mu-1}$ to get a new coloring $\varphi''$. Then  both $s_{\Delta-1}=s_1$
	and $s_{\mu-1}$ miss $\Delta$ with respect to $\varphi''$, giving a contradiction
	to Lemma~\ref{thm:vizing-fan1}~\eqref{thm:vizing-fan1a} that $V(F_{\varphi''}(r,s_1: s_{\mu-1}))$ is $\varphi''$-elementary. 
	
	\medskip 
	
	\emph{\noindent Subcase 1.1.3: $\mu \prec \tau \prec \alpha+1$}.  
		By Claim~\ref{not1_3}, $us_{\mu-1}\in P_r(\tau,1)=P_{s_{\tau-1}}(\tau,1)$. Let $P_u(\tau,1)$ be the subpath of $P_r(\tau,1)$ starting at $u$ not containing $us_{\mu-1}$. Then $P_u(\tau,1)$ ends at  either $r$ or $s_{\tau-1}$. Consequently if we shift from $s_\tau$ to $s_\alpha$, then $P_u(\tau,1)$ will end at either $s_\tau$ or $s_{\tau-1}$. Thus we can do the following operations:
\[
\begin{bmatrix}
s_{\tau}:s_{\alpha} & P_{u}(\tau,1)  & us_{\mu} & ux & ur\\
\text{shift} & \tau/1 & \tau \rightarrow \mu & \mu \rightarrow \alpha+1 & \alpha+1 \rightarrow 1
\end{bmatrix}.
\]

Denote the new coloring by $\varphi'$. Now $(r,rs_1,s_1,\ldots,s_{\mu-1})$ is a multifan, but $\pbar'(s_{\mu-1})=\pbar'(r)=\{\tau\}$, giving a contradiction to Lemma~\ref{thm:vizing-fan1}~\eqref{thm:vizing-fan1a}.
	
	\medskip

	\emph{\noindent Subcase 1.2: $\tau\notin \pbar(V(F))$}.  Thus, $\tau\in \{\alpha+2,\ldots, \Delta-1\}$.  A precoloring for $L$ in this case is depicted in Figure~\ref{f8}. 
	
	\smallskip


	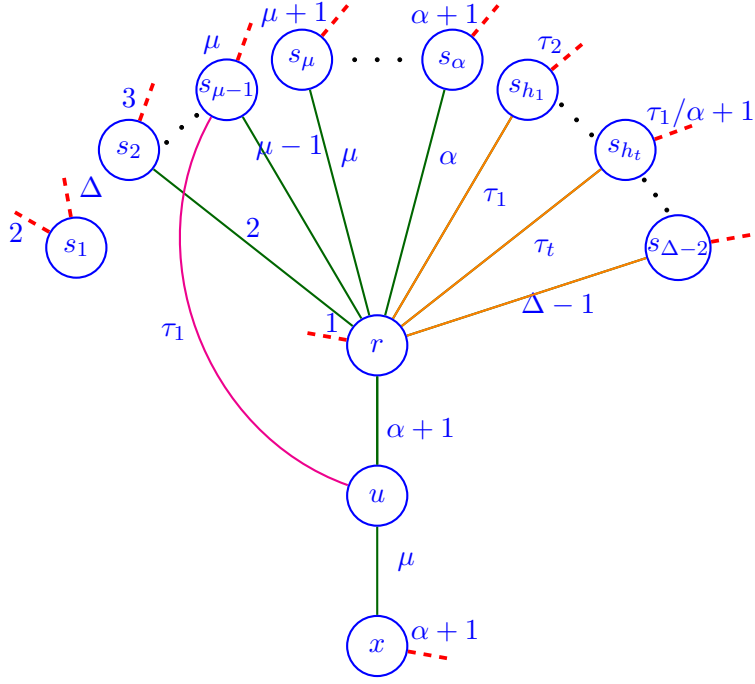
\begin{figure}[!htb]
		\begin{center}
			\begin{tikzpicture}[scale=1]
			
			{\tikzstyle{every node}=[draw ,circle,fill=white, minimum size=0.8cm,
				inner sep=0pt]
				\draw[blue,thick](0,-3) node (r)  {$r$};
				\draw [blue,thick](-4, -1.7) node (s1)  {$s_1$};
				\draw [blue,thick](-3.3, -0.4) node (s2)  {$s_2$};
				\draw[blue,thick] (-2, 0.4) node (sa)  {$s_{\mu-1}$};
				\draw [blue,thick](-1, 0.8) node (sa2)  {$s_{\mu}$};
				\draw [blue,thick](1, 0.8) node (sb)  {$s_{\alpha}$};
				\draw [blue,thick](2, 0.4) node (sb2)  {$s_{h_1}$};
				\draw[blue,thick] (3.3, -0.4) node (sd1)  {$s_{h_t}$};
				\draw [blue,thick](4, -1.7) node (sd2)  {$s_{\Delta-2}$};
				\draw [blue,thick](0, -5) node (u)  {$u$};
				\draw [blue,thick](0, -7) node (x)  {$x$};
			}
			\path[draw,thick,black!60!green]
			(r) edge node[name=la,above,pos=0.5] {\color{blue}$2$} (s2)
			(r) edge node[name=la,pos=0.85] {\color{blue}\quad\quad$\mu-1$} (sa)
			(r) edge node[name=la,pos=0.7] {\color{blue}\quad\,\,\,$\mu$} (sa2)
			(r) edge node[name=la,pos=0.7] {\color{blue}\quad\,\,\,$\alpha$} (sb)
			(r) edge node[name=la,pos=0.6] {\color{blue}\qquad\,$\tau_1$} (sb2)
			(r) edge node[name=la,pos=0.6] {\color{blue}\qquad\quad$\alpha+1$} (u)
			(r) edge node[name=la,pos=0.5] {\color{blue}\quad \qquad$\tau_t$} (sd1)
			(r) edge node[name=la,pos=0.4] {\color{blue}\qquad\quad\,\,\,\,\,$\Delta-1$} (sd2)
			(r) edge node[name=la,pos=0.4] {\color{blue}\qquad\quad\,\,\,} (u)
			(u) edge node[name=la,pos=0.4] {\color{blue}\qquad$\mu$} (x);
			\path[draw,thick,magenta]
			(u) edge [bend left=50] node [left,name=la,pos=0.5] {\color{blue}\qquad$\tau_1$} (sa);
			
			%
			
			\draw[orange, thick] (r) --(sb2); 
			\draw [orange, thick](r) --(sd1); 
			\draw [orange, thick](r) --(sd2); 
			
			\draw[dashed, red, line width=0.5mm] (r)--++(170:1cm); 
			\draw[dashed, red, line width=0.5mm] (s1)--++(150:1cm); 
			\draw[dashed, red, line width=0.5mm] (s1)--++(100:1cm); 
			\draw[dashed, red, line width=0.5mm] (s2)--++(70:1cm); 
			\draw[dashed, red, line width=0.5mm] (sa)--++(70:1cm); 
			\draw[dashed, red, line width=0.5mm] (sa2)--++(50:1cm); 
			\draw[dashed, red, line width=0.5mm] (sb)--++(50:1cm); 
			\draw[dashed, red, line width=0.5mm] (sb2)--++(40:1cm); 
			\draw[dashed, red, line width=0.5mm] (sd1)--++(20:1cm); 
			\draw[dashed, red, line width=0.5mm] (sd2)--++(10:1cm); 
			\draw[dashed, red, line width=0.5mm] (x)--++(350:1cm); 
			\draw[blue] (-0.6, -2.7) node {$1$};  
			\draw[blue] (-4.8, -1.5) node {$2$};  
			\draw[blue] (-3.8, -0.9) node {$\Delta$};  
			\draw[blue] (-3.3, 0.3) node {$3$};  
			\draw[blue] (-2.2, 1) node {$\mu$};  
			\draw[blue] (-1.1, 1.4) node {$\mu+1$};  
			\draw[blue] (0.9, 1.4) node {$\alpha+1$}; 
			\draw[blue] (0.9, -6.8) node {$\alpha+1$}; 
			\draw[blue] (2.3, 1.0) node {$\tau_2$}; 
			\draw[blue] (4.3, 0.1) node {$\tau_1/\alpha+1$}; 
			

			
			{\tikzstyle{every node}=[draw ,circle,fill=black, minimum size=0.05cm,
				inner sep=0pt]
				\draw(-2.4,0.06) node (f1)  {};
				\draw(-2.6,-0.1) node (f1)  {};
				\draw(-2.8,-0.3) node (f1)  {};
				\draw(-0.3,0.8) node (f1)  {};
				\draw(0,0.8) node (f1)  {};
				\draw(0.3,0.8) node (f1)  {};
				\draw(2.45,0.2) node (f1)  {};
				\draw(2.65,0) node (f1)  {};
				\draw(2.85,-0.2) node (f1)  {};
				\draw(3.55,-0.8) node (f1)  {};
				\draw(3.7,-1) node (f1)  {};
				\draw(3.85,-1.2) node (f1)  {};
			} 
			\end{tikzpicture}
		\end{center}
		\caption{Precoloring of $L$ in Subcase 1.2 of Lemma~\ref{Lem:2-non-adj2}}
		\label{f8}
	\end{figure}

	
	Let $\tau_1=\tau$ and $s_{h_1}\in N_{\Delta-1}(r)$ such that $\varphi(rs_{h_1})=\tau_1$. 
	Since $us_{\mu-1}\in P_r(\tau,1, \varphi) $  and  the color on $us_{\mu-1}$ 
	is not 1 under  every $L$-stable coloring by Claim~\ref{not1_3}, 
	for every  $L$-stable $\varphi' \in \CC^{\Delta}(G-rs_1)$  obtained from $\varphi$ 
	through a sequence of  Kempe $(1,*)$-changes, it holds that 
	$\varphi'(us_{\mu-1})=\varphi(us_{\mu-1})$.   Therefore, 
	$u\in P_r(\tau,1, \varphi')$  by  Claim~\ref{not1_3} again. 
	Applying  Lemma~\ref{Lemma:pseudo-fan0} on $L$ with $u$ playing the role of $w$, 
		there exists a sequence of distinct vertices $s_{h_1}, \ldots, s_{h_t}\in  \{s_{\alpha+1}, \ldots, s_{\Delta-2} \}$ satisfying the following conditions: 
\begin{enumerate}[(a)]
	\item  $\varphi(rs_{h_{i+1}})=\pbar(s_{h_i})=\tau_{i+1}\in  \{\alpha+2, \cdots, \Delta-1\} $ for each $i\in [1,t-1]$; 
	\item $s_{h_i}$ and $r$ are $(\tau_{i+1}, 1)$-linked with respect to $\varphi$ for each $i\in [1,t-1]$;
	\item $\pbar(s_{h_t})=\tau_1$ or $\alpha+1$, and if $\pbar(s_{h_t})=\tau_1$, then  $s_{h_t}$ and $r$ are $(\tau_{1}, 1)$-linked with respect to $\varphi$. 
\end{enumerate} 
	\medskip 
	
		\emph{\noindent \setword{Subcase 1.2.1}{Subcase 1.2.1}: $\pbar(s_{h_t})=\alpha+1$.}
	
	\smallskip
	In this case, $t\ge 1$. By Claim~\ref{not1_3},  $us_{\mu-1}\in P_r(\tau,1)=P_{z}(\tau,1)$,   for some vertex
	$z\in V(G)\setminus (V(F)\cup \{x,s_{h_1}, \ldots, s_{h_t}\})$. \emph{Assume first that $s_{h_t}\ne x$}. 
	
	{\it If  $P_z(\tau,1)$ meets $u$ before $s_{\mu-1}$}, we do the following operations:
\[
\begin{bmatrix}
P_{[z,u]}(\tau,1)  & us_{\mu} & ux & ur\\
\tau/1 & \tau \rightarrow \mu & \mu \rightarrow \alpha+1 & \alpha+1 \rightarrow 1
\end{bmatrix}.
\]
Denote the new coloring by $\varphi'$. Now $(r,rs_1,s_1,\ldots,s_{\mu-1},rs_{h_1},s_{h_1},\ldots,s_{h_t})$ is a multifan, but $\pbar'(s_{h_t})=\pbar'(r)=\{\alpha+1\}$, giving a contradiction to Lemma~\ref{thm:vizing-fan1}~\eqref{thm:vizing-fan1a}.
	
%
	
		{\it If $P_z(\tau,1)$ meets $s_{\mu-1}$ before $u$}, we do the following operations:
\[
\begin{bmatrix}
P_{[z,s_{\mu-1}]}(\tau,1) & s_{h_1}:s_{h_t}  & us_{\mu} & ux & ur\\
\tau/1 & \text{shift} & \tau \rightarrow \mu & \mu \rightarrow \alpha+1 & \alpha+1 \rightarrow \tau
\end{bmatrix}.
\]
Denote the new coloring by $\varphi'$. Now $(r,rs_1,s_1,\ldots,s_{\mu-1})$ is a multifan, but $\pbar'(s_{\mu-1})=\pbar'(r)=\{1\}$, giving a contradiction to  Lemma~\ref{thm:vizing-fan1}~\eqref{thm:vizing-fan1a}.
	
%
	
	\emph{Assume now that $s_{h_t}= x$}. We first consider the case when $t\ge 2$. Note that $P_r(1,\alpha+1)=P_{s_\alpha}(1,\alpha+1)$ and so $r\notin P_x(1,\alpha+1)$. Let  $\varphi_1=\varphi/P_x(1,\alpha+1)$. Then $P_r(1,\tau_{t},\varphi_1)=rx$. We next let  $\varphi_2=\varphi_1/P_{s_{h_{t-1}}}(\tau_t,1,\varphi_1)$. Then $\varphi_2$ is $F$-stable with respect to $\varphi$ and $\pbar_2(x)=\pbar_2(s_{h_{t-1}})=1$. Now do $(\alpha+1,1)$-swap at both $x$ and $s_{h_{t-1}}$. 
This gives back to the previous case when $\pbar(s_{h_t})=\alpha+1$ and $s_{h_t}\ne x$
with $s_{h_{t-1}}$ in place of $s_{h_t}$.

Thus we assume that $t=1$. Let $\varphi_1=\varphi/P_x(1,\alpha+1)$. Then $P_r(1,\tau_{1},\varphi_1)=rx$. We next let  $\varphi_2=\varphi_1/Q$, where $Q$ is the $(\tau,1)$-chain containing $us_{\mu-1}$ under $\varphi_1$. Then $\varphi_2$ is $F$-stable with respect to $\varphi$, but  $P_{s_{\mu-1}}(1,\Delta,\varphi_2)$ ends at $x$, giving a contradiction to Lemma~\ref{thm:vizing-fan1}~\eqref{thm:vizing-fan1b}.

\medskip

	\emph{\noindent Subcase 1.2.2: $\pbar(s_{h_t})=\tau_1$.}
		In this case,  $t\ge 2$.  By Claim~\ref{not1_3},  $us_{\mu-1}\in P_r(\tau,1)=P_{s_{h_t}}(\tau,1)$. 
		
	\emph{Suppose first that $r\notin P_{s_\alpha}(\tau_1, \alpha+1)$.}	
		Let $\varphi'=\varphi/P_{s_\alpha}(\tau_1,\alpha+1)$. Note that $F'=(r,rs_1,s_1,\ldots,s_\alpha,rs_{h_1},s_{h_1},\ldots,s_{h_t})$ is a multifan under $\varphi'$. If the other end of $P_{s_\alpha}(\tau_1,\alpha+1,\varphi)$ is not $s_{h_t}$, then $\pbar'(s_\alpha)=\pbar'(s_{h_t})=\tau_1$, giving a contradiction Lemma~\ref{thm:vizing-fan1} \eqref{thm:vizing-fan1a} . If the other end of $P_{s_\alpha}(\tau_1,\alpha+1,\varphi)$ is $s_{h_t}$, then $\pbar'(s_{h_t})=\varphi'(ru)=\pbar'(x)=\alpha+1$. Note that $\tau_1$ is in $\pbar'(V(F'))$ now, we are back to subcase 1.1.3.
		
	\smallskip
	\emph{Assume  now that $r\in P_{s_\alpha}(\tau_1, \alpha+1)$}. We consider the following three  cases. 
	
	\emph{\noindent Subcase A: $s_\alpha$ and $s_{h_t}$ are $(\tau_1,\alpha+1)$-linked.} Let $\varphi'=\varphi/P_x(\tau_1,\alpha+1)$. Then in the new coloring, $P_{s_{\mu-1}}(\tau_1,\mu,\varphi')=s_{\mu-1}ux$. Let $\varphi''=\varphi'/P_{s_{\mu-1}}(\tau_1,\mu,\varphi')$. Then $F^*=(r,rs_1, s_1,\dots, rs_{\mu-1}, s_{\mu-1}, rs_{h_1}, s_{h_1}, \ldots, rs_{h_t}, s_{h_t})$ is a multifan with respect to $\varphi''$. 
		However, $\pbar''(s_{\mu-1})=\pbar''(s_{h_t})=\tau_1$, showing a contradiction to Lemma~\ref{thm:vizing-fan1} \eqref{thm:vizing-fan1a} that  $V(F^*)$
		is $\varphi''$-elementary.
		
	\emph{\noindent \setword{Subcase B}{Subcase 1.2.2.3} : $s_\alpha$ and $s_{h_t}$ are $(\tau_1,\alpha+1)$-unlinked, but 
		$s_\alpha$ and $x$ are $(\tau_1,\alpha+1)$-linked.} 
	Recall  $r\in P_{s_\alpha}(\tau_1, \alpha+1)$. 
		Let $\varphi'=\varphi/P_{s_{h_t}}(\alpha+1,\tau_1)$. This reduces the problem to \ref{Subcase 1.2.1}. 
	\medskip

\emph{\noindent Subcase C: $s_\alpha$ is   $(\tau_1,\alpha+1)$-unlinked with both $s_{h_t}$ and $x$.}
We may assume that $x$ and $s_{h_t}$ are $(\alpha+1, \tau_1)$-linked. For otherwise, 
let $\varphi'=\varphi/P_{s_{h_t}}(\alpha+1,\tau_1)$, we are back to \ref{Subcase 1.2.1} again.

Recall that $r\in P_{s_\alpha}(\alpha+1,\tau_1)$. If $P_{s_\alpha}(\alpha+1, \tau_1)$ meets $s_{h_1}$ before $s_{\mu-1}$, we shift  from $s_{h_1}$ to $s_{h_t}$. This yields a new coloring $\varphi'$ such that 
$r\not\in P_{s_\alpha}(\alpha+1,\tau_1,\varphi') $, and so we are back to the first subcase of Subcase 1.2.2. If $P_{s_\alpha}(\alpha+1, \tau_1)$ meets $s_{\mu-1}$ before $s_{h_1}$, then shifting  from $s_{h_1}$ to $s_{h_t}$ yields a new coloring $\varphi'$ such that
$s_\alpha$ and $x$ are $(\alpha+1,\tau)$-linked with respect to $\varphi'$, which reduces the problem to 
\ref{Subcase 1.2.2.3}.  

	\medskip 
	
	\emph{\noindent Case 2: $u\sim s_\mu$}.  
		Let $\varphi(us_\mu)=\tau$. 	Note that $\tau\ne \mu, \mu+1, \alpha+1$. A precoloring of $L$ in  this case is depicted in Figure~\ref{f9}.

\begin{CLA}\label{cla3.3}
    Either $\tau=\Delta$ or $\tau$ is a 2-inducing color with $\tau\prec \mu$.
\end{CLA}
	\proof[Proof of Claim~\ref{cla3.3}] 
	Let $\varphi'$ be the coloring obtained from $\varphi$ by uncoloring $rs_{\mu}$, shifting from $s_2$ to $s_{\mu-1}$ and coloring $rs_1$ by 2. Then $F'=(r,rs_{\mu},s_{\mu},rs_{\mu+1},s_{\mu+1},\ldots,s_{\alpha},rs_{\mu-1},s_{\mu-1},\ldots,s_1)$ is a multifan under $\varphi'$,  where  $\pbar'(s_{\mu})=\{\mu,\mu+1\}$, $\varphi'(ru)=\pbar'(ux)=\alpha+1$ is the last $(\mu+1)$-inducing color, and $\varphi'(ux)=\mu$ and $u \sim s_{\mu}$. Thus by Lemma~\ref{Lem:2-non-adj1*}, $\tau$ is a $\mu$-inducing color with respect to $\varphi'$ and $F'$. So with respect to the original coloring $\varphi$ and $F$, we have either $\tau=\Delta$ or $\tau$ is a 2-inducing color with $\tau\prec \mu$. \qed

\emph{\noindent Subcase 2.1: $\tau$ is a 2-inducing color with $\tau\prec \mu$}.  

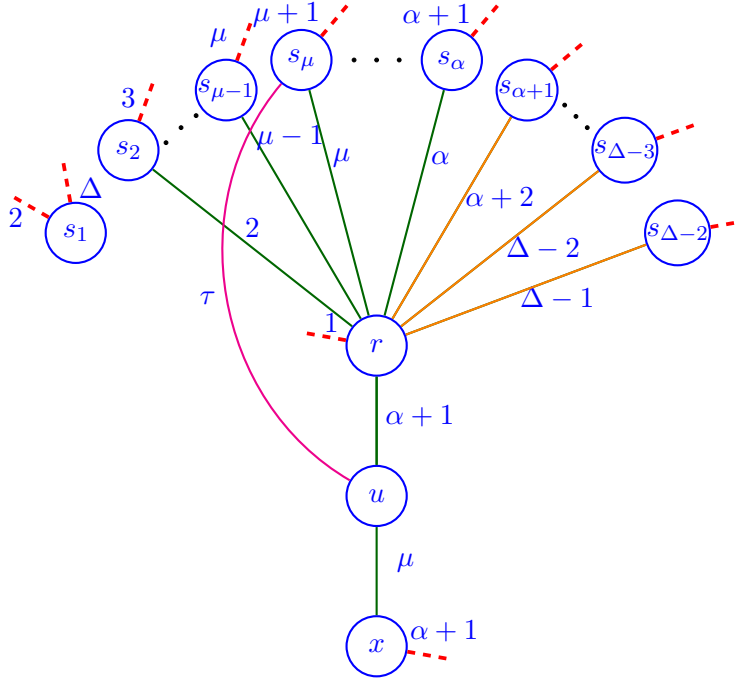
\begin{figure}[!htb]
	\begin{center}
		\begin{tikzpicture}[scale=1]
		
		{\tikzstyle{every node}=[draw ,circle,fill=white, minimum size=0.8cm,
			inner sep=0pt]
			\draw[blue,thick](0,-3) node (r)  {$r$};
			\draw [blue,thick](-4, -1.5) node (s1)  {$s_1$};
			\draw [blue,thick](-3.3, -0.4) node (s2)  {$s_2$};
			\draw[blue,thick] (-2, 0.4) node (sa)  {$s_{\mu-1}$};
			\draw [blue,thick](-1, 0.8) node (sa2)  {$s_{\mu}$};
			\draw [blue,thick](1, 0.8) node (sb)  {$s_{\alpha}$};
			\draw [blue,thick](2, 0.4) node (sb2)  {$s_{\alpha+1}$};
			\draw[blue,thick] (3.3, -0.4) node (sd1)  {$s_{\Delta-3}$};
			\draw [blue,thick](4, -1.5) node (sd2)  {$s_{\Delta-2}$};
			\draw [blue,thick](0, -5) node (u)  {$u$};
			\draw [blue,thick](0, -7) node (x)  {$x$};
		}
		\path[draw,thick,black!60!green]
		(r) edge node[name=la,above,pos=0.5] {\color{blue}$2$} (s2)
		(r) edge node[name=la,pos=0.9] {\color{blue}\quad\quad \,\,\,$\mu-1$} (sa)
		(r) edge node[name=la,pos=0.7] {\color{blue}\quad$\mu$} (sa2)
		(r) edge node[name=la,pos=0.7] {\color{blue}\quad$\alpha$} (sb)
		(r) edge node[name=la,pos=0.6] {\color{blue}\qquad\,\,\,$\alpha+2$} (sb2)
		(r) edge node[name=la,pos=0.45] {\color{blue}\qquad\quad$\alpha+1$} (u)
		(r) edge node[name=la,pos=0.5] {\color{blue}\quad \qquad$\Delta-2$} (sd1)
		(r) edge node[name=la,pos=0.4] {\color{blue}\qquad\quad\,\,\,\,\,$\Delta-1$} (sd2)
		(r) edge node[name=la,pos=0.4] {\color{blue}\qquad\quad\,\,\,} (u)
		(u) edge node[name=la,pos=0.4] {\color{blue}\qquad$\mu$} (x);
		\path[draw,thick,magenta]
		(u) edge [bend left=50] node [left,name=la,pos=0.5] {\color{blue}\qquad$\tau$} (sa2);

		\draw[orange, thick] (r) --(sb2); 
		\draw [orange, thick](r) --(sd1); 
		\draw [orange, thick](r) --(sd2); 
		
		\draw[dashed, red, line width=0.5mm] (r)--++(170:1cm); 
		\draw[dashed, red, line width=0.5mm] (s1)--++(150:1cm); 
		\draw[dashed, red, line width=0.5mm] (s1)--++(100:1cm); 
		\draw[dashed, red, line width=0.5mm] (s2)--++(70:1cm); 
		\draw[dashed, red, line width=0.5mm] (sa)--++(70:1cm); 
		\draw[dashed, red, line width=0.5mm] (sa2)--++(50:1cm); 
		\draw[dashed, red, line width=0.5mm] (sb)--++(50:1cm); 
		\draw[dashed, red, line width=0.5mm] (sb2)--++(40:1cm); 
		\draw[dashed, red, line width=0.5mm] (sd1)--++(20:1cm); 
		\draw[dashed, red, line width=0.5mm] (sd2)--++(10:1cm); 
		\draw[dashed, red, line width=0.5mm] (x)--++(350:1cm); 
		\draw[blue] (-0.6, -2.7) node {$1$};  
		\draw[blue] (-4.8, -1.3) node {$2$};  
		\draw[blue] (-3.8, -0.9) node {$\Delta$};  
		\draw[blue] (-3.3, 0.3) node {$3$};  
		\draw[blue] (-2.1, 1.1) node {$\mu$};  
		\draw[blue] (-1.2, 1.4) node {$\mu+1$};  
		\draw[blue] (0.8, 1.4) node {$\alpha+1$}; 
		\draw[blue] (0.9, -6.8) node {$\alpha+1$}; 
		{\tikzstyle{every node}=[draw ,circle,fill=black, minimum size=0.05cm,
			inner sep=0pt]
			\draw(-2.4,0.06) node (f1)  {};
			\draw(-2.6,-0.1) node (f1)  {};
			\draw(-2.8,-0.3) node (f1)  {};
			\draw(-0.3,0.8) node (f1)  {};
			\draw(0,0.8) node (f1)  {};
			\draw(0.3,0.8) node (f1)  {};
			\draw(2.5,0.2) node (f1)  {};
			\draw(2.65,0.05) node (f1)  {};
			\draw(2.8
			,-0.1) node (f1)  {};
		} 
		
		\end{tikzpicture}
	\end{center}
	\caption{Precoloring of $L$ in Case 2.1 of Lemma~\ref{Lem:2-non-adj2}}
	\label{f9}
\end{figure}


\smallskip 

By Lemma~\ref{Lemma:extended multifan} \eqref{Evizingfan-e},  $r\in P_{s_\alpha}(\alpha+1,\tau)=P_{s_{\tau-1}}(\alpha+1,\tau)$.  
Let $\varphi'=\varphi/P_x(\alpha+1,\tau)$. 
Then $\pbar'(x)=\tau$. 
It must  be still the case that $u\in P_r(\tau,1, \varphi')=P_{s_{\tau-1}}(\tau,1, \varphi')$. 
For otherwise, swapping colors along $P_x(\tau,1,\varphi')$ and the $(\tau,1)$-chain containing $u$ (only swap once if the two chains are the same)
gives a coloring $\varphi''$ such that $P_r(\mu, 1, \varphi'')=rs_\mu ux$, 
showing a contradiction to Lemma~\ref{thm:vizing-fan1}~\eqref{thm:vizing-fan1b} 
that $r$ and $s_{\mu-1}$ are $(\mu,1)$-linked with respect to $\varphi''$. 
Let $\varphi^*=\varphi'/P_x(\tau,1,\varphi')$. Now $\pbar^*(x)=1$ and $u\in P_r(\tau,1,\varphi^*)$. We consider the following two cases. 

\medskip 

\emph{\noindent Subcase 2.1.1: $P_{s_{\tau-1}}(\tau, 1,\varphi^*)$ meets $u$ before $s_{\mu}$}.  We do the following operations:

\[
\begin{bmatrix}
s_{\tau}: s_{\mu-1}  & us_\mu r & ux & P_{[s_{\tau-1}, u]}(\tau,1,\varphi^*) \\
\text{shift} & \tau/\mu & \mu\rightarrow 1 & 1/\tau
\end{bmatrix}.
\]
Denote the new coloring by $\varphi''$. Now $(r,rs_1,s_1,\ldots,s_{\tau-1})$ is a multifan, but $\pbar''(s_{\tau-1})=\pbar''(r)=\{1\}$, giving a contradiction to  Lemma~\ref{thm:vizing-fan1}~\eqref{thm:vizing-fan1a}.


\medskip 

\emph{\noindent Subcase 2.1.2: $P_{s_{\tau-1}}(\tau, 1,\varphi^*)$ meets $s_\mu$ before $u$. }  
We find a contradiction in this case through the following steps. 
\begin{enumerate}[(1)]
	\item Swap colors along $P_{[s_{\tau-1}, s_\mu]}(\tau,1,\varphi^*)$ (now $1$ is missing at $s_{\tau-1}$ and $s_\mu$, the coloring is conflicting at $s_\mu$ with respect to $\tau$); \vspace{-2mm}
	\item $us_{\mu}: \tau\rightarrow \mu$, $ur: \alpha+1\rightarrow \tau$
	(now both 1 and $\alpha+1$ are missing at $r$, the coloring is conflicting at $r$ with respect to $\tau$ and is conflicting at $u$ with respect to $\mu$, but the conflict from step (1) is resolved); \vspace{-2mm}
	\item Do $rs_\mu: \mu \rightarrow \mu+1$. Note that the missing color at $s_{\tau-1}$ is 1. Shift from $s_2$ to $s_{\mu-1}$,
	and from $s_{\mu+1}$ to $s_\alpha$ (now 2 is missing at $r$, and the conflict from step (2) at $r$ is resolved); \vspace{-2mm}
	\item Color $rs_1$ by 2 if $\tau\ne 2$, and color $rs_1$ by 1 if $\tau=2$; \vspace{-2mm}
	\item  The coloring is now only conflicting at $u$ with respect to $\mu$.
	Uncolor $ux$. Now $F=(u,ux, x, us_\mu, s_\mu )$ 
	is a multifan with respect to $ux$ and the current coloring.
	However, the color 1 is missing at both $x$ and $s_\mu$,
	showing a contradiction to Lemma~\ref{thm:vizing-fan1} \eqref{thm:vizing-fan1a}
	that $V(F)$ is elementary with respect to $ux$ and the current coloring. 
\end{enumerate}

\medskip 

\emph{\noindent Subcase 2.2: $\tau=\Delta$}.  

By Lemma~\ref{Lemma:extended multifan} \eqref{Evizingfan-d}, we know that $r\in P_{s_\alpha}(\alpha+1,\Delta)=P_{s_1}(\alpha+1,\Delta)$.  
Let $\varphi'=\varphi/P_x(\alpha+1,\Delta)$. 
Then $\pbar'(x)=\Delta$. 
It must  be still the case that $u\in P_r(\Delta,1, \varphi')=P_{s_{1}}(\Delta,1, \varphi')$. 
For otherwise, swapping colors along $P_x(\Delta,1,\varphi')$ and the $(\Delta,1)$-chain containing $u$ (only swap once if the two chains are the same) gives a coloring $\varphi''$ such that $P_r(\mu, 1, \varphi'')=rs_\mu ux$, showing a contradiction to Lemma~\ref{thm:vizing-fan1}~\eqref{thm:vizing-fan1b} 
that $r$ and $s_{\mu-1}$ are $(\mu,1)$-linked with respect to $\varphi''$. 
\medskip 

\emph{\noindent Subcase 2.2.1: $P_{s_1}(\Delta, 1,\varphi')$ meets  $s_{\mu}$ before $u$}.  We do the following operations:
\[
\begin{bmatrix}
P_{[s_1, s_\mu]}(\Delta,1,\varphi') & rs_{\mu} & s_{\mu}ux\\
1/\Delta & \mu\rightarrow 1 & \Delta/\mu
\end{bmatrix}.
\]
Denote the new coloring by $\varphi''$. Now $(r,rs_1,s_1,\ldots,s_{\mu-1})$ is a multifan, but $\pbar''(s_{\mu-1})=\pbar''(r)=\{\mu\}$, giving a contradiction to Lemma~\ref{thm:vizing-fan1} \eqref{thm:vizing-fan1a}.

%
%

\medskip 

\emph{\noindent Subcase 2.2.2: $P_{s_1}(\Delta, 1,\varphi')$  meets $u$ before $s_\mu$}.  
We find a contradiction as below: 
\smallskip

\begin{enumerate}[(1)]
	\item Swap colors along $P_{[s_1, u]}(\Delta,1,\varphi')$ (now $1$ is missing at $s_{1}$ and $u$, the coloring is conflicting at $u$ with respect to $\Delta$); \vspace{-2mm}
	\item $us_{\mu}: \Delta\rightarrow \mu$, $ur: \alpha+1\rightarrow 1$ (now $\alpha+1$ is missing at $r$, the coloring is conflicting at $u$ and $s_{\mu}$ with respect to $\mu$, but the conflict from step (1) is resolved); \vspace{-2mm}
	\item  Do $rs_\mu: \mu \rightarrow \mu+1$.  Shift from $s_2$ to  $s_{\mu-1}$ and from $s_{\mu+1}$ to  $s_\alpha$ (now 2 is missing at $r$, and the conflict at $s_\mu$ from step (2) is resolved);  \vspace{-2mm}
	\item Color $rs_1$ by 2;  \vspace{-2mm}
	\item The coloring is now only conflicting at $u$ with respect to $\mu$.
	Uncolor $ux$. Now $F=(u, ux, x, us_\mu, s_\mu )$ 
	is a multifan with respect to $ux$ and the current coloring.
	However, the color $\Delta$ is missing at both $x$ and $s_\mu$,
	showing a contradiction to Lemma~\ref{thm:vizing-fan1} \eqref{thm:vizing-fan1a}
	that $V(F)$ is elementary with respect to $ux$ and the current coloring. 
\end{enumerate}
\medskip 
	This finishes the proof of Lemma~\ref{Lem:2-non-adj2}.
\end{proof}

\section{Proof of Theorem~\ref{Thm:vizing-fan}}\label{thm2.1proof}
We prove the following version of Theorem \ref{Thm:vizing-fan}.

\begin{THM}\label{Thm:vizing-fan2b}
	If $G$ is an HZ-graph with maximum degree $\Delta\ge 4$, then for every vertex $r\in V_{\Delta}$, the following two statements hold. 
	\begin{enumerate}[(i)]
		\item For every $u\in N_{\Delta}(r)$, 
		$N_{\Delta-1}(r)=N_{\Delta-1}(u)$.  \label{common2}
		\item There exist $s_1\in N_{\Delta-1}(r)$ 
		and  a coloring 
		$\varphi\in \CC^\Delta(G-rs_1)$ such that $N_{\Delta-1}[r]$ is the vertex set of either a typical 2-inducing multifan or a typical 2-inducing pseudo-multifan with respect to $rs_1$ and $\varphi$.   Consequently $N_{\Delta-1}[r]$ 
		is $\varphi$-elementary.  
		\label{ele2}
	\end{enumerate}
\end{THM}
\begin{proof}
	Let  $N_{\Delta-1}(r)=\{s_1,\ldots, s_{\Delta-2}\}$. 
	We choose a vertex in $N_{\Delta-1}(r)$, say $s_1$, a coloring $\varphi\in\CC^\Delta(G-rs_1)$ and a multifan $F$  with respect to $rs_1$
	and $\varphi$ such that $F$ is maximum at $r$. 
	That is, $|V(F)|$ is maximum among all multifans with respect to $rs_i$  and $\varphi'\in \CC^k(G-rs_i)$ for some $s_i\in N_{\Delta-1}(r)$.
	Assume that $\pbar(r)=1$ and $\pbar(s_1)=\{2,\Delta\}$, and $F=F_\varphi(r,s_1:s_p)$ 
	is such a multifan. By Lemma~\ref{2-inducing }, we may assume that 
	$F_\varphi(r,s_1:s_p)=F_\varphi(r,s_1:s_\alpha)$ is a typical  2-inducing 
	multifan, where $\alpha=p$.  
	
		Let $u\in N_{\Delta}(r)$. Roughly speaking,  the main proof idea is the following. 
	By assuming $\varphi(ru)=\alpha+1$ and $\pbar(x)=\alpha+1$
	for $x\in N_{\Delta-1}(u)\setminus N_{\Delta-1}(r)$, we will apply Lemmas~\ref{Lem:2-non-adj1} and \ref{Lem:2-non-adj2} to show that $u$ has at least two $(\Delta-1)$-neighbors outside
	of $N_{\Delta-1}(r)$. By further applying Lemmas~\ref{Lem:2-non-adj1} and \ref{Lem:2-non-adj2}, 
	we can even find three $(\Delta-1)$-neighbors of $u$ outside
	of $N_{\Delta-1}(r)	$.  A contradiction is then deduced at that point.


	\begin{CLA}\label{ux-color}
		Let $u\in N_{\Delta}(r)$. 	We may additionally assume that $\varphi(ru)=\alpha+1$, which is the last $2$-inducing color of 
		$F_\varphi(r,s_1:s_\alpha)$. 
	\end{CLA}
	\proof[Proof of Claim~\ref{ux-color}]
	Since $F_\varphi(r,s_1:s_\alpha)$ is a maximum typical 2-inducing   multifan, 
	 $\varphi(ru)\in \{\alpha+1, \Delta\}$. 
	Assume instead that $\varphi(ru)=\Delta$. 
	If $\alpha=1$, then we are done by exchanging the 
	role of $2$ and $\Delta$. Thus we assume that $\alpha\ge 2$. 
	Shift from $s_2$ to $s_{\alpha-1}$, 
	color $rs_1$ by 2 and uncolor $rs_\alpha$. 
	Then $F^*=(r,rs_\alpha,s_\alpha, rs_{\alpha-1},s_{\alpha-1}, \ldots, rs_1, s_1)$ is an $\alpha$-inducing multifan such that $\Delta$
	is the last $\alpha$-inducing color. 
	Now,  permuting  the name of some colors and the label of some  vertices 
	in $F^*$  yields the desired assumption.
	\qed  
	

	\begin{CLA}\label{u-sharecolor}
		For any $z\in N_{\Delta-1}(u)\setminus V(F)$ and any $F$-stable $\varphi'\in \CC^\Delta(G-rs_1)$,  if $\varphi'(ru)=\alpha+1$ and $\pbar'(z)=\alpha+1$, 
		then $\varphi'(uz)\in \pbar'(V(F))\setminus\{1\}$. 
	\end{CLA}
	\proof[Proof of Claim~\ref{u-sharecolor}]
	Assume to the contrary that $\varphi'(uz)\in \{1, \alpha+2,\dots, \Delta-1\}$. 
	We first claim that $\varphi'(uz)\ne 1$. 
	As otherwise,  $P_r(\alpha+1,1, \varphi')=ruz$, 
	contradicting Lemma~\ref{thm:vizing-fan1} \eqref{thm:vizing-fan1b} that $r$
	and $s_\alpha$ are $(\alpha+1,1)$-linked with respect to $\varphi'$. Let 
$\varphi'(uz)=\tau=\tau_1\in  \{\alpha+2,\dots, \Delta-1\}$, and $s_{h_1}\in N_{\Delta-1}(r)$ such that $\varphi'(rs_{h_1})=\tau_1$.   By Lemma~\ref{Lemma:extended multifan} \eqref{Evizingfan-a},
	$uz\in P_r(\tau,1,\varphi')$, and  $uz\in P_r(\tau_1,1,\varphi'')$
	for every  $L$-stable $\varphi'' \in \CC^{\Delta}(G-rs_1)$, 
	 where $L=(F,ru,u,uz,z)$. 
	Applying Lemma~\ref{Lemma:pseudo-fan0} \eqref{fan-and-x} on $L$ with $u$ playing the role $w$,  
	there exists a sequence of distinct vertices $s_{h_1}, \ldots, s_{h_t}\in  \{s_{\alpha+1}, \ldots, s_{\Delta-2} \}$ satisfying the following conditions: 
	\begin{enumerate}[(a)]
		\item  $\varphi'(rs_{h_{i+1}})=\pbar'(s_{h_i})=\tau_{i+1}\in  \{\alpha+2, \cdots, \Delta-1\} $ for each $i\in [1,t-1]$;  
		\item $s_{h_i}$ and $r$ are $(\tau_{i+1}, 1)$-linked with respect to $\varphi'$ for each $i\in [1,t-1]$; 
		\item $\pbar'(s_{h_t})=\tau_1$ or $\alpha+1$, and if $\pbar'(s_{h_t})=\tau_1$, then  $s_{h_t}$ and $r$ are $(\tau_{1}, 1)$-linked with respect to $\varphi'$. 
	\end{enumerate} 
	
	\medskip 
	
	\medskip 
	
	\emph{\noindent Case $\pbar'(s_{h_t})=\tau_1$.}
	In this case,  $t\ge 2$. By Lemma~\ref{Lemma:extended multifan} \eqref{Evizingfan-a}, $uz\in P_r(\tau,1)=P_{s_{h_t}}(\tau,1)$. 
	
	\emph{Suppose first that  $P_{r}(\tau,1)$ meets $z$ before $u$}. 
	Equivalently, $P_{s_{h_t}}$  meets $u$ before $z$. 
	We do the following operations:
		\[
	\begin{bmatrix}
	P_{[r,z]}(\tau,1) & ruz\\
	1/\tau & (\alpha+1)/\tau
	\end{bmatrix}.
	\]
	Denote the new coloring by $\varphi''$. Now $(r,rs_1,s_1,\ldots,s_{\alpha})$ is a multifan, but $\pbar''(s_{\alpha})=\pbar''(r)=\{\alpha+1\}$, giving a contradiction to Lemma~\ref{thm:vizing-fan1} \eqref{thm:vizing-fan1a}.

	%
	
	\emph{Thus $P_{r}(\tau,1)$ meets $u$ before $z$}.  Equivalently, $P_{s_{h_t}}$  meets $z$ before $u$. 
Shifting  from $s_{h_1}$
	to $s_{h_t}$ implies that $P_{r}(\tau,1)$ meets $z$ before $u$
	with respect to the current coloring. This gives back to the precious case.

	\medskip

	\emph{\noindent Case $\pbar'(s_{h_t})=\alpha+1$.}
	In this case, $t\ge 1$. 
	\emph{If $z\ne s_{h_t}$}, then we 
	shift  from $s_{h_1}$ to $s_{h_t}$, 
	and do $ru: \alpha+1\rightarrow \tau_1$, $uz: \tau_1\rightarrow \alpha+1$.
	Denote the new coloring by $\varphi''$. As $\varphi''$ is $F$-stable with respect to $\varphi'$, and so is $F$-stable  with respect to $\varphi$, 
	we see that  $F^*=(F, rs_{h_t}, s_{h_t}, rs_{h_{t-1}}, s_{h_{t-1}}, \ldots, rs_{h_1}, s_{h_1})$
	is a multifan that contains more vertices than
	$F$ does,  
	showing a 
	contradiction to the choice of $\varphi$.  
	\emph{Thus we assume that $z=s_{h_t}$}. This gives that $\varphi'(rz)=\tau_t$.  So $t\ge 2$. 
Note that  $uz\in P_r(\tau,1)=P_{w}(\tau,1)$, for some vertex $w\in V(G)\setminus(V(F)\cup \{s_{h_1}, \ldots, s_{h_t}\})$.
	We consider two cases to reach contradictions. 
	
	{\it If  $P_{w}(\tau,1)$ meets $u$ before $z$}, 
	We do the following operations:
	\[
	\begin{bmatrix}
P_{[w,u]}(\tau,1) & ru & uz\\
	1/\tau & \alpha+1\rightarrow 1& \tau \rightarrow \alpha+1
	\end{bmatrix}.
	\]
	Denote the new coloring by $\varphi''$. Now $(r,rs_1,s_1,\ldots,s_{\alpha})$ is a multifan, but $\pbar''(s_{\alpha})=\pbar''(r)=\{\alpha+1\}$, giving a contradiction to Lemma~\ref{thm:vizing-fan1} \eqref{thm:vizing-fan1a}.
	

	
	{\it If  $P_{w}(\tau,1)$ meets $z$ before $u$},
	We do the following operations:
	\[
	\begin{bmatrix}
	P_{[w,z]}(\tau,1) &  s_{h_1}:s_{h_{t-1}} & rs_{h_t}=rz & ruz\\
	1/\tau &  \text{shift} & \tau_t\rightarrow 1 & (\alpha+1)/\tau 
	\end{bmatrix}.
	\]
	Denote the new coloring by $\varphi''$. Now $(r,rs_1,s_1,\ldots,s_{\alpha})$ is a multifan, but $\pbar''(s_{\alpha})=\pbar''(r)=\{\alpha+1\}$, giving a contradiction to Lemma~\ref{thm:vizing-fan1} \eqref{thm:vizing-fan1a}.
	\qed

	\begin{CLA}\label{u-nonadj}
		Let  $z\in N_{\Delta-1}(u)\setminus V(F)$  and any $F$-stable $\varphi'\in \CC^\Delta(G-rs_1)$ such that $\varphi'(ru)=\alpha+1$ and $\pbar'(z)=\alpha+1$, and let  $\varphi'(uz)=\tau$.
		Then 	$\tau\in \pbar'(V(F))\setminus \{1\}$, and 
		$u\not\sim s_1,s_\alpha$ 	if $\tau =\Delta$; and $u\not\sim s_{\tau-1}, s_\tau$ if $\tau \in \{2,\ldots, \alpha+1\}$.
	\end{CLA}
	\proof[Proof of Claim~\ref{u-nonadj}]
	By Claim~\ref{u-sharecolor},   $\tau\in \pbar'(V(F))\setminus \{1\}$. 
	Thus, $\tau \in \{2, \ldots, \alpha+1, \Delta\}$. 
	Applying Lemmas~\ref{Lem:2-non-adj1} and \ref{Lem:2-non-adj2} yields 
	the conclusion. 
	\qed 
	
	\begin{CLA}\label{pseudo-fan}
		Suppose that $N_{\Delta-1}(r)=N_{\Delta-1}(u)$ for every $u\in N_\Delta(r)$.
		Then for every  $F$-stable coloring $\varphi'\in \CC^\Delta(G-rs_1)$,  $N_{\Delta-1}[r]$ is $\varphi'$-elementary. In particular, $N_{\Delta-1}[r]$ is the vertex set of 
		either a  typical  2-inducing multifan or a typical 2-inducing pseudo-multifan  
		with respect to $rs^*$ and  $\varphi^*\in \CC^\Delta(G-rs^*)$  for some $s^*\in N_{\Delta-1}(r)$. 
	\end{CLA}
	
	\proof[Proof of Claim~\ref{pseudo-fan}] 
	Assume to the contrary that there exists an $F$-stable coloring $\varphi'\in \CC^\Delta(G-rs_1)$ such that $N_{\Delta-1}[r]$ is not  $\varphi'$-elementary.  
	Since $V(F)$ is $\varphi'$-elementary, there exists $z\in N_{\Delta-1}[r]\setminus V(F)$
	such that $\pbar'(z)\in \pbar'(F)$ or there exists $z^*\ne z$ with $z^*\in N_{\Delta-1}[r]\setminus V(F)$
	such that $\pbar'(z)=\pbar'(z^*)$.  Let $\pbar'(z)=\delta$. 
	If $\delta\in \pbar'(F)$, then $z$ and $r$
	are $(\delta,1)$-unlinked,   so we do $(\delta,1)-(1,\alpha+1)$-swaps at $z$; 
	if  $\pbar'(z)=\pbar'(z^*)$, we may assume, without loss of generality, that $z$ and $r$
	are $(\delta,1)$-unlinked, we again do  $(\delta,1)-(1,\alpha+1)$-swaps at $z$. 
	In either case, we find an $F$-stable coloring $\varphi''\in \CC^\Delta(G-rs_1)$ such that 
	$\pbar''(z)=\alpha+1$.  Since for any $u\in N_\Delta(r)$,  it holds that  $N_{\Delta-1}(r)=N_{\Delta-1}(u)$,
	we can choose $u\in N_\Delta(r)$  such that $\varphi''(ur)=\alpha+1$, where $\alpha+1$ is the last 2-inducing color of $F_{\varphi''}(r,s_1:s_\alpha)$.  Since $N_{\Delta-1}(r)=N_{\Delta-1}(u)$, 
	$L=(F_{\varphi''}(r,s_1:s_\alpha), ru, u, uz, z)$ is a lollipop with respect to $\varphi''$. 
	By Claim~\ref{u-nonadj},  $u$ is not adjacent to at least one vertex in $N_{\Delta-1}(r)$, 
	which in turn shows $N_{\Delta-1}(r)\ne N_{\Delta-1}(u)$, giving a contradiction. 
	
	Therefore, for every $F$-stable  coloring $\varphi'\in \CC^\Delta(G-rs_1)$, it holds that  $N_{\Delta-1}[r]$ is $\varphi'$-elementary.  Consequently, there 
	is a   multifan or a  pseudo-multifan  with vertex set $N_{\Delta-1}[r]$. 
	By permuting the name of the colors and the label of vertices in 
	$N_{\Delta-1}(r)$, we can assume that the multifan or  pseudo-multifan with vertex set $N_{\Delta-1}[r]$ is typical. If $N_{\Delta-1}(r)$ is the vertex set of a multifan,  by  Lemma~\ref{2-inducing },  we can further assume that the multifan is typical 2-inducing. 
	 If $N_{\Delta-1}(r)$ is the vertex set of a pseudo-multifan,  by Lemma~\ref{pseudo-fan-ele:e} and  Lemma~\ref{2-inducing }, we can further assume that the pseudo-multifan is typical 2-inducing. 
	\qed 
	
	{\bf By Claim~\ref{pseudo-fan}, it suffices to only show Theorem~\ref{Thm:vizing-fan2b} \eqref{common2}}.
	Assume to the contrary that there exist $r\in N_\Delta$ and $u\in N_{\Delta-1}(r)$ such that 
	$N_{\Delta-1}(u)\setminus N_{\Delta-1}(r)\ne \emptyset$. 
	
		\begin{CLA}\label{x-missing-color}
			For every $z\in N_{\Delta-1}(u)\setminus N_{\Delta-1}(r)$, there is  an $F$-stable coloring 
			$\varphi'\in \CC^\Delta(G-rs_1)$ with respect to $\varphi$ such that
			$\varphi'(ru)=\alpha+1$ and  $\pbar'(z)=\alpha+1$. 
		\end{CLA}
		\proof[Proof of Claim~\ref{x-missing-color}]
		%
		Let $z\in N_{\Delta-1}(u)\setminus N_{\Delta-1}(r)$. 
		By Claim~\ref{ux-color}, assume that $\varphi(ru)=\alpha+1$. 
		Let $\pbar(z)=\delta$.  If $\delta=\alpha+1$, we simply let $\varphi'=\varphi$.  So    $\delta\ne \alpha+1$. 
		If $\delta\in \pbar(V(F))$, 
		we let  $\varphi'$ be obtained from $\varphi$ by doing $(\delta,1)-(1,\alpha+1)$-swaps
		at $z$.   This gives that $\pbar'(z)=\alpha+1$. 
		By Lemma~\ref{thm:vizing-fan1} \eqref{thm:vizing-fan1b}, $\varphi'$
		is $F$-stable and  $\varphi'(ru)=\varphi(ru)=\alpha+1$. Thus $\varphi'$  is a desired coloring. 
		
		Assume now that $\delta\in \{\alpha+2, \ldots, \Delta-1\}$. 
		If there is an $F$-stable $\varphi'' \in \CC^{\Delta}(G-rs_1)$ 
		such that $z\not\in P_r(\delta,1, \varphi'')$, 
		let $\varphi'$ be obtained from $\varphi''$
		by doing $(\delta,1)-(1,\alpha+1)$-swaps
		at $z$. Since 
		$\varphi(ru)=\alpha+1$ and $r$ and $s_\alpha$
		are $(\alpha+1,1)$-linked with respect to $\varphi$
		by Lemma~\ref{thm:vizing-fan1} \eqref{thm:vizing-fan1b},  
		it holds that  $\varphi'$ is $F$-stable  and 
		$\varphi'(ru)=\varphi''(ru)=\varphi(ru)=\alpha+1$. Thus, $\varphi'$ is a desired coloring and 
		we are done.  
		Therefore every  $F$-stable $\varphi'' \in \CC^{\Delta}(G-rs_1)$ satisfies  $z\in P_r(\delta,1, \varphi'')$. 
		 	Applying Lemma~\ref{Lemma:pseudo-fan0} \eqref{only-fan} with  $z$ playing the role of $w$,  there exists 
	$s_{h_t}\in N_{\Delta-1}(r)\setminus V(F)$ such that $\pbar(s_{h_t})=\delta$
	and $s_{h_t}$ and $r$ are $(\delta,1)$-linked with respect to $\varphi$. 
	Now let  $\varphi'$ be obtained from $\varphi$ by doing $(\delta,1)-(1,\alpha+1)$-swaps
	at $z$.  Then  $\varphi'$ is a desired coloring.   
				\qed 
		
		\begin{CLA}\label{u-outneighbor-x}
			We may assume that $|N_{\Delta-1}(u)\setminus N_{\Delta-1}(r) |\ge 2$. 
		\end{CLA}
		\proof[Proof of Claim~\ref{u-outneighbor-x}]
		Let $x\in N_{\Delta-1}(u)\setminus N_{\Delta-1}(r)$. 
		By Claim~\ref{x-missing-color},  we choose an $F$-stable coloring from 
		$\CC^\Delta(G-rs_1)$ and call it still $\varphi$ such that $\varphi(ru)=\alpha+1$ and $\pbar(x)=\alpha+1$. 
		By Claim~\ref{u-nonadj},  $\varphi(ux)\in  \{2, \ldots, \alpha+1, \Delta\}$. 
		If $|V(F)|\ge 3$, then 
		Claim~\ref{u-nonadj} gives that 
		$|N_{\Delta-1}(u)\setminus N_{\Delta-1}(r) |\ge 2$. 
		Thus we have 	$V(F)=\{r,s_1\}$. Consequently, $\alpha+1=2$, and $\varphi(ux)=\Delta$ by 
		the fact that $\varphi(ux)\in  \{2, \ldots, \alpha+1, \Delta\}$. 
		We assume further that 
		$$
		 N_{\Delta-1}(u)\setminus N_{\Delta-1}(r) =\{x\}.
		$$
	By Claim~\ref{u-nonadj}, $u\not\sim s_1$. 
		We consider two cases.
		Assume first that there exists an $F$-stable $\varphi'\in \CC^\Delta(G-rs_1)$ 
		such that $N_{\Delta-1}[r]$ is not $\varphi'$-elementary.  
		By exchanging the role of $2$ and $\Delta$ if necessary, we may assume 
		$\varphi'(ru)=2$.
		Since $V(F)$ is $\varphi'$-elementary, there exists $z\in N_{\Delta-1}(r)\setminus V(F)$
		such that $\pbar'(z)\in \pbar'(F)$ or there exists $z^*\ne z$ with $z^*\in N_{\Delta-1}(r)\setminus V(F)$
		such that $\pbar'(z)=\pbar'(z^*)$. Let $\pbar'(z)=\delta$. 
		If $\delta\in \pbar'(F)$,   then as $r$ and $z$ are $(\delta,1)$-unlinked, we do $(\delta,1)-(1,2)$-swaps at $z$; 
		if  $\pbar'(z)=\pbar'(z^*)$, we may assume, without loss of generality, that $z$ and $r$
		are $(\delta,1)$-unlinked, we again do  $(\delta,1)-(1,2)$-swaps at $z$. 
		In either case, we find an $F$-stable coloring $\varphi''\in \CC^\Delta(G-rs_1)$ with  
		$\varphi''(ru)=\varphi'(ru)=2$
		and $\pbar''(z)=2$. Note that $z\in N_{\Delta-1}(u)$ since $z\ne x$ and $N_{\Delta-1}(u)\setminus  N_{\Delta-1}(r)=\{x\}$. 
	By	Claim~\ref{u-sharecolor}, 
	 $\varphi''(uz)\in \{2, \Delta\}$, which implies $\varphi''(uz)=\Delta$ by noting $\varphi''(ru)=2$.
Furthermore, we assume $uz\in P_{s_1}(\Delta,1,\varphi'')=P_r(\Delta,1,\varphi'')$. 
		Since $\varphi''(ru)=2$ and $\varphi''(uz)=\Delta$, $\varphi''(ux)\ne 2, \Delta$. 
		Thus $\varphi''(ux)\in \{1,3, 4, \ldots, \Delta-1\}$, which 
		implies $\pbar''(x) \ne 2$ by 
		 Claim~\ref{u-sharecolor}.  
		Let $\pbar''(x)=\tau$ and $\varphi''(ux)=\lambda$. 
		Note that if $\tau=\Delta$ then $\lambda\ne 1$,  as $u\in P_{s_1}(\Delta,1,\varphi'')=P_r(\Delta,1,\varphi'')$. 
		Thus if $\tau=\Delta$ or $1$, we do  $(\tau,1)-(1,2)$-swaps at $x$. 
		As the color of $ux$ is not $\Delta$ after these swaps, we get  a contradiction to Claim~\ref{u-sharecolor}. 
		Thus, we assume that $\tau\in \{3, 4, \ldots, \Delta-1\}$, and that $P_x(\tau, 1, \varphi''')=P_r(\tau, 1, \varphi''')$
		for any $L$-stable  coloring $\varphi'''$,  where $L=(F,ru, u, uz,z)$ is a lollipop.  Let $\tau_1=\tau$ and $s_{h_1}\in N_{\Delta-1}(r)$ such that $\varphi''(rs_{h_1})=\tau_1$. 
		Applying  Lemma~\ref{Lemma:pseudo-fan0} \eqref{fan-and-x} on $L$ with $x$ playing the role of $w$, 
				there exists a sequence of distinct vertices $s_{h_1}, \ldots, s_{h_t}\in  \{s_{\alpha+1}, \ldots, s_{\Delta-2} \}$ satisfying the following conditions: 
		\begin{enumerate}[(a)]
			\item   $\varphi''(rs_{h_{i+1}})=\pbar''(s_{h_i})=\tau_{i+1}\in  \{\alpha+2, \cdots, \Delta-1\} $ for each $i\in [1,t-1]$;  
					\item $s_{h_i}$ and $r$ are $(\tau_{i+1}, 1)$-linked with respect to $\varphi''$ for each $i\in [1,t-1]$;
			\item $\pbar''(s_{h_t})=\tau_1$ or $2$, and if $\pbar''(s_{h_t})=\tau_1$, then  $s_{h_t}$ and $r$ are $(\tau_{1}, 1)$-linked with respect to $\varphi''$. 
		\end{enumerate} 
		As $x$ and $r$ are $(\tau_1,1)$-linked, we conclude that $\pbar''(s_{h_t})=2$. 
		As $\varphi''(uz)=\Delta$, $\varphi''(ur)=2$, if $s_{h_t}\ne z$, then  
		$\varphi''(us_{h_t})\in \{1, 3, 4, \ldots, \Delta-1 \}$. 
	This gives a contradiction to Claim~\ref{u-sharecolor}. 
		Thus we assume that $s_{h_t}=z$. 
		Notice that  $r\in P_{s_1}(\tau,2,\varphi'')$ by the maximality of $|V(F)|$.
		Therefore
		$r\in P_x(\tau,2,\varphi'')=P_{s_1}(\tau,2,\varphi'')$. 
		So $s_{h_t}$ is $(2,\tau)$-unlinked with $s_1, x$ and $r$ with respect to $\varphi''$. 
		We do a $(2,\tau)$-swap at $s_{h_t}$ and then shifting  from $s_{h_1}$
		to $s_{h_t}$. This gives a coloring such   that $s_1$ and $x$ are $(\tau,2)$-unlinked with respect to the coloring. 
		Again, with respect to the current coloring,  $r\in P_{s_1}(\tau,2)$ by the maximality of $|V(F)|$.
		We do a $(\tau,2)$-swap at $x$ to get a coloring $\varphi'''$. 
		Note that $\varphi'''(ru)=\varphi''(ru)=2$, $\varphi'''(ux)=\varphi''(ux)=\lambda$,
		$\pbar'''(x)=2$, and $\varphi'''(uz)=\varphi''(uz)=\Delta$.
		Therefore, $\varphi'''(ux)=\lambda\in \{1,3,4,\ldots, \Delta-1\}$, 
		showing  a contradiction to Claim~\ref{u-sharecolor}.

		Thus  we assume that   $N_{\Delta-1}[r]$  is $\varphi'$-elementary for every $F$-stable $\varphi'\in \CC^\Delta(G-rs_1)$ with respect to $\varphi$. 
		In particular, $N_{\Delta-1}[r]$  is $\varphi$-elementary, and  as $|V(F)|=2$ and $F$  is maximum at $r$, 
		we know that $N_{\Delta-1}[r]$ 
		is contained in a pseudo-multifan $S=S_{\varphi}(r,s_1:s_1:s_{\Delta-2})$. Let $\delta\in \pbar(V(S))\setminus \pbar(V(F))$.
		By Lemma~\ref{pseudo-fan-ele}~\eqref{pseudo-b}, $\pbar^{-1}_S(\delta)$
		is $(\delta,2)$- and $(\delta,\Delta)$-linked with $s_1$ 
		and the corresponding chains contain the vertex $r$ with respect to $\varphi$.  
		Recall that $\varphi(ru)=2, \varphi(ux)=\Delta$, and $\pbar(x)=2$.
		Let $\varphi'$ be obtained from $\varphi$ by doing 
	 the following swaps of colors at $x$:
		$$(2,\delta)-(\delta, \Delta)-(\Delta,1)-(1,2).$$
		Since $\varphi'$ is $F$-stable, $\varphi'(ru)=2$, $\varphi'(ux)=\delta$, and $\pbar'(x)=2$, we get 
		  a contradiction to Claim~\ref{u-sharecolor}. 
		\qed 
		
		\begin{CLA}\label{u-2-out-neighbor}
			Let $x,y\in N_{\Delta-1}(u)\setminus N_{\Delta-1}(r)$ be distinct, and $\varphi'\in \CC^\Delta(G-rs_1)$
			be any $F$-stable coloring with $\varphi'(ru)=\alpha+1$.  
		If $\pbar'(x)\in \pbar'(V(F))$ and $\pbar'(x)\ne 1$, then $\pbar'(y)\not\in \pbar'(V(F))$ and $y$ and $r$
			are $(\pbar'(y), 1)$-linked with respect to $\varphi'$. 
		\end{CLA}
		\proof[Proof of Claim~\ref{u-2-out-neighbor}]
		
		The second  part of the claim  follows easily from the first part. Since otherwise, a $(\pbar'(y),1)$-swap at $y$
		implies that $1$ is missing at $y$, contradicting the first part. 

		Assume to the contrary that $\pbar'(x)\in \pbar'(V(F))$ and $\pbar'(y)\in \pbar'(V(F))$.  
		We claim that $\pbar'(x)=\pbar'(y)=\alpha+1$ or 
		$\pbar'(x)=\alpha+1$ and $\pbar'(y)=1$.  
		By doing  $(\pbar'(x),1)-(1,\alpha+1)$-swaps at $x$, 
		we  assume that $\pbar'(x)=\alpha+1$. 
		Since $1,\alpha+1\in \pbar'(V(F))$, we still 
		have $\pbar'(y)\in \pbar'(V(F))$. 
		  If $\pbar'(y)=\alpha+1$, then we are done. 
		Otherwise, doing a $(\pbar'(y),1)$-swap at $y$ gives a desired coloring.
		$$\text{Let $\varphi'(ux)=\tau$ and $\varphi'(uy)=\lambda$.}$$
		We consider now two cases to finish the proof of Claim~\ref{u-2-out-neighbor}. 
		
		\medskip 
		\emph{\noindent \setword{Case A}{Case A}: $\pbar'(x)=\pbar'(y)=\alpha+1$.}
		
		\medskip

		By Claim~\ref{u-sharecolor}, $\tau, \lambda\in \pbar'(V(F))\setminus\{1\}$. 
		Assume, without loss of generality, that $\tau\ne \Delta$. 
		Then $\tau\in \{2,\ldots, \alpha+1\}$ is a 2-inducing color of $F$. 
		By Lemma~\ref{Lemma:extended multifan}~\eqref{Evizingfan-d} 
		that $r\in P_{s_\alpha}(\Delta,\alpha+1)=P_{s_1}(\Delta,\alpha+1)$, 
		we know $\lambda\ne \Delta$. 
		Thus $\lambda\in \{2,\ldots, \alpha+1\}$ is also a 2-inducing color. 
		By symmetry between $x$ and $y$, we assume  $\lambda \prec \tau$. 
		Shift from $s_2$ to $s_{\lambda-1}$, uncolor $rs_\lambda$, 
		then color $rs_1$ by 2.  Denote the resulting coloring by $\varphi''$. 
		Now  $F^*=(r, rs_\lambda, s_\lambda, rs_{\lambda+1}, s_{\lambda+1}, \ldots, rs_\alpha, s_\alpha, rs_{\lambda-1}, s_{\lambda-1}, \ldots, rs_1,s_1)$ is a new  multifan with respect to $\varphi''$ that has the same vertex set 
		as $F_{\varphi'}(r,s_1:s_\alpha)$. In this new multifan $F^*$, 
		$\lambda$ is itself a $\lambda$-inducing color, yet $\tau$ is a $(\lambda+1)$-inducing color. 
		However, $r\in P_y(\alpha+1, \lambda,\varphi'')$, showing a contradiction to 
		Lemma~\ref{Lemma:extended multifan}~\eqref{Evizingfan-d}  that $r\in P_{s_\alpha}(\alpha+1,\lambda, \varphi'')=P_{s_{\lambda}}(\alpha+1,\lambda,\varphi'')$.  
		
		\medskip 
		\emph{\noindent Case B: $\pbar'(x)=\alpha+1$ and $\pbar'(y)=1$.}
		
		\medskip 
		
		We  assume that $x$ and $y$ are $(\alpha+1,1)$-linked with respect to $\varphi'$. For otherwise, a $(1,\alpha+1)$-swap at 
		$y$ reduces the problem to \ref{Case A}.

		We show that $\tau, \lambda \ne \Delta$.  If this is not the case, then  by swapping colors along $P_{[x,y]}(\alpha+1,1)$ and exchanging the role of $x$ and $y$ if necessary,  we assume that 
		$\tau \ne \Delta$ and $\lambda= \Delta$. Let $\varphi''$ be obtained 
		from $\varphi'$ by a $(1,\Delta)$-swap at $y$. 
		By Lemma~\ref{Lemma:extended multifan}~\eqref{Evizingfan-d}, $r\in P_{s_1}(\Delta, \alpha+1,\varphi'')=P_{s_\alpha}(\Delta,\alpha+1,\varphi'')$. 
		Thus, we can do a $(\Delta, \alpha+1)$-swap at $y$ without affecting the coloring of $F_{\varphi''}(r,s_1:s_\alpha)$ and $\varphi''(ru)$. 
		Thus, let $\varphi^*=\varphi''/P_y(\Delta, \alpha+1,\varphi'')$. 
		We see that $P_r(\alpha+1,1,\varphi^*)=ruy$, showing a contradiction to Lemma~\ref{thm:vizing-fan1}~\eqref{thm:vizing-fan1b} that $r$
		and $s_\alpha$ are $(\alpha+1,1)$-linked with respect to $\varphi^*$. 
		
		Since  $\tau, \lambda \ne \Delta$, 
		both $\tau$ and $\lambda$ are 2-inducing colors of $F$ 
		by Claim~\ref{u-sharecolor}.  
		By swapping colors along $P_{[x,y]}(\alpha+1,1)$ and exchanging the role of $x$
		and $y$ if necessary, we assume 
		$  \lambda \prec \tau$.  
		Note that $r\in P_{s_1}(\Delta, \lambda)=P_{s_{\lambda-1}}(\Delta, \lambda)$  and 
		$r\in P_{s_1}(\Delta,\alpha+1)=P_{s_\alpha}(\Delta,\alpha+1)$ by Lemma~\ref{Lemma:extended multifan}~\eqref{Evizingfan-e}  and ~\eqref{Evizingfan-d}, respectively. 
		Let $\varphi''$ be obtained from $\varphi'$ by doing the following swaps at $y$: 
		$$ (1,\Delta)-(\Delta, \lambda )-(\lambda,1)-(1,\Delta)-(\Delta,\alpha+1).$$
		Note that $\varphi''$ is $F$-stable, and 
		 that $P_r(\alpha+1,1,\varphi'')=ruy$, showing a contradiction to Lemma~\ref{thm:vizing-fan1}~\eqref{thm:vizing-fan1b} that $r$
		and $s_\alpha$ are $(\alpha+1,1)$-linked with respect to $\varphi''$. 
		\qed 
		
		By  Claim~\ref{x-missing-color} and Claim~\ref{u-outneighbor-x},  we let 
		$x,y\in N_{\Delta-1}(u)\setminus N_{\Delta-1}(r)$  with $x\ne y$, 
		and assume that 
		$\varphi(ru)=\alpha+1$  and $\pbar(x)=\alpha+1$. 
		By Claim~\ref{u-2-out-neighbor}, 
		we also assume that $\pbar(y)=\delta \in \{\alpha+2,\ldots, \Delta-1\}$
		and $y$ and $r$ are $(\delta,1)$-linked with respect to such a coloring $\varphi$. 
		Let $\delta_1=\delta$, $s_{h_1}\in N_{\Delta-1}(r)$ such that 
		$\varphi(rs_{h_1})=\delta_1$,  and $L=(F,ru,u, ux,x)$. 
		By Claim~\ref{u-2-out-neighbor}, 
		for any $L$-stable $\varphi' \in \CC^{\Delta}(G-rs_1)$,   it holds that 
		$y\in P_r(\delta,1,\varphi')$. 
		Applying Lemma~\ref{Lemma:pseudo-fan0}~\eqref{fan-and-x} 
		on $L$ with $y$ playing the role of $w$,
		there exists a sequence of distinct vertices $s_{h_1}, s_{h_2}, \ldots, s_{h_t}\in  \{s_{\alpha+1}, \ldots, s_{\Delta-2} \}$ satisfying the following conditions: 
		\begin{enumerate}[(a)]
			\item  $\varphi(rs_{h_{i+1}})=\pbar(s_{h_i})=\delta_{i+1}\in  \{\alpha+2, \cdots, \Delta-1\}$ for each $ i\in [1,t-1]$; 
						\item $s_{h_i}$ and $r$ are $(\delta_{i+1}, 1)$-linked with respect to $\varphi$ for each $i\in [1,t-1]$;
			\item $\pbar(s_{h_t})=\delta_1$ or $\alpha+1$, and if $\pbar(s_{h_t})=\delta_1$, then  $s_{h_t}$ and $r$ are $(\delta_1, 1)$-linked with respect to $\varphi$. 
		\end{enumerate} 
		If $\pbar(s_{h_t})=\delta_1$, then since $s_{h_t}$ and $r$ are $(\delta,1)$-linked, a $(\delta,1)$-swapping at $y$ implies that the color 1 is missing at $y$, showing a contradiction to Claim~\ref{u-2-out-neighbor}. 
		Therefore, we assume that $\pbar(s_{h_t})=\alpha+1$. 
		Let $\varphi(ux)=\tau$ and $\varphi(uy)=\lambda$.  Since 
		$\varphi(ru)=\alpha+1$, $\alpha+1\notin \{\tau,\lambda,\delta_1\}$. 
		By Claim~\ref{u-nonadj}, $\tau\in \pbar(V(F))\setminus\{1\}$, 
		 and if $\varphi(ux)=\tau=\Delta$, 
		\begin{equation}\label{2-small-neighborsa}
		s_{1}, s_{\alpha} \not\in N_{\Delta-1}(u),
		\end{equation}
		if $\varphi(ux)=\tau\ne \Delta$, 
		\begin{equation}\label{2-small-neighborsb}
		s_{\tau-1}, s_{\tau} \not\in N_{\Delta-1}(u). 
		\end{equation}

		\begin{CLA}\label{u-3-neighbor}
			$|N_{\Delta-1}(u)\setminus N_{\Delta-1}(r) |\ge 3$. 
		\end{CLA}
		\proof[Proof of Claim~\ref{u-3-neighbor}]
		We  first show that if $\lambda=\Delta$,  
		\begin{equation}\label{2-small-neighbors2a}
		s_{1}, s_{\alpha} \not\in N_{\Delta-1}(u);  
		\end{equation}
	and if $\lambda\ne \Delta$,  
		\begin{equation}\label{2-small-neighbors2b}
		s_{\lambda-1}, s_{\lambda} \not\in N_{\Delta-1}(u). 
		\end{equation} 
		
	To see this, let $\varphi'$ be obtained from $\varphi$
		by first doing an $(\alpha+1,1)$-swap at both $x$ and $s_{h_t}$, 
		and then the shifting from $s_{h_1}$ to $s_{h_t}$. 
		Now, $\pbar'(r)=\delta_1$ and $\varphi'(ux)=\varphi(ux)=\tau$. 
		Let $\varphi''=\varphi'/P_y(\delta_1, \alpha+1, \varphi')$. 
		Note that $\varphi''(ux)=\varphi(ux)=\tau$ and $\varphi''(uy)=\varphi(uy)=\lambda$.  Applying Claim~\ref{u-sharecolor} to the coloring $\varphi''$ gives 
		 $\varphi''(uy)=\lambda\in \pbar''(V(F))\setminus\{\delta_1\}$.  
		As $\tau,\lambda,\delta_1,\alpha+1 \in \pbar''(V(F))$ and they are all distinct, 
		$|V(F)|\ge |\{\delta_1,\tau,\lambda,\alpha+1\}|-1 = 3$.  Then ~\eqref{2-small-neighbors2a} and ~\eqref{2-small-neighbors2b} follow   from  Claim~\ref{u-nonadj}. 
		 		These two facts, 
		together with \eqref{2-small-neighborsa} and  \eqref{2-small-neighborsb},
		imply 
		$$
		\text{either}\quad s_1,s_\alpha, s_{\lambda-1}, s_{\lambda} \not\in N_{\Delta-1}(u), \quad \text{or} \quad s_1,s_\alpha, s_{\tau-1}, s_{\tau} \not\in N_{\Delta-1}(u).
		$$
		Note that $s_1\ne s_\alpha$ by $|V(F)|\ge 3$. 
		We obtain $
		|N_{\Delta-1}(u)\setminus N_{\Delta-1}(r)|\ge  3 
		$	from the above  unless 
		\[\text{either} \quad {\lambda}={\alpha}=2 \quad \text{or} \quad {\tau}={\alpha}=2.
		\]
		Therefore we assume  $\alpha=2$ and $\{\lambda,\tau\}=\{2,\Delta\}$. By symmetry, we may assume $\alpha=2$, $\tau=2$ and $\lambda=\Delta$. 
		Furthermore, we have $|N_{\Delta-1}(u)\setminus N_{\Delta-1}(r)|=2$ as the facts from \eqref{2-small-neighborsa} to \eqref{2-small-neighbors2b}
		imply that  $s_1, s_2\not\in N_{\Delta-1}(u)$. Therefore 
		$N_{\Delta-1}(r)\setminus\{s_1,s_2\}\subseteq N_{\Delta-1}(u)$. 
		In particular, $s_{h_t}\in N_{\Delta-1}(u)$. Since $r$
		and $s_\alpha$ are $(1,\alpha+1)$-linked with respect to $\varphi$
		and $\pbar(s_{h_t})=\alpha+1$, it follows that $\varphi(us_{h_t})\ne 1$. This, together with the facts that  $\pbar(V(F))=\{1,2,3,\Delta\}$, $\varphi(ru)=3$, $\varphi(ux)=2$, and $\varphi(uy)=\Delta$,  implies  that $\varphi(us_{h_t})\in \{4,\ldots, \Delta-1\}$, 
		showing a contradiction to Claim~\ref{u-sharecolor}. 
		\qed 
		
		By Claim~\ref{u-3-neighbor}, 
		let $x,y,z\in N_{\Delta-1}(u)\setminus N_{\Delta-1}(r)$ be distinct.  
		We assume  $\varphi(ru)=\alpha+1$ and $\pbar(x)=\alpha+1$ by 
		Claim~\ref{x-missing-color}. 
		By Claim~\ref{u-2-out-neighbor},  we assume  $\pbar(y)=\delta$
		and $\pbar(z)=\lambda$ with  
		$\delta, \lambda\in \{\alpha+2,\ldots, \Delta-1\}$, 
		$y$ and $r$ are $(\delta,1)$-linked with respect to $\varphi$,  and $z$ and $r$ are $(\lambda, 1)$-linked with respect to $\varphi$. 
		Consequently,  $\lambda\ne \delta$.

		Let $L=(F,ru, u,ux, x)$. Clearly, for any $L$-stable coloring 
		$\varphi'\in \CC^\Delta(G-rs_1)$, $y\in P_r(\delta,1,\varphi')$. 
		Applying Lemma~\ref{Lemma:pseudo-fan0}~\eqref{fan-and-x} at $L$ and $y$,
		there exists a sequence of distinct vertices $s_{h_1}, s_{h_2}, \ldots, s_{h_t}\in  \{s_{\alpha+1}, \ldots, s_{\Delta-2} \}$ satisfying the following conditions: 
		\begin{enumerate}[(a)]
			\item  $\varphi(rs_{h_1})=\delta_1$, $\varphi(rs_{i+1})=\pbar(s_{h_i})=\delta_{i+1}\in  \{\alpha+2, \cdots, \Delta-1\}$ for each $ i\in [1,t-1]$; 
			
			\item $s_{h_i}$ and $r$ are $(\delta_{i+1}, 1)$-linked with respect to $\varphi$ for each $i\in [1,t-1]$;
			\item $\pbar(s_{h_t})=\delta_1$ or $\alpha+1$, and if $\pbar(s_{h_t})=\delta_1$, then  $s_{h_t}$ and $r$ are $(\delta_1, 1)$-linked. 
		\end{enumerate} 
		If $\pbar(s_{h_t})=\delta_1$, then since $s_{h_t}$ and $r$ are $(\delta,1)$-linked with respect to $\varphi$, a $(\delta,1)$-swapping at $y$ gives a contradiction to Claim~\ref{u-2-out-neighbor}. 
		Therefore, we assume  $\pbar(s_{h_t})=\alpha+1$. 
		Furthermore, we assume that $s_{h_t}$
		and $x$ are $(\alpha+1,1)$-linked with respect to $\varphi$. 
		For otherwise, first doing an $(\alpha+1,1)$-swap at $s_{h_t}$, then shifting from $s_{h_1}$
		to $s_{h_t}$ give a coloring $\varphi'$ such that $\varphi'(ru)=\varphi(ru)=\alpha+1$, $\pbar'(y)=\pbar'(r)=\delta_1$, while $\pbar'(x)=\alpha+1$. Since $\varphi'$
		is $F$-stable up to exchanging the role of 1 and $\delta_1$, 
		we obtain a contradiction to Claim~\ref{u-2-out-neighbor}. 
		As $z$ and $r$ are $(\lambda,1)$-linked with respect to $\varphi$ and $s_{h_i}$ and $r$ are $(\delta_{i+1}, 1)$-linked for each $i\in [1,t-1]$, $\lambda \ne \delta_i$ for each $i\in [2,t]$. 
				Let $\lambda_1=\lambda$ and $s_{p_1}$ be the neighbor of $r$
		such that $\varphi(rs_{p_1})=\lambda_1$. 
		For any $L$-stable coloring 
		$\varphi'\in \CC^\Delta(G-rs_1)$, $z\in P_r(\lambda,1,\varphi')$. 
		Applying Lemma~\ref{Lemma:pseudo-fan0}~\eqref{fan-and-x} on $L=(F,ru, u,ux,x)$ and $z$,
		there exists a sequence of distinct vertices $s_{p_1}, s_{p_2}, \ldots, s_{p_k}\in  \{s_{\alpha+1}, \ldots, s_{\Delta-2} \}$ satisfying the following conditions: 
		\begin{enumerate}[(a)]
			\item  $\varphi(rs_{p_{i+1}})=\pbar(s_{p_i})=\lambda_{i+1}\in  \{\alpha+2, \cdots, \Delta-1\}$ for each $i\in [1,k-1]$; 
						\item $s_{p_i}$ and $r$ are $(\lambda_{i+1}, 1)$-linked with respect to $\varphi$ for each $i\in [1,k-1]$;
			\item $\pbar(s_{p_k})=\lambda_1$ or $\alpha+1$, and if $\pbar(s_{p_k})=\lambda_1$, then  $s_{p_k}$ and $r$ are $(\lambda_1, 1)$-linked with respect to $\varphi$. 
		\end{enumerate} 
		
		Recall that  $s_{p_1}\ne s_{h_i}$ for each $i\in[1,t]$. 
		Furthermore, as  $s_{h_i}$ and $r$ are $(\delta_{i+1}, 1)$-linked for each $i\in [1,t-1]$ and for each $j\in [1,k-1]$, 
		$s_{p_j}$ and $r$ are $(\lambda_{j+1}, 1)$-linked, $s_{p_1}\ne s_{h_i}$
		for each $i\in [1,t]$ implies that $\lambda_2 \not\in\{\delta_1,\ldots, \delta_t\}$.
		Consequently, $s_{p_2}\ne s_{h_i}$ for each $i\in[1,t]$. Repeating the same process,
		we get  $s_{p_j}\ne s_{h_i}$ for each $j\in [1,k]$ and each $i\in[1,t]$.  
		
		We may still assume that $s_{h_t}$ and $x$ are $(\alpha+1,1)$-linked with respect to $\varphi$. 
		For otherwise, first doing an $(\alpha+1,1)$-swap at $s_{h_t}$, then shifting from $s_{h_1}$
		to $s_{h_t}$ give a coloring $\varphi'$  such that $\varphi'(ru)=\varphi(ru)=\alpha+1$, $\pbar'(y)=\pbar'(r)=\delta_1$, while $\pbar'(x)=\alpha+1$.
		As $\varphi'$ is $F$-stable up to exchanging the role of 1 and $\delta_1$, we obtain  
		a contradiction to Claim~\ref{u-2-out-neighbor}. 
		
		If $\pbar(s_{p_k})=\lambda_1$, then since $s_{p_k}$ and $r$ are $(\lambda_1,1)$-linked, 
		a $(\lambda_1,1)$-swap at $z$ gives a contradiction to Claim~\ref{u-2-out-neighbor}. 
		Thus  $\pbar(s_{p_k})=\alpha+1$. 
	We do a sequence of Kempe changes around $r$ from $s_{p_{k}}$ to $s_{p_1}$ as below:
		\begin{enumerate}[(1)]
			\item Swap colors along $P_{s_{p_{k}}}(\alpha+1,1)$ (after (1),  $P_r(1, \lambda_{k})=rs_{{p_{k}}}$);
			\item Swap colors along $P_{s_{p_{k-1}}}(\lambda_{k-1},1)$ (after (2), $P_r(1, \lambda_{k-1})=rs_{{p_{k-1}}}$);
			\item Continue the same kind of Kemple change from $s_{p_{k-2}}$ to $s_{p_3}$; 
			\item Swap colors along $P_{s_{p_{2}}}(\lambda_{3},1)$ (after (4), $P_r(1, \lambda_{2})=rs_{{p_{2}}}$);
			\item Swap colors along $P_{s_{p_{1}}}(\lambda_{2},1)$ (after (5), $P_r(1, \lambda_{1})=rs_{{p_{1}}}$). 
		\end{enumerate}
		Let the current  coloring  be $\varphi'$. 
		Note that $\varphi'$ is $F$-stable, $\varphi'(ru)=\varphi(ru)$,
		$\varphi'(ux)=\varphi(ux)$,  and $\pbar'(x)=\pbar(x)$, 
		but $z$ and $r$ are $(\lambda,1)$-unlinked with respect to $\varphi'$. 
		Now doing a  $(\lambda,1)$-swap at $z$ gives a contradiction to Claim~\ref{u-2-out-neighbor}.
		
		This finishes the proof of Theorem~\ref{Thm:vizing-fan}. 
\end{proof}

\section{Proof of Theorem~\ref{Thm:adj_small_vertex}}

\begin{THM2}
If $G$ is an HZ-graph with maximum degree $\Delta\ge 4$,  then for every two adjacent vertices $x, y\in V_{\Delta-1}$, $N_\Delta(x)=N_\Delta(y)$.
\end{THM2}
\begin{proof}
Assume to the contrary that  $N_\Delta(x)\ne N_\Delta(y)$. Then there exists a vertex $r\in N_\Delta(x)\setminus N_\Delta(y)$. 
	Equivalently,  $x\in N_{\Delta-1}(r)$ and  $y\not\in N_{\Delta-1}(r)$. 
			By Theorem~\ref{Thm:vizing-fan2b} \eqref{ele2},  let $s_1\in N_{\Delta-1}(r)$
	and $\varphi\in \CC^\Delta(G-rs_1)$, and  $F=F_\varphi(r,s_1: s_\alpha)$ be the typical 2-inducing multifan 
		such that either $V(F)=N_{\Delta-1}[r]$ or  $F$ is contained in
a pseudo-multifan with vertex set  $N_{\Delta-1}[r]$. 
		Let $ N_{\Delta-1}(r)=\{s_1,\ldots, s_{\Delta-2}\}$.
	We consider two cases according to  if $x\in V(F)$
	to finish the proof. 
	
	\medskip 
	
	Assume first that  $x\notin V(F)$. This implies that $V(F)\ne N_{\Delta-1}[r]$. 
	Applying  Theorem~\ref{Thm:vizing-fan2b} \eqref{ele2}, it then follows that  
	$N_{\Delta-1}[r]$ is  the vertex set of a typical 2-inducing pseudo-multifan.  Let $\pbar(x)=\delta$ and $\pbar(y)=\lambda$. 
	By Lemma~\ref{thm:vizing-fan1} \eqref{thm:vizing-fan1b}
	or Lemma~\ref{pseudo-fan-ele} \eqref{pseudo-a1}, we know
	that $\pbar^{-1}_S(\lambda)$ and $r$ are $(\lambda,1)$-linked, and $x$ and $r$ are $(\delta,1)$-linked. 
	By doing  $(\lambda,1)-(1,\delta)$-swaps at $y$ if necessary, we can assume  $\pbar(y)=\delta$. 
	Let $\varphi(xy)=\tau$.  Then $P_x(\delta,\tau)=xy$, showing a contradiction to 
	Lemma~\ref{pseudo-fan-ele}~\eqref{pseudo-b} or \eqref{pseudo-c}. 
	
	\medskip 
	
	Assume then that $x\in V(F)$.
%
	Let $x=s_i$ for some $i\in[1,\alpha]$, and $\varphi'$ be 
	obtained from $\varphi$  by shifting from $s_2$ to $s_{i-1}$, uncoloring $rs_i$, 
	and coloring $rs_1$ by 2.   
	The sequence 
	  $$F^*=(r,rs_i, s_i,rs_{i+1}, s_{i+1}, \ldots, rs_\alpha, s_\alpha, rs_{i-1}, s_{i-1}, \ldots, rs_1,s_1)$$ is  a multifan with respect to $\varphi'$.   It is clear that if $F$ is a multifan on $N_{\Delta-1}[r]$, then $F^*$ is 
		still a multifan on $N_{\Delta-1}[r]$.
		If  $F$ is contained in a pseudo-multifan on  $N_{\Delta-1}[r]$,  by Lemma~\ref{pseudo-fan-ele:e},
	after the operations from $F$ to get $F^*$, the resulting  
	sequence of the original pseudo-multifan is still a pseudo-multifan
	that contains $F^*$. 
	


By permuting the name of 
	colors and the label of the vertices in $N_{\Delta-1}(r)$, 
	we may assume that  $x=s_1$, and $F$ is a  typical multifan on  $N_{\Delta-1}[r]$  or is contained in a  pseudo-multifan  $S$  with   $V(S)=N_{\Delta-1}[r]$. 
Still denote the current coloring by $\varphi$.

	Let $\pbar(y)=\delta$. By doing a $(\delta,1)$-swap at $y$ if necessary, 
	we assume  $\pbar(y)=1$. 
	Let 	$\varphi(s_1y)=\tau$.  By exchangeing the role of the color 2 and $\Delta$ 
	if necessary, we may assume that 
	$\varphi(s_1y)$ is a 2-inducing color of $F$ or is a pseudo-missing color of the pseudo-multifan $S$. Let  $\varphi'=\varphi/P_y(1,\Delta,\varphi)$. 
	Now $P_{s_1}(\Delta,\tau, \varphi')=s_1y$.
		This gives a contradiction to Lemma~\ref{thm:vizing-fan2}~\eqref{thm:vizing-fan2-b}
	that $s_1$ and $\pbar'^{-1}_F(\tau)$  are  $(\tau, \Delta)$-linked if  $\tau$ is 2-inducing, 
	and gives a contradiction to Lemma~\ref{pseudo-fan-ele}~\eqref{pseudo-b} that $s_1$ and $\pbar'^{-1}_S(\tau)$
	are $(\tau,\Delta)$-linked if 
	$\tau$ is a pseudo-missing color of $S$. 
\end{proof}

\section{Proof of Theorem~\ref{Thm:nonadj_Delta_vertex}}

\begin{THM3}
	Let $G$ be an HZ-graph with maximum degree $\Delta\ge 7$ and $u, r\in V_{\Delta}$.
	If $N_{\Delta-1}(u)\ne N_{\Delta-1}(r)$ and $N_{\Delta-1}(u)\cap  N_{\Delta-1}(r)\ne \emptyset$,
	then $|N_{\Delta-1}(u)\cap  N_{\Delta-1}(r)|=\Delta-3$, i.e. $|N_{\Delta-1}(u)\setminus N_{\Delta-1}(r)|=|N_{\Delta-1}(r)\setminus  N_{\Delta-1}(u)|=1$. 
\end{THM3}	

\begin{proof}
Assume to the contrary that there exist $u,r\in N_{\Delta}$
such that $1\le |N_{\Delta-1}(r)\cap N_{\Delta-1}(u)|\le \Delta-4$. 
By Theorem~\ref{Thm:vizing-fan2b} \eqref{ele2}, there exist $s_1\in N_{\Delta-1}(r)$
and $\varphi\in \CC^\Delta(G-rs_1)$ such that $N_{\Delta-1}[r]$
is  the vertex set of a  typical 2-inducing multifan or a typical  2-inducing pseudo-multifan. 
Let $ N_{\Delta-1}(r)=\{s_1,\ldots, s_{\Delta-2}\}$ and 
 $x,y\in N_{\Delta-1}(u)\setminus N_{\Delta-1}(r)$ be two distinct vertices. 
 We consider two cases. 

\medskip 
\noindent{\bf Case 1}: $N_{\Delta-1}[r]$ is the vertex set of a typical 2-inducing pseudo-multifan. 

Let $S=S_\varphi(r, s_1:s_\alpha: s_{\Delta-2})$ be this pseudo-multifan with $F_\varphi(r,s_1:s_\alpha)$ being the typical 2-inducing multifan  contained in $S$.
We consider two subcases that each leads to a contradiction. 

{\it \noindent Subcase 1.1: There exists $s_i\in N_{\Delta-1}(u)\cap N_{\Delta-1}(r)$ for some $i\in [1,\alpha]$.}

By shifting from $s_2$
to $s_{i-1}$, uncoloring  $rs_i$, and coloring  $rs_1$ by 2, 
we obtain a new multifan $F^*=(r,rs_i,s_i, rs_{i+1}, s_{i+1}, \ldots, rs_\alpha, s_\alpha, rs_{i-1}, s_{i-1}, \ldots, rs_1,s_1)$. 
By  permuting the name of 
colors and the label of the vertices in $N_{\Delta-1}(r)$ such that $i+1$ is permuted to $2$
and $s_i$ is renamed as $s_1$, 
we  assume that  $s_1\in N_{\Delta-1}(u)\cap N_{\Delta-1}(r)$ and $F^*$ is a  typical multifan. 

Recall that $x\in N_{\Delta-1}(u)\setminus N_{\Delta-1}(r)$.
Let $\pbar(x)=\lambda$.  By Lemma~\ref{thm:vizing-fan1} \eqref{thm:vizing-fan1b}
or Lemma~\ref{pseudo-fan-ele} \eqref{pseudo-a1}, we know
that $\pbar^{-1}_S(\lambda)$ and $r$ are $(\lambda,1)$-linked. 
By doing a $(\lambda,1)$-swap at $x$ if necessary, we assume 
$\pbar(x)=1$.  By exchanging the role of  the colors 2 and $\Delta$ if necessary, we 
assume that $\varphi(s_1u)$  equals 1, or is a 2-inducing color of $F$, or is a pseudo-missing color 
of $S$. Note that by Lemma~\ref{pseudo-fan-ele}~\eqref{pseudo-b}, 
for a pseudo-missing color $\delta$
of $S$, and for any color $\tau\in \pbar(V(F))$, 
$\pbar_S^{-1}(\delta)$ and $\pbar_S^{-1}(\tau)$
are $(\delta, \tau)$-linked and $r\in P_{\pbar_S^{-1}(\delta)}(\delta,\tau)$.

Let $\varphi(ux)=\tau$. If $\tau$ is a 2-inducing color of $F$ or  a pseudo-missing color of $S$, we  do  $(1,\Delta)-(\Delta,\tau)-(\tau,1)$-swaps at $x$. If $\tau$ is a $\Delta$-inducing color of $F$, let $\delta\in \pbar(S)$ be a pseudo-missing color,   we  do  $(1,\delta)-(\delta,\tau)-(\tau,1)-(1,\Delta)-(\Delta,\delta)-(\delta,1)$-swaps at $x$. In both cases, 
we let $\varphi'$ be the resulting  coloring. Clearly, $\varphi'(ux)=\Delta$ 
and $\pbar'(x)=1$. 
Since $\varphi(s_1u) \ne \Delta, \tau$,  still  $\varphi'(s_1u)$ 
equals 1, or is a 2-inducing color of $F$,  or is a pseudo-missing color 
of $S$.   

Let $\varphi'(s_1u)=\gamma$. 
 Since $s_1$ and $r$ are $(\Delta,1)$-linked with respect to $\varphi'$, 
 $\gamma\ne 1$.  Thus,   $\gamma$  is a 2-inducing color of $F$,  or is a pseudo-missing color 
of $S$.  By Lemma~\ref{thm:vizing-fan2} \eqref{thm:vizing-fan2-a} or Lemma~\ref{pseudo-fan-ele}~\eqref{pseudo-b},  $u\in P_x(1,\gamma,\varphi')$. 
We then do a $(1,\gamma)$-swap at $x$. Now $K=(r, rs_1,s_1, s_1u,  u, ux, x)$ is a Kierstead path with respect to $rs_1$ and the current coloring. Let $\delta\in \pbar(S)$ be a pseudo-missing color.  If $\gamma$ is  a pseudo-missing color, we do nothing. Otherwise, 
we do a $(\gamma,\delta)$-swap at $x$ (by Lemma~\ref{pseudo-fan-ele}~\eqref{pseudo-b}, this swap does not change the coloring of $S$).  Denote by $\varphi''$ the current coloring. 
Since $d_G(s_1)=\Delta-1$,  in both cases, 
by Lemma~\ref{Lemma:kierstead path1}, $x$ and $s_1$ are  $(\pbar''(x),2)$-linked.
Since $\pbar''(x)$ is a pseudo-missing color of $S$, 
we achieve a contradiction to Lemma~\ref{pseudo-fan-ele}~\eqref{pseudo-b}. 

{\it \noindent Subcase 1.2: For each $i\in [1,\alpha]$,  $s_i\notin N_{\Delta-1}(u)\cap N_{\Delta-1}(r)$.}

In this case, there exists $s_{h_1}\in \{s_{\alpha+1}, \ldots, s_{\Delta-2}\}$
such that $s_{h_1}\in N_{\Delta-1}(u)\cap N_{\Delta-1}(r)$. 
Let $\varphi(rs_{h_1})=\delta_1, \pbar(s_{h_1})=\delta_2$. 
We claim $\varphi(ux)=\delta_2$. Otherwise, let $\varphi(ux)=\delta^*\ne \delta_2$. By Lemma~\ref{pseudo-fan-ele}~\eqref{pseudo-b} and \eqref{pseudo-c}, 
we  assume  $\pbar(x)=\delta_2$. 
Then a $(\delta_2, \delta^*)$-swap at $x$ gives  $\varphi(ux)=\delta_2$. 
Let $\varphi(s_{h_1}u)=\tau$. 
Again by Lemma~\ref{pseudo-fan-ele}~\eqref{pseudo-b} and \eqref{pseudo-c},
we may first assume  $\pbar(x)=\delta_1$, and then by doing a $(\delta_1,\tau)$-swap at $x$, we assume $\pbar(x)=\tau$.
However, with respect to the current coloring,  this implies  $P_{s_{h_1}}(\delta_2, \tau)=s_{h_1}ux=P_x(\delta_2, \tau)$, showing a contradiction to 
Lemma~\ref{pseudo-fan-ele}~\eqref{pseudo-a1}, \eqref{pseudo-b} or \eqref{pseudo-c}
that $s_{h_1}$ and $\pbar^{-1}_S(\tau)$ are $(\delta_2,\tau)$-linked. 
This finishes the proof for Subcase 1.2.

\medskip 
\noindent{\bf Case 2}: $N_{\Delta-1}[r]$  is the vertex set of a  typical 2-inducing multifan. 

We may assume that  $N_{\Delta-1}[r]$  is the vertex set of a  typical  multifan with respect to 
$\varphi$ and $rs_1$.  
 If $s_1\in N_{\Delta-1}(u)\cap N_{\Delta-1}(r)$, 
then we are done. Otherwise, let $s_i\in N_{\Delta-1}(u)\cap N_{\Delta-1}(r)$. 
Thus, by shifting from $s_2$
to $s_{i-1}$,  uncoloring  $rs_i$, coloring  $rs_1$ by 2,   and permuting the name of 
colors and the label of the vertices in $N_{\Delta-1}(r)$, 
we  assume that  $s_1\in N_{\Delta-1}(u)\cap N_{\Delta-1}(r)$, and the resulting multifan $$F^*=(r, rs_i, s_i, rs_{i+1}, s_{i+1}, \ldots, rs_\alpha, s_\alpha, rs_{i-1}, s_{i-1}, \ldots, rs_1, s_1)$$ is a  typical multifan.  
We  let $F_\varphi(r,s_1: s_\alpha:s_{\Delta-2})$ be such a typical multifan.

\begin{CLA}\label{x-y-missing-colors}
We may assume that $\pbar(x)=2$ and $\pbar(y)=\Delta$ 
or $\pbar(x)=\pbar(y)=\Delta$. 
\end{CLA}

\proof[Proof of Claim~\ref{x-y-missing-colors}]

Let $\pbar(x)=\delta$. By doing  $(\delta,1)-(1,2)$-swaps at $x$ if necessary, 
we can assume  $\pbar(x)=2$. 
Now, let $\pbar(y)=\lambda$. If $\lambda=2$, then doing $(2,1)-(1,\Delta)$-swaps at both $x$
and $y$ gives  $\pbar(x)=\pbar(y)=\Delta$. Thus, we assume  $\lambda\ne 2$.
Then doing $(\lambda,1)-(1,\Delta)$-swaps at $y$ gives  $\pbar(y)=\Delta$. 
\qed 

\medskip 
By Claim~\ref{x-y-missing-colors}, we naturally consider two subcases below. 

\medskip 
\emph{\noindent \setword{Subcase 2.1}{Case 1}: $\pbar(x)=2$ and $\pbar(y)=\Delta$.}

\medskip 
Note that by doing first a $(2,1)$-swap at $x$, then a $(1,\Delta)$-swap at both $x$ and $y$, and finally a 
$(1,2)$-swap at $y$, we can always identify this current case with the case that $\pbar(x)=\Delta$ and $\pbar(y)=2$.
Let $\varphi(ux)=\tau$ and $\varphi(uy)=\lambda$.  By exchanging the role of 
the two colors 2 and $\Delta$,  we consider two cases below:
\begin{enumerate}[(A)]
	\item  $\varphi(uy)=\lambda=1$. 
	\item $\varphi(uy)=\lambda$ is 2-inducing. 
\end{enumerate}
(When $\varphi(uy)$ is $\Delta$-inducing, by  assuming  $\pbar(x)=\Delta$ and $\pbar(y)=2$,   the argument will be symmetric to the argument for the above cases.)

In both cases of (A) and (B), we do  $(\Delta,\lambda)-(\lambda,1)$-swaps at $y$ and still call the resulting coloring $\varphi$. 
Let $\varphi(s_1u)=\delta$.   Denote by $S(u;s_1, x,y)$ the star subgraph of $G$
that is centered at $u$ consisting of edges $us_1, ux, uy$. 
The current coloring on $S(u; s_1, x, y)$ is as shown in $J_1$  of  Figure~\ref{J1}.  
We  modify the current coloring so that the color on 
 $S(u; s_1, x, y)$  is  as shown in $J_2$ of Figure~\ref{J1}. 

\begin{figure}[!htb]
	\begin{center}
		\begin{tikzpicture}[scale=1]
		
		{\tikzstyle{every node}=[draw ,circle,fill=white, minimum size=0.5cm,
			inner sep=0pt]
			\draw[blue,thick](0,-3) node (s1)  {$s_1$};
			\draw [blue,thick](0, -5) node (u)  {$u$};
			\draw [blue,thick](-1, -7) node (x)  {$x$};
			\draw [blue,thick](1, -7) node (y)  {$y$};
		}
		\path[draw,thick,black!60!green]
		
		(s1) edge node[name=la,pos=0.4] {\color{blue}\,\,\, $\delta$} (u)
		(u) edge node[name=la,pos=0.4] {\color{blue}$\tau$\quad\quad} (x)
		(u) edge node[name=la,pos=0.4] {\color{blue}\,\,\,\,\,\,\, $\Delta$} (y);
		
		
		\draw[dashed, red, line width=0.5mm] (s1)--++(140:1cm);
		\draw[dashed, red, line width=0.5mm] (s1)--++(40:1cm); 
				\draw[dashed, red, line width=0.5mm] (x)--++(200:1cm); 
				\draw[dashed, red, line width=0.5mm] (y)--++(340:1cm);
		\draw[blue] (-0.6, -2.9) node {$2$};  
		\draw[blue] (0.6, -2.9) node {$\Delta$};  
		\draw[blue] (-1.5, -7.4) node {$2$}; 
		\draw[blue] (1.5, -7.4) node {$1$};
		
		\draw[blue] (0, -8.5) node {$J_1$}; 
		
		\begin{scope}[shift={(6,0)}]

		{\tikzstyle{every node}=[draw ,circle,fill=white, minimum size=0.5cm,
			inner sep=0pt]
			\draw[blue,thick](0,-3) node (s1)  {$s_1$};
			\draw [blue,thick](0, -5) node (u)  {$u$};
			\draw [blue,thick](-1, -7) node (x)  {$x$};
			\draw [blue,thick](1, -7) node (y)  {$y$};
		}
		\path[draw,thick,black!60!green]
		
		(s1) edge node[name=la,pos=0.4] {\color{blue}\,\,\, $1$} (u)
		(u) edge node[name=la,pos=0.4] {\color{blue}$\tau$\quad\quad} (x)
		(u) edge node[name=la,pos=0.4] {\color{blue}\,\,\,\,\,\,\, $\Delta$} (y);
		
		
		\draw[dashed, red, line width=0.5mm] (s1)--++(140:1cm);
		\draw[dashed, red, line width=0.5mm] (s1)--++(40:1cm); 
		\draw[dashed, red, line width=0.5mm] (x)--++(200:1cm); 
		\draw[dashed, red, line width=0.5mm] (y)--++(340:1cm);
		\draw[blue] (-0.6, -2.9) node {$2$};  
		\draw[blue] (0.6, -2.9) node {$\Delta$};  
		\draw[blue] (-1.5, -7.4) node {$2$}; 
		\draw[blue] (1.5, -7.4) node {$\delta$};
			\draw[blue] (0, -8.5) node {$J_2$}; 
		\end{scope}
		\end{tikzpicture}
	\end{center}
	\caption{Coloring of $S(u; s_1, x, y)$}
	\label{J1}
\end{figure}
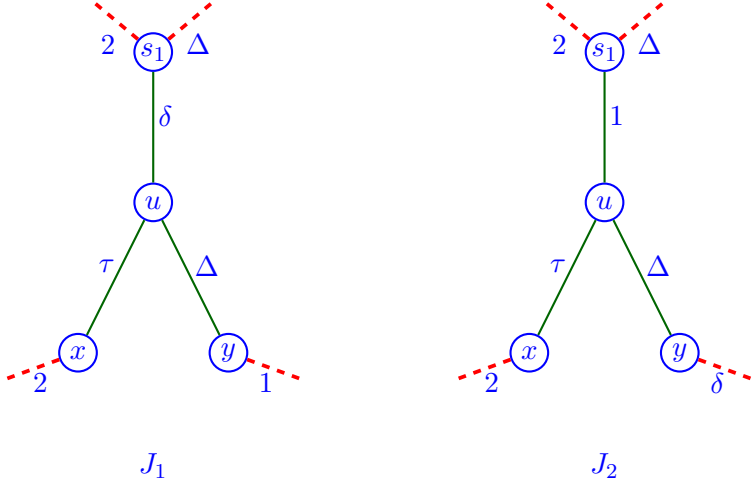


Since $s_1$ and $r$ are $(\Delta,1)$-linked by Lemma~\ref{thm:vizing-fan1}~\eqref{thm:vizing-fan1b}, 
we know  $\varphi(s_1u)=\delta\ne 1$. 
If $u\in P_y(1,\delta)$, then a $(1,\delta)$-swap at $y$ gives $J_2$. Thus, we assume 
$u\not\in P_y(1,\delta)$. This implies that  $\delta$ is $\Delta$-inducing. (For otherwise, a $(1,\delta)$-swap at $y$ implies that $s_1$ and $\pbar^{-1}_F(\delta)$ are $(\Delta, \delta)$-unlinked, showing a contradiction to Lemma~\ref{thm:vizing-fan2} \eqref{thm:vizing-fan2-a}.)
We first do a $(1,\delta)$-swap at $y$, and call the resulting coloring $\varphi$. 

If $\varphi(ux)=\tau$   is  1 or is $\Delta$-inducing, we do  $(2,\tau)-(\tau,1)$-swaps at $x$.
Again, as $P_{s_1}(\Delta,\delta)=s_1uy$, $\delta$ is still a $\Delta$-inducing color by Lemma~\ref{thm:vizing-fan2} \eqref{thm:vizing-fan2-a}. 
Since $\varphi(ux)=2$, we know $u\in P_x(1,\delta)$. 
Since otherwise, a $(1,\delta)$-swap at $x$ implies  $P_{s_1}(\delta,2)=s_1ux$,
contradicting Lemma~\ref{thm:vizing-fan2}~\eqref{thm:vizing-fan2-a} that $s_1$
and $\pbar^{-1}_F(\delta)$ are $(2, \delta)$-linked with respect to the current coloring. 
Now, let $\varphi'$ be obtained from  $\varphi$ by doing a $(1,\delta)$-swap at both $x$
and $y$. We get  $P_{s_1}(\Delta,1,\varphi')=s_1uy$, showing a contradiction to 
Lemma~\ref{thm:vizing-fan1}~\eqref{thm:vizing-fan1b} that $s_1$ and $r$
are $(\Delta,1)$-linked with respect to $\varphi'$. 
Thus   $\varphi(ux)=\tau$  and $\tau$ is 2-inducing.  Based on the coloring shown in  
$J_1$ of Figure~\ref{J1} after a $(1,\delta)$-swap at $y$, 
we do the following swaps of colors to get the coloring shown in $J_2$ of  Figure~\ref{J1}:
\begin{enumerate}[(1)]
	\item $(2,1)$-swap at $x$. (Note that $u\not\in P_x(1,\delta)$ and $u\not\in P_y(\delta,1)$ after (1). 
	Since otherwise, a $(1,\delta)$-swap at both $x$ and $y$ gives  $P_{s_1}(\Delta,1)=s_1uy$,  showing a contradiction to 
	Lemma~\ref{thm:vizing-fan1}~\eqref{thm:vizing-fan1b} that $s_1$ and $r$
	are $(\Delta,1)$-linked.)
	\item $(1,\delta)$-swap at both $x$ and $y$. (Since $\delta$ is $\Delta$-inducing and $\tau$
	is 2-inducing,  $\pbar^{-1}_F(\delta)$ and $\pbar^{-1}_F(\tau)$
	are $(\delta,\tau)$-linked  by Lemma~\ref{thm:vizing-fan2} \eqref{thm:vizing-fan2-a}.)
	\item $(\tau,\delta)$-swap at both $\pbar^{-1}_F(\delta)$ and $\pbar^{-1}_F(\tau)$. (Now, $\delta$
	is a 2-inducing color. As a consequence, $u\in P_y(1,\delta)$.  Since otherwise, a $(1,\delta)$-swap at $y$ implies  $P_{s_1}(\Delta,\delta)=s_1uy$,
	contradicting Lemma~\ref{thm:vizing-fan2}~\eqref{thm:vizing-fan2-a} that $s_1$
	and $\pbar^{-1}_F(\delta)$ are $(\Delta, \delta)$-linked. )
	\item $(1,\delta)$-swap at $x$, $y$ (and $u$).
	\item $(1,2)$-swap at $x$. 
\end{enumerate}

Now $K=(r,rs_1,s_1,s_u, u,uy,y)$ is a Kierstead path with respect to $rs_1$ 
and the current coloring. 
Since $d_G(s_1)=\Delta-1$, by Lemma~\ref{Lemma:kierstead path1},  $y$ and $s_1$ are $(\delta,2)$-linked. 
This implies that 
$\pbar(y)=\delta$ is a 2-inducing color of $F$, as otherwise, 
$s_1$ and $\pbar_F^{-1}(\delta)$ should be $(\delta,2)$-linked. 
If $\varphi(ux)=\tau$ is $\Delta$-inducing, then as $\pbar^{-1}_F(\delta)$ and $\pbar^{-1}_F(\tau)$
are $(\delta,\tau)$-linked by Lemma~\ref{thm:vizing-fan2} \eqref{thm:vizing-fan2-a}, 
we can do a $(\delta,\tau)$-swap at $y$. 
We will then reach a contradiction as $\pbar^{-1}_F(\tau)$ and $s_1$
are $(\tau,2)$-linked. 
Therefore, $\varphi(ux)=\tau$ is $2$-inducing. We first do $(2,1)-(1,\Delta)$
swaps at $x$.  At this step, $\varphi(s_1u)=1$, $\varphi(uy)=\Delta$, 
 $\pbar(y)=\delta$, and we still have the fact that 
$y$ and $s_1$ are $(\delta,2)$-linked by Lemma~\ref{Lemma:kierstead path1}.  
Call this fact $(*)$.

Now, we do a $(\Delta,\tau)$-swap at $x$. 
Note that  $s_1$ and $r$
are $(\Delta,1)$-linked by Lemma~\ref{thm:vizing-fan1}~\eqref{thm:vizing-fan1b}. 
Thus, $u\in P_x(1,\tau)$.  We do a $(\tau,1)$-swap at $x$ (and $(u)$). 
The coloring of $S(u;s_1,x,y)$
is now shown in Figure~\ref{J3} $J_3$.
Since $2,\delta \not\in \{\Delta,\tau,1\}$, $y$ and $s_1$ are still $(\delta,2)$-linked with respect to 
the current coloring by fact $(*)$. We consider two cases to finish the remaining part of the proof. 

\begin{figure}[!htb]
	\begin{center}
		\begin{tikzpicture}[scale=1]
		
		{\tikzstyle{every node}=[draw ,circle,fill=white, minimum size=0.5cm,
			inner sep=0pt]
			\draw[blue,thick](0,-3) node (s1)  {$s_1$};
			\draw [blue,thick](0, -5) node (u)  {$u$};
			\draw [blue,thick](-1, -7) node (x)  {$x$};
			\draw [blue,thick](1, -7) node (y)  {$y$};
		}
		\path[draw,thick,black!60!green]
		
		(s1) edge node[name=la,pos=0.4] {\color{blue}\,\,\, $\tau$} (u)
		(u) edge node[name=la,pos=0.4] {\color{blue}$\Delta$\quad\quad\,} (x)
		(u) edge node[name=la,pos=0.4] {\color{blue}\,\,\,\,\,\,\, $1$} (y);
		
		
		\draw[dashed, red, line width=0.5mm] (s1)--++(140:1cm);
		\draw[dashed, red, line width=0.5mm] (s1)--++(40:1cm); 
		\draw[dashed, red, line width=0.5mm] (x)--++(200:1cm); 
		\draw[dashed, red, line width=0.5mm] (y)--++(340:1cm);
		\draw[blue] (-0.6, -2.9) node {$2$};  
		\draw[blue] (0.6, -2.9) node {$\Delta$};  
		\draw[blue] (-1.5, -7.4) node {$1$}; 
		\draw[blue] (1.5, -7.4) node {$\delta$};
		
		\draw[blue] (0, -8.5) node {$J_3$}; 
		
		\begin{scope}[shift={(6,0)}]

		{\tikzstyle{every node}=[draw ,circle,fill=white, minimum size=0.5cm,
			inner sep=0pt]
			\draw[blue,thick](0,-3) node (s1)  {$s_1$};
			\draw [blue,thick](0, -5) node (u)  {$u$};
			\draw [blue,thick](-1, -7) node (x)  {$x$};
			\draw [blue,thick](1, -7) node (y)  {$y$};
		}
		\path[draw,thick,black!60!green]
		
		(s1) edge node[name=la,pos=0.4] {\color{blue}\,\,\, $\delta$} (u)
		(u) edge node[name=la,pos=0.4] {\color{blue}$\Delta$\quad\quad\,} (x)
		(u) edge node[name=la,pos=0.4] {\color{blue}\,\,\,\,\,\,\, $1$} (y);
		
		
		\draw[dashed, red, line width=0.5mm] (s1)--++(140:1cm);
		\draw[dashed, red, line width=0.5mm] (s1)--++(40:1cm); 
		\draw[dashed, red, line width=0.5mm] (x)--++(200:1cm); 
		\draw[dashed, red, line width=0.5mm] (y)--++(340:1cm);
		\draw[blue] (-0.6, -2.9) node {$2$};  
		\draw[blue] (0.6, -2.9) node {$\Delta$};  
		\draw[blue] (-1.5, -7.4) node {$1$}; 
		\draw[blue] (1.5, -7.4) node {$\tau$};
		\draw[blue] (0, -8.5) node {$J_4$}; 
		\end{scope}
		\end{tikzpicture}
	\end{center}
	\caption{Coloring of $S(u; s_1, x, y)$}
	\label{J3}
\end{figure}
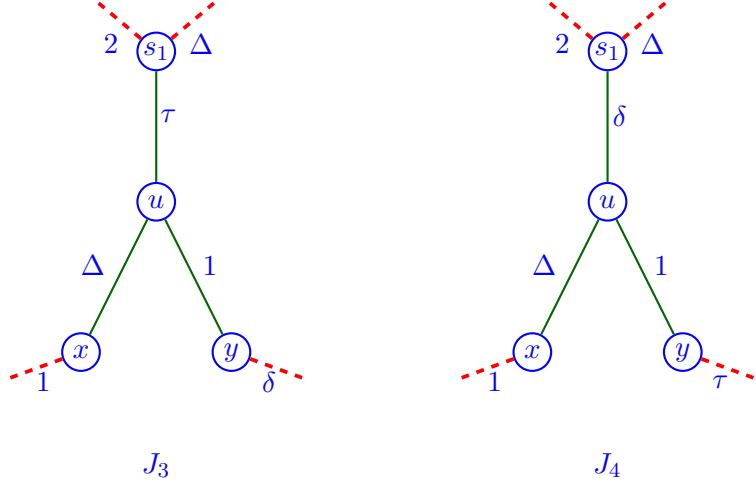

\medskip 
\emph{\noindent \setword{Subcase 2.1.1}{Case 1.1} : $\tau\prec \delta$.}
\medskip 

Let $s_{\delta-1}\in N_{\Delta-1}(r)$ such that $\pbar(s_{\delta-1})=\delta$.
Since $y$ and $s_1$ are still $(\delta,2)$-linked with respect to 
the current coloring and $\delta$ is 2-inducing, 
by Lemma~\ref{thm:vizing-fan2} \eqref{thm:vizing-fan2-b}, 
$r\in P_{s_{\delta-1}}(\delta,2)$.  We reach a contradiction  through the following Kempe changes: 
\begin{enumerate}[(1)]
	\item $(\delta,2)$-swap at $s_{\delta-1}$ (and $r$). 
	\item $(2,1)$-swap at $x$ and $s_{\delta-1}$. 
	\item Shift from $s_\tau$ to $s_{\delta-1}$, where $\varphi(rs_\tau)=\tau$. 
\end{enumerate}
Denote the new coloring by $\varphi'$. 
Now, $\pbar'(r)=\tau$, $\varphi'(s_1u)=\tau$, $\varphi'(ux)=\Delta$, and $\pbar'(x)=2$,
and $K=(r,rs_1,s_1,s_1u, u, ux,x)$ is a Kierstead path with respect to $rs_1$
and $\varphi'$. Since $d_G(s_1)=\Delta-1$, 
we get  a contradiction  to Lemma~\ref{Lemma:kierstead path1} that $\{r,s_1,u,x\}$ is $\varphi'$-elementary.

\medskip 
\emph{\noindent Subcase 2.1.2: $\delta\prec \tau$.}
\medskip 

We only show that by performing Kempe changes, we can find a coloring  such that the color on $S(u;s_1,x,y)$ with respect to the coloring 
is as given in Figure~\ref{J3} $J_4$.  The remaining part of the proof will be symmetric  to 
~\ref{Case 1.1}. 
Based on the coloring in Figure~\ref{J3} $J_3$,  do a $(1,\delta)$-swap at both $x$
and $y$. 

If  $u\in P_y(1,\tau)$,  do a $(1,\tau)$-swap at $y$ (and $u$), and still denote the resulting coloring by $\varphi$. 
Note that $u\in P_x(\delta,1)$ (as otherwise, a $(\delta,1)$-swap at $x$ implies that $P_{s_1}(\Delta,1)=s_1ux$,  showing a contradiction to Lemma~\ref{thm:vizing-fan1}~\eqref{thm:vizing-fan1b} that $s_1$
and $r$ are $(\Delta,1)$-linked). Then a $(\delta,1)$-swap at $x$ (and $u$) gives Figure~\ref{J3} $J_4$. 

 Thus $u\not\in P_y(1,\tau)$.
Under this assumption, it must be the case that  
  $u\in P_r(1,\tau)$ (otherwise, performing a $(\delta,\Delta)$-swap at $x$ and a $(\tau,1)$-swap at $u$ gives  $P_{s_1}(1,\Delta)=s_1uy$, showing a contradiction to Lemma~\ref{thm:vizing-fan1}~\eqref{thm:vizing-fan1b} that $s_1$
and $r$ are $(\Delta,1)$-linked).  
Since $u\in P_r(1,\tau)$, 
doing a $(1,\tau)$-swap at $y$ and $(\Delta,\delta)$-swap at $x$  gives  $P_{s_1}(\Delta,\tau)=s_1uy$, 
implying that $\tau$ is a $\Delta$-inducing color of the current multifan by Lemma~\ref{thm:vizing-fan2} \eqref{thm:vizing-fan2-a}. 
Note that  $\delta$ is still a 2-inducing color of the current multifan. 
Thus, $\pbar^{-1}_F(\delta)$ and $\pbar^{-1}_F(\tau)$
are $(\delta,\tau)$-linked by Lemma~\ref{thm:vizing-fan2} \eqref{thm:vizing-fan2-a}. 

Also, since $\tau$ is $\Delta$-inducing and $\delta$ is 2-inducing, 
we know  $u\not\in P_y(\tau,\delta)$. 
Since otherwise, a $(\tau,\delta)$-swap at $y$ gives $P_{s_1}(\Delta,\delta)=s_1uy$,
showing a contradiction to Lemma~\ref{thm:vizing-fan2} \eqref{thm:vizing-fan2-a}
that $s_1$ and $\pbar^{-1}_F(\delta)$ are $(\Delta,\delta)$-linked. 
We now reach a contradiction by performing the following operations: 
\begin{enumerate}[(1)]
	\item $(\tau, \delta)$-swap at $y$ (now $P_y(\delta, \Delta)=yux$). (Note that $u\in P_{\pbar^{-1}_F(\delta)}(\delta,\tau)=P_{\pbar^{-1}_F(\tau)}(\delta,\tau)$. For otherwise, a $(\tau,\delta)$-swap at $u$ gives $P_{s_1}(\Delta,\delta)=s_1uy$, 
	showing a contradiction to the fact that $\delta$ is still a 2-inducing color.)
	\item $(\Delta, \delta)$-swap along $xuy$.  
\end{enumerate}
After Step (2) above, we have that  $\pbar^{-1}_F(\delta)$ and $\pbar^{-1}_F(\tau)$
are $(\delta,\tau)$-unlinked. However, $\tau$ is still a $\Delta$-inducing color and $\delta$
is 2-inducing, showing a contradiction to Lemma~\ref{thm:vizing-fan2} \eqref{thm:vizing-fan2-a}.

\medskip 

\emph{\noindent Subcase 2.2: $\pbar(x)=\Delta$ and $\pbar(y)=\Delta$.}

\medskip

\begin{CLA}\label{only2-out-neighbors}
	We may assume that $|N_{\Delta-1}(u)\cap N_{\Delta-1}(r)|=\Delta-4$. 
\end{CLA}

	\proof[Proof of Claim~\ref{only2-out-neighbors}] 
	We may assume that $x$ and $y$ are $(\Delta,1)$-linked. For otherwise, performing $(\Delta,1)-(1,2)$-swaps at $x$ reduces the problem to \ref{Case 1}. 
	Assume to the contrary that $|N_{\Delta-1}(u)\cap N_{\Delta-1}(r)|\le \Delta-5$.
Then there exists $z\in N_{\Delta-1}(u)\setminus  N_{\Delta-1}(r)$ such that $z\ne x,y$. 
Let $\pbar(z)=\lambda$. If $\lambda=2$, 
by exchanging  the role of $x$ and $z$,  we reduce the problem to \ref{Case 1}. 
Thus, $\lambda \ne 2$.  Doing  $(\lambda,1)-(1,2)$-swaps at $z$. 
By exchanging  the role of $x$ and $z$,  we reduce the problem to \ref{Case 1}. 
\qed 

\begin{CLA}\label{2-sequence-fan}
	We  assume that $F(r,s_1:s_\alpha:s_{\Delta-2})$ is a typical multifan with  two sequences.  
	That is, $F$ contains both $2$-inducing sequence and $\Delta$-inducing sequence. 
\end{CLA}
\proof[Proof of Claim~\ref{2-sequence-fan}]
Recall that $F_\varphi(r,s_1: s_\alpha:s_{\Delta-2})$ is a typical multifan. 
As $\Delta\ge 7$, $|N_{\Delta-1}(u)\cap N_{\Delta-1}(r)|=\Delta-4\ge 3$ 
by Claim~\ref{only2-out-neighbors}. If $F$
is a typical 2-inducing  multifan, then let $s_i\in N_{\Delta-1}(u)\cap N_{\Delta-1}(r)$
such that $s_i\ne s_1$ and that  $\pbar(s_i)$ is not the last 2-inducing color of $F$. 
Then we shift from $s_2$ to $s_{i-1}$, uncolor $rs_i$, and color $rs_1$ by 2. 
Now $F^*=(r, rs_i, s_i, rs_{i+1}, s_{i+1}, \ldots, rs_{\Delta-2}, s_{\Delta-2}, rs_{i-1}, s_{i-1}, \ldots, rs_1, s_1)$
is a multifan with  two sequences.  By permuting the name of colors 
and the label of vertices in $\{s_1,\ldots, s_{\Delta-2}\}$, we can assume that 
$F=F^*$ is a typical multifan with  two sequences.  
\qed 

Let $\varphi(s_1u)=\delta$, $\varphi(ux)=\tau$,  and $\varphi(uy)=\lambda$.  By exchanging the role of 
the two colors 2 and $\Delta$,  we have two possibilities for $\varphi(uy)$:
\begin{enumerate}[(A)]
	\item  $\varphi(uy)=\lambda=1$. 
	\item $\varphi(uy)=\lambda$ is 2-inducing. 
\end{enumerate}
(When $\varphi(uy)$ is $\Delta$-inducing, we will first assume that $\pbar(x)=2$ and $\pbar(y)=2$ (by performing $(\Delta,1)-(1,2)$-swaps at both $x$ and $y$). Then all the argument will be symmetric to the argument for the above cases.) 
We now consider two cases to finish the proof. 

\medskip 
\emph{\noindent \setword{Subcase 2.2.1}{Case 2.1}: $\varphi(uy)=\lambda$ is not the last 2-inducing color of $F$.}
\medskip 

We first perform $(\Delta,\lambda)-(\lambda,1)$-swaps at both $x$ and $y$. 
Since $\lambda$ is not the last 2-inducing color, the resulting multifan still has  two sequences. 
The current coloring of $S(u;s_1,x,y)$ is given in Figure~\ref{L1} $L_1$. 
Since $s_1$ and $r$
are $(\Delta,1)$-linked by Lemma~\ref{thm:vizing-fan1}~\eqref{thm:vizing-fan1b}, $\delta\ne 1$. 
We next show $u\in P_y(1,\delta)$ that will lead to the coloring in Figure~\ref{L1} $L_2$. 

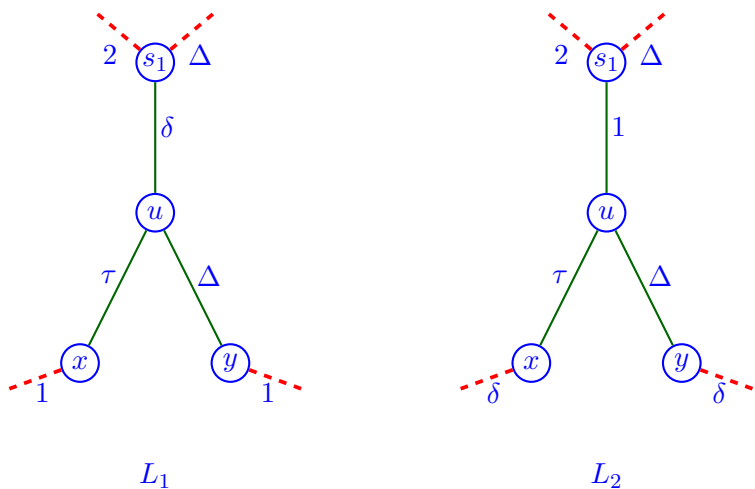
\begin{figure}[!htb]
	\begin{center}
		\begin{tikzpicture}[scale=1]
		
		{\tikzstyle{every node}=[draw ,circle,fill=white, minimum size=0.5cm,
			inner sep=0pt]
			\draw[blue,thick](0,-3) node (s1)  {$s_1$};
			\draw [blue,thick](0, -5) node (u)  {$u$};
			\draw [blue,thick](-1, -7) node (x)  {$x$};
			\draw [blue,thick](1, -7) node (y)  {$y$};
		}
		\path[draw,thick,black!60!green]
		
		(s1) edge node[name=la,pos=0.4] {\color{blue}\,\,\, $\delta$} (u)
		(u) edge node[name=la,pos=0.4] {\color{blue}$\tau$\quad\quad} (x)
		(u) edge node[name=la,pos=0.4] {\color{blue}\,\,\,\,\,\,\, $\Delta$} (y);
		
		
		\draw[dashed, red, line width=0.5mm] (s1)--++(140:1cm);
		\draw[dashed, red, line width=0.5mm] (s1)--++(40:1cm); 
		\draw[dashed, red, line width=0.5mm] (x)--++(200:1cm); 
		\draw[dashed, red, line width=0.5mm] (y)--++(340:1cm);
		\draw[blue] (-0.6, -2.9) node {$2$};  
		\draw[blue] (0.6, -2.9) node {$\Delta$};  
		\draw[blue] (-1.5, -7.4) node {$1$}; 
		\draw[blue] (1.5, -7.4) node {$1$};
		
		\draw[blue] (0, -8.5) node {$L_1$}; 
		
		\begin{scope}[shift={(6,0)}]

		{\tikzstyle{every node}=[draw ,circle,fill=white, minimum size=0.5cm,
			inner sep=0pt]
			\draw[blue,thick](0,-3) node (s1)  {$s_1$};
			\draw [blue,thick](0, -5) node (u)  {$u$};
			\draw [blue,thick](-1, -7) node (x)  {$x$};
			\draw [blue,thick](1, -7) node (y)  {$y$};
		}
		\path[draw,thick,black!60!green]
		
		(s1) edge node[name=la,pos=0.4] {\color{blue}\,\,\, $1$} (u)
		(u) edge node[name=la,pos=0.4] {\color{blue}$\tau$\quad\quad} (x)
		(u) edge node[name=la,pos=0.4] {\color{blue}\,\,\,\,\,\,\, $\Delta$} (y);
		
		
		\draw[dashed, red, line width=0.5mm] (s1)--++(140:1cm);
		\draw[dashed, red, line width=0.5mm] (s1)--++(40:1cm); 
		\draw[dashed, red, line width=0.5mm] (x)--++(200:1cm); 
		\draw[dashed, red, line width=0.5mm] (y)--++(340:1cm);
		\draw[blue] (-0.6, -2.9) node {$2$};  
		\draw[blue] (0.6, -2.9) node {$\Delta$};  
		\draw[blue] (-1.5, -7.4) node {$\delta$}; 
		\draw[blue] (1.5, -7.4) node {$\delta$};
		\draw[blue] (0, -8.5) node {$L_2$}; 
		\end{scope}
		\end{tikzpicture}
	\end{center}
	\caption{Coloring of $S(u; s_1, x, y)$}
	\label{L1}
\end{figure}

\begin{CLA}\label{u-in-p-y}
	$u\in P_y(1,\delta)$.
\end{CLA}
\proof[Proof of Claim~\ref{u-in-p-y}]
Assume to the contrary that $u\not\in P_y(1,\delta)$. 
This implies that  $\delta$ is a $\Delta$-inducing color (since doing a $(1,\delta)$-swap at $y$ gives   $P_{s_1}(\Delta,\delta)=s_1uy$,
 implying that $s_1$ is $(\Delta,\delta)$-unlined with vertices in $F$). 
 If $\varphi(ux)=\tau$ is $\Delta$-inducing, then we perform $(1,2)-(2,\tau)-(\tau,1)$-swaps 
at both $x$ and $y$ based on the coloring of $L_1$ in Figure~\ref{L1}. 
Now, we must have that $u\in P_x(1,\delta)$ or $u\in P_y(1,\delta)$ since $\delta$
is either $2$-inducing or $\Delta$-inducing.  Let $\varphi'$ be obtained from the 
current coloring by 
performing a $(1,\delta)$-swap at both $x$ and $y$. 
Then both $K_1=(r,rs_1,s_1,s_u,u,ux,x)$ and $K_2=(r,rs_1,s_1,s_u,u,uy,y)$
are Kierstead paths with respect to $rs_1$ and $\varphi'$. 
Since $d_G(s_1)=\Delta-1$, applying Lemma~\ref{Lemma:kierstead path1}, 
$x$ and $s_1$ are $(\delta, \Delta)$-linked and $y$ and $s_1$ are $(\delta, 2)$-linked. 
However, by Lemma~\ref{thm:vizing-fan2}, 
$s_1$ and $\pbar^{-1}_F(\delta)$  are either $(\delta,2)$ or $(\delta, \Delta)$-linked, showing a contradiction. 

Thus we assume  that $\varphi(ux)=\tau$ and $\tau$ is $2$-inducing. 
Based on the coloring of $S(u;s_1,x,y)$ as given in Figure~\ref{L1} $L_1$, 
we perform $(1,\tau)-(\tau,\delta)$-swaps at both $x$ and $y$.
Let the current coloring be $\varphi'$. Note that either $\varphi'(s_1u)=\delta$ or 
$\varphi'(s_1u)=\tau$. 
If $\varphi'(s_1u)=\delta$, then  doing  a $(\delta,1)$-swap at both $x$
and $y$ gives  $P_{s_1}(\Delta,1)=s_1uy$, which gives  a contradiction to 
Lemma~\ref{thm:vizing-fan1}~\eqref{thm:vizing-fan1b} that $s_1$ and $r$
are $(\Delta,1)$-linked. Thus $\varphi'(s_1u)=\tau$. 
We first do a $(\delta,1)$-swap at both $x$ and $y$. 
Then since $\tau$ is 2-inducing, $u\in P_y(1,\tau)$ (since otherwise, doing a $(1,\tau)$-swap at $y$ implies that $P_{s_1}(\tau, \Delta)=s_1uy$, showing a contradiction to  Lemma~\ref{thm:vizing-fan2} \eqref{thm:vizing-fan2-a}).
Thus we do a $(1,\tau)$-swap at both $x$ and $y$. 
Note that $\delta$ is still $\Delta$-inducing and $\tau$ 
is 2-inducing.  Thus  $\pbar^{-1}_F(\delta)$ and $\pbar^{-1}_F(\tau)$
are $(\delta,\tau)$-linked by Lemma~\ref{thm:vizing-fan2} \eqref{thm:vizing-fan2-a}. 
Let $\varphi'$ be obtained from the current coloring by 
doing  a $(\delta,\tau)$-swap at $y$. 
Then $K=(r,rs_1,s_1,s_1u, u, uy, y)$ is a Kierstead path with respect to 
$rs_1$ and $\varphi'$.  Since $d_G(s_1)=\Delta-1$, applying Lemma~\ref{Lemma:kierstead path1}, 
 $y$ and $s_1$ are $(\delta,2)$-linked. 
Since $\delta$ is still $\Delta$-inducing and $\tau$ 
is 2-inducing, we achieve a contradiction to the fact that $s_1$
and $\pbar^{-1}_F(\delta)$ are $(\delta, 2)$-linked by Lemma~\ref{thm:vizing-fan2} \eqref{thm:vizing-fan2-a}. 
Therefore it must be the case  $u\in P_y(1,\delta)$.
\qed 

Since $u\in P_y(1,\delta)$, we perform a $(1,\delta)$-swap at both $x$ and $y$
gives $L_2$ in Figure~\ref{L1}.  It deduces that $\delta$ must be a 2-inducing color, as $y$
and $s_1$ are $(\delta,2)$-linked. 
Recall that $F$ still has two sequences.  Let $\gamma$ be a $\Delta$-inducing color 
of $F$. Since $\pbar^{-1}_F(\delta)$ and $\pbar^{-1}_F(\gamma)$
are $(\delta,\gamma)$-linked by Lemma~\ref{thm:vizing-fan2}\eqref{thm:vizing-fan2-a}, 
we do a $(\delta,\gamma)$-swap at $y$. 
This implies that $s_1$ and $y$ are $(\gamma,2)$-linked, showing a 
contradiction to the fact that $s_1$ and $\pbar^{-1}_F(\gamma)$ are $(\gamma,2)$-linked. 

\medskip 
\emph{\noindent \setword{Subcase 2.2.2}{Case 2.2}: $\varphi(uy)=\lambda$ is the last 2-inducing color of $F$.}
\medskip 

If $\varphi(ux)=\tau$ is $2$-inducing, then $\tau \prec \lambda$. 
This gives back to the previous case. 
If $\varphi(ux)=\tau$ is $\Delta$-inducing and $\tau$ 
is not the last $\Delta$-inducing color, then 
 by doing $(\Delta,1)-(1,2)$-swaps at $x$ and $y$, a similar 
 proof follows as in the previous case by exchanging the role of 2 and $\Delta$. 
 Thus $\tau$ is the last $\Delta$-inducing color of $F$. 

Let $C_u$ be the cycle in $G_\Delta$ that contains $u$.  By Theorem~\ref{Thm:vizing-fan2b} \eqref{common2},  for every vertex on $C_u$, its $(\Delta-1)$-neighborhood is $N_{\Delta-1}(u)$. 
As $|V(C_u)|\ge 3$, there exist $u^*, u'\in V(C_u)\setminus\{u\}$
such that one of $\varphi(u^*y)$ and $ \varphi(u'y)$ is neither $\tau$ nor $\lambda$. 
Assume that $\varphi(u^*y)\not\in \{\tau, \lambda\}$. 
Now let $u^*$ play the role of $u$, we reduce the problem to the previous case.

%



This finishes the proof of Theorem~\ref{Thm:nonadj_Delta_vertex}. 
\end{proof}

\end{document}